\documentclass[a4paper,10pt,leqno]{amsart}
\title{$L^2$-torsion of automorphisms}
\author{Sam Hughes}
        \address{Mathematicians Institut der Universit\"at Bonn\\
                Endenicher Allee 60\\
                53115 Bonn, Germany}
         \email{sam.hughes.maths@gmail.com; hughes@math.uni-bonn.de}
          \urladdr{https://samhughesmaths.github.io}
\author{Wolfgang L\"uck}
        \address{Mathematicians Institut der Universit\"at Bonn\\
                Endenicher Allee 60\\
                53115 Bonn, Germany}
         \email{wolfgang.lueck@him.uni-bonn.de}
         \urladdr{http://www.him.uni-bonn.de/lueck}
         
         \date{March, 2026}
       \makeatletter
     \@namedef{subjclassname@2020}{%
  \textup{2020} Mathematics Subject Classification}
\makeatother

\usepackage[pagebackref=true]{hyperref}
\usepackage{color}
\usepackage{pdfsync}
\usepackage{calc}
\usepackage{enumerate,amssymb}
\usepackage{graphicx}
\usepackage{datetime2}
\usepackage[arrow,curve,matrix,tips,2cell]{xy}
\usepackage{thmtools}
\usepackage{thm-restate}

\renewcommand*{\backref}[1]{}
\renewcommand*{\backrefalt}[4]
{\ifcase #1
        No citation in the text.
    \or
        Cited on Page #2.
    \else
        Cited on Pages #2.
    \fi
}

\makeatletter
\def\@tocline#1#2#3#4#5#6#7{\relax
  \ifnum #1>\c@tocdepth % then omit
  \else
    \par \addpenalty\@secpenalty\addvspace{#2}%
    \begingroup \hyphenpenalty\@M
    \@ifempty{#4}{%
      \@tempdima\csname r@tocindent\number#1\endcsname\relax
    }{%
      \@tempdima#4\relax
    }%
    \parindent\z@ \leftskip#3\relax \advance\leftskip\@tempdima\relax
    \rightskip\@pnumwidth plus4em \parfillskip-\@pnumwidth
    #5\leavevmode\hskip-\@tempdima
      \ifcase #1
       \or\or \hskip 1em \or \hskip 2em \else \hskip 3em \fi%
      #6\nobreak\relax
    \hfill\hbox to\@pnumwidth{\@tocpagenum{#7}}\par% <---- \dotfill -> \hfill
    \nobreak
    \endgroup
  \fi}
\makeatother

\SelectTips{eu}{10}\UseTips\UseAllTwocells
\DeclareMathAlphabet{\matheurm}{U}{eur}{m}{n}
\newcommand{\eub}[1]{\underline{E}#1}

\DeclareMathOperator{\colim}{colim}

\DeclareMathOperator{\cyl}{cyl}

\DeclareMathOperator{\GL}{GL}

\DeclareMathOperator{\id}{id}

\DeclareMathOperator{\Isom}{Isom}

\DeclareMathOperator{\pr}{pr}
\DeclareMathOperator{\res}{res}
\DeclareMathOperator{\rk}{rk}

\DeclareMathOperator{\tors}{tors}

\DeclareMathOperator{\vol}{vol}
\DeclareMathOperator{\Wh}{Wh}

\newcommand{\FIN}{\mathcal{FIN}}

  \newcommand{\IC}{\mathbb{C}}
  
  \newcommand{\IE}{\mathbb{E}}

  \newcommand{\IN}{\mathbb{N}}

  \newcommand{\IQ}{\mathbb{Q}}
  \newcommand{\IR}{\mathbb{R}}

  \newcommand{\IZ}{\mathbb{Z}}

  \newcommand{\calf}{\mathcal{F}}

  \newcommand{\calh}{\mathcal{H}}

  \newcommand{\calk}{\mathcal{K}}

  \newcommand{\caln}{\mathcal{N}}

  \newcommand{\calp}{\mathcal{P}}

  \newcommand{\calt}{\mathcal{T}}

  \newcommand{\pt}{\{\bullet\}}

  \newcommand{\EGF}[2]{E_{#2}(#1)}

\newcommand{\calfin}{{\mathcal F}{\mathcal I}{\mathcal N}}

     \newcounter{commentcounter}

\DeclareMathOperator{\JSJ}{\mathrm{JSJ}}
\DeclareMathOperator{\Flex}{\mathrm{Flex}}
\DeclareMathOperator{\Stab}{\mathrm{Stab}}

\DeclareMathOperator{\fr}{\mathrm{fr}}
\DeclareMathOperator{\Mod}{\mathrm{Mod}}
\DeclareMathOperator{\LM}{\mathrm{LM}}
\DeclareMathOperator{\Aut}{\mathrm{Aut}}
\DeclareMathOperator{\Out}{\mathrm{Out}}

     \theoremstyle{plain} \newtheorem{theorem}{Theorem}[section]
      \newtheorem{lemma}[theorem]{Lemma}
     \newtheorem{corollary}[theorem]{Corollary}
     \newtheorem{prop}[theorem]{Proposition}
     \newtheorem{proposition}[theorem]{Proposition}
     \newtheorem{conjecture}[theorem]{Conjecture}
     
      \newtheorem*{theorem*}{Theorem}
     \newtheorem*{theoremA*}{Theorem A} \newtheorem*{theoremB*}{Theorem B}

     \theoremstyle{definition} \newtheorem{definition}[theorem]{Definition}
      \newtheorem{example}[theorem]{Example}
      
      \newtheorem{remark}[theorem]{Remark}

     \newtheorem*{definition*}{Definition}

     \theoremstyle{remark}

     \makeatletter\let\c@equation=\c@theorem\makeatother

      \newcommand{\CAT}{\operatorname{CAT}}

     \newcommand{\version}[1] %marks the date of last editing and compilation
     {\begin{center} last edited on #1\\
         last compiled on \today \ at \DTMcurrenttime.\\
         name of tex-file: \jobname
       \end{center}}

\begin{document}

     \begin{abstract}
       We develop the theory of $L^2$-torsion of an automorphism of a group and compute it for every automorphism of a group which is hyperbolic and one-ended relative to a finite collection of virtually polycyclic groups.  We also prove a combination formula for the $L^2$-torsion of a group in terms of the $L^2$-torsion of its stabilisers of a sufficiently nice action on a contractible space. We apply it to compute the $L^2$-torsion of a selection of $\CAT(0)$
       lattices, of many relatively hyperbolic groups and their automorphisms, of higher
       dimensional graph manifolds, and of handlebody groups. 
     \end{abstract}

     \maketitle

%%%%%%%%%%%%%%%%%%%%%%%%%%%%%%%%%%%%%%%%%%%%%%%%%%%%%%%%%%%%%%%%%%%%
%%%%%%%%%%%%%%%%%%%%%%%%%% Introduction %%%%%%%%%%%%%%%%%%%%%%%%%%%%%%%%
%%%%%%%%%%%%%%%%%%%%%%%%%%%%%%%%%%%%%%%%%%%%%%%%%%%%%%%%%%%%%%%%%%%%

  % \commentw{I hope  that the current version will allow that in Sam's
  %   part we can point at a statement in these notes. Probably  most relevant for Sam's part will be
  %   Theorem~\ref{the:application_to_automorphism} and Example~\ref{exa:graph_of_groups_group_autos}.}

  % \commentw{I have replaced the notions like ``admissible'' by $\det$-finite, see
  %   Definition~\ref{def:det_finite_G-space} and the condition (DFJ), see
  %   Definition~\ref{def:Condition_(DFJ}}

  % \commentw{Take a look at~\cite{Martinez-Pedroza-Przytzycki(2019)}.}

  % \commentw{Add a remark about twisted $L^2$-torsion.}

  \typeout{------------------- Introduction -----------------}
  \section{Introduction}\label{sec:introduction}

The paper deals with $L^2$-torsion and computations of it,  in particular for groups and group automorphisms.
Before we are describing our results, we give a brief survey about the relevance of $L^2$-torsion.

The $L^2$-torsion $\rho^{(2)}$ is an invariant of groups and spaces which is defined for a
large class of groups and spaces with vanishing $L^2$-homology and can be defined
analytically and topologically.  When defined, the invariant is a real number which
behaves similarly to an Euler characteristic in the sense that it is multiplicative
through finite covers.  It seems plausible that it should behave like a hyperbolic volume.
Outside of the world of closed locally symmetric spaces~\cite{Olbrich(2002)} and
$3$-manifolds~\cite{Lueck-Schick(1999),Lueck(2002)}, computing $L^2$-torsion remains a
formidable challenge and there are very few examples, e.g.,~\cite{Clay(2017free),
  Wegner(2009)}.  For the definition of $L^2$-torsion and a discussion about the
Determinant Conjecture, which ensures the $L^2$-torsion is defined, we refer the reader to
Section~\ref{sec:Basics_about_L_upper_2-torsion}.  For a comprehensive introduction to
$L^2$-invariants the reader is referred to~\cite{Lueck(2002)}.

One should think of $L^2$-torsion as a generalization of the notion of volume. Namely, if
$M$ is a compact manifold of odd dimension $n$ whose interior is a complete hyperbolic
manifold of finite volume, then its volume is up to a dimension constant $C_n \not = 0$
proportional to the $L^2$-torsion, see L\"uck--Schick~\cite{Lueck-Schick(1999)}.  There
are other notions generalizing the notion of the volume for hyperbolic manifolds of finite
volume such as its minimal volume entropy $\mathcal E_{\mathrm min}(M)$ and its simplicial
volume $||M||$.  Work of Gromov~\cite{Gromov(1982)}, Pieroni \cite{Pieroni2019},
Soma~\cite{Soma(1981)}, and Thurston~\cite{Thurston(1978)} shows that $\mathrm{Vol}(M)$,
$\mathcal E_{\mathrm min}(M)^3$, and $||M||$ are proportional for hyperbolic orientable
closed $3$-manifolds and that $\mathrm{Vol}(M)$ and $||M||$ are proportional for
hyperbolic orientable closed manifolds up to a dimension constant.  Moreover, for a
closed, orientable $n$–manifold Gromov \cite[page~37]{Gromov(1982)} showed that
$\mathcal E_{\mathrm{min}}(M)^n\geq c_n||M||$ where $c_n$ only depends on $n = \dim M$.
We remark that there are examples of closed aspherical manifolds $M$, where
$\rho^{(2)}(\widetilde{M}) = 0$ and $||M|| \not = 0$ holds, see~\cite[Example~4.18 on
page~498]{Lueck(2002)}. There is the question whether for an aspherical closed manifold
$M$ with $||M| = 0$ the universal covering is $\det$-$L^2$-acyclic and satisfies
$\rho^{(2)}(\widetilde{M}) = 0$, see~\cite[Question~14.39 on page 488]{Lueck(2002)}.
Moreover, there are examples where $\rho^{(2)}(G)=0$, but
$\mathcal E_{\mathrm min}(G)\neq 0$, compare~\cite[Theorem~1.1]{Bregman-Clay(2021)} with
\cite{Clay(2017free)}.

The following conjecture is taken from~\cite[Conjecture~1.12~(2)]{Lueck(2013l2approxfib)}.
  For locally symmetric spaces it reduces to the conjecture of Bergeron and
  Venkatesh~\cite[Conjecture~1.3]{Bergeron-Venkatesh(2013)}.

  \begin{conjecture}[Homological torsion growth and $L^2$-torsion]%
  \label{con:Homological_torsion_growth_and_L2-torsion}
  Let $M$ be an aspherical closed manifold. Consider a descending chain of subgroups
  $\pi_1(M) = G_0 \supseteq G_1 \supseteq G_2 \supseteq \cdots $
  such that $G_i$ is normal in $G$, the index $[G:G_i]$ is finite, and $\bigcap_{i \ge 0} G_i = \{1\}$.
  Let $p \colon \overline{M} \to M$ be the universal covering. Put $M[i] := G_i\backslash \overline{M}$.
  We obtain a $[G:G_i]$-sheeted covering $p[i] \colon M[i] \to M$. 

     Then we get for  any natural number $n$ with $2n +1 \not= \dim(M)$
      \[
      \lim_{i \to \infty} \;\frac{\ln\big(\bigl|\tors\bigl(H_n(M[i];\IZ)\bigr)\bigr|\bigr)}{[G:G_i]} = 0.
      \]
      If the dimension $\dim(M) = 2m+1$ is odd, then $\widetilde{M}$ is $\det$-$L^2$-acyclic and we get
      \[
      \lim_{i \to \infty} \;\frac{\ln\big(\bigl|\tors\bigl(H_m(M[i];\IZ)\bigr)\bigr|\bigr)}{[G:G_i]} = (-1)^m \cdot
      \rho^{(2)}(\widetilde{M}).
      \]
    \end{conjecture}

 Considerations concerning the Singer Conjecture due to Avramidi-Okun-Schreve before Theorem~4 appearing
 in~\cite{Avramidi-Okun-Schreve(2021)} lead to the following modification whose conclusion is weaker and
 appears in~\cite[Section~6.7]{Kirstein-Kremer-Lueck(2025)} for $X$ an aspherical closed manifold of odd dimension and
 in~\cite[Conjecture~8.9]{Lueck(2016_l2approx)} in general.

\begin{conjecture}[Modified Homological torsion growth and $L^2$-torsion]%
      \label{con:Modified_Homological_torsion_growth_and_L2-torsion}
      Let $X$ be a connected finite $CW$-complex which is $\det$-$L^2$-acyclic.
      Put
      \[\rho^{\IZ}(X[i])= \sum_{n = 0}^{2m+1} (-1)^n\cdot
        \frac{\ln(|\tors(H_n(X[i];\IZ))|)}{[G:G_i]}.
       \]

      Then the limit
      $\lim_{i \to \infty} \rho^{\IZ}(X[i])$
      exists and is given by
      \[
       \lim_{i \to \infty} \rho^{\IZ}(X[i]) =
       \rho^{(2)}(\widetilde{X}).
      \]
    \end{conjecture}

    If $M$ is a closed hyperbolic $3$-manifold, the conjecture of Bergeron and Venkatesh,
    Conjecture~\ref{con:Homological_torsion_growth_and_L2-torsion}, and
    Conjecture~\ref{con:Modified_Homological_torsion_growth_and_L2-torsion} are equivalent
    and reduce to the assertion
    \[
    \lim_{i \to \infty} \;\frac{\ln\big(\bigl|\tors\bigl(H_1(M[i];\IZ)\bigr)\bigr|\bigr)}{[G:G_i]} = \frac{1}{6\pi} \cdot \mathrm{Vol}(M).
    \]
    Note that there are groups $G$ where $\mathcal E_{\mathrm{min}}(BG)\neq 0$ but the
    torsion homology growth vanishes in every degree
    \cite[Corollary~C]{Andrew-Hughes-Kudlinska(2024)}.
    
    Our first result is a `sum' or `combination' formula for the $L^2$-torsion of a group
    in terms of the $L^2$-torsion of its stabilisers of a sufficiently nice action on a
    contractible space.  One should compare this to the cheap $\alpha$-rebuilding property
    of Abert, Bergeron, Fraczyk, and Gaboriau~\cite{Abert-Bergeron-Fraczyk-Gaboriau(2025)}
    and its algebraic analogue due to Li, L\"oh, Moraschini, Sauer, and
    Uschold~\cite{Loeh-Li-Moraschini-Sauer-Uschold(2024)}. We note that an exposition of a
    version of the following theorem also appeared
    in~\cite[\S4.2.4]{Loeh-Li-Moraschini-Sauer-Uschold(2024)}.

    \begin{restatable*}{theorem}{thmCombination}\label{thmCombination}
      Let $G$ be a group acting cocompactly on a contractible $CW$-complex $X$ such that
      the fixed point sets of finite subgroups of $G$ are contractible.  Suppose that each
      cell stabiliser $H_\sigma$ of the action of $G$ is $L^2$-acyclic and admits a finite
      model for $\underline{E}H_\sigma$.  If $G$ satisfies the Determinant Conjecture,
      then
      \[
        \rho^{(2)}(G) = \sum_{n \ge 0} \sum_{i \in I_n} (-1)^n \cdot
        \rho^{(2)}(H^n_{i_n}).
      \]
    \end{restatable*}

    Here $\underline EG$ is the \emph{classifying space for proper actions} of $G$.  That
    is, the classifying space for the family of finite subgroups.  In the case where $G$
    is torsion-free we are simply asking that the stabilisers admit a finite classifying
    space.  The additional complexity in the torsion case is because traditionally
    $L^2$-torsion is not defined for groups which are not virtually torsion-free.  We
    circumvent this issue by developing the theory using $\underline{E}G$ in
    Section~\ref{sec:Basics_about_L_upper_2-torsion}.

    Using Theorem~\ref{thmCombination} we make a number of computations in
    Section~\ref{sub:Computations}.  For example we give a vanishing criterion for a
    $\CAT(0)$ lattice in a product of locally compact groups
    (Proposition~\ref{vanishing_CAT0}), prove vanishing for Leary--Minasyan groups
    (Example~\ref{ex:LearyMinasyan}) and we compute the $L^2$-torsion of a generalised
    graph manifold (Theorem~\ref{thm:graphMan}).  We now highlight one more computation
    here in the introduction. Handlebody groups are mapping class groups of
    $3$-dimensional handlebodies and are an important class of groups arising in
    low-dimensional topology, see~\cite{Andrew-Hensel-Hughes-Wade(2025)} for more
    information.  It is known~\cite[Theorem~6.1]{Andrew-Hensel-Hughes-Wade(2025)} that the
    $L^2$-Betti numbers and homology torsion growth of the handlebody groups vanishes.
    Here we show that $L^2$-torsion vanishes as well, verifying
    Conjecture~\ref{con:Homological_torsion_growth_and_L2-torsion} for the groups and
    solving~\cite[Problem~28]{Andrew-Hensel-Hughes-Wade(2025)}.

\begin{restatable*}{theorem}{thmHandlebody}
    Let $g\geq 2$ and let $V_g$ denote the genus $g$ handlebody.
  Then, \[\rho^{(2)}(\Mod(V_g))=0.\]
\end{restatable*}

We also prove a theorem about the vanishing of $L^2$-torsion for a polynomially growing
automorphisms of many families of groups.  An automorphism $\Phi$ is \emph{polynomially
  growing} of \emph{degree} at most $d$ if for each $g\in G$ there is a constant $C$ such
that $|\Phi^n(g)|<Cn^d+C$ for all $n\in\IN$.  The following theorem
answers~\cite[Question~1.2]{Andrew-Guerch-Hughes-Kudlinska(2023)} and when combined
with~\cite[Theorem~A]{Andrew-Guerch-Hughes-Kudlinska(2023)} resolves
Conjecture~\ref{con:Homological_torsion_growth_and_L2-torsion} for the relevant groups.

\begin{restatable*}{theorem}{thmAnthology}\label{thm:poly_anthology_L2tors}
  Let $\Gamma$ be a group isomorphic to one of
  \begin{itemize}
  \item $G\rtimes_\Phi \IZ$ with $G$ residually finite and hyperbolic;
  \item $G\rtimes_\Phi\IZ$ with $G$ residually finite and hyperbolic relative to a finite
    collection of virtually polycyclic groups;
  \item $A_L\rtimes_\Phi \IZ$ where $A_L$ is a right-angled Artin group and
    $\Phi \in \mathrm{Aut}(A_L)$ is untwisted; or
  \item $W_L\rtimes_\Phi \IZ$ where $W_L$ is a right-angled Coxeter group.
  \end{itemize}
  If $\Phi$ is polynomially growing, then $\rho^{(2)}(\Gamma)=0$.
\end{restatable*}

Whilst the previous theorem is specifically about groups we actually work much more
generally and study the \emph{$L^2$-torsion $\rho^{(2)}(\Phi)$ of an automorphism}
$\Phi\colon G\to G$.  This generalisation allows us to side-step assumptions about the
Determinant Conjecture and only assume it for certain subgroups of $G$.  Moreover, when
the Determinant Conjecture holds for $\Gamma=G\rtimes_\Phi \IZ$ we have
$\rho^{(2)}(\Phi)=\rho^{(2)}(\Gamma)$.  We take up this task in
Sections~\ref{sec:L2-torsion_of_a_selfhomotopy_equivalence}
and~\ref{sec:L2-torsion_of_an_automorphism_of_a_det-finite_group}.  Our main application
is that we compute the $L^2$-torsion of any automorphism of a one-ended hyperbolic group
in terms of its canonical JSJ decomposition.  For background on JSJ decompositions
see~\cite{Guirardel-Levitt(2017JSJ)}.  More generally we prove:

\begin{restatable*}{theorem}{thmRelHypAutos}\label{thmRelHypAutos}
  Let $G$ be a group hyperbolic and one-ended relative to a finite collection $\calp$ of
  virtually polycyclic groups, let $\Phi \in \calk(\calt_G)$, and let
  $\Gamma=G\rtimes_\Phi\IZ$.  Then
  \[\rho^{(2)}(\Phi)= \sum_{v\in\Flex(G)} \rho^{(2)}(G_v\rtimes_{\Phi|_{G_v}}\IZ),\]
  where $\Flex(G)$ is the set of flexible vertices in a JSJ decomposition of $G$.
\end{restatable*}

In the previous theorem, $\calk(\calt_G)$ is a certain finite index subgroup of
$\Aut(G,\calp)$, see
Section~\ref{sec:Ltwo_rel_hyp}.

As an application of the previous theorem we obtain vanishing for all polynomially growing
automorphisms of all one-ended groups hyperbolic relative to a finite collection of
virtually polycyclic groups --- see Theorem~\ref{thm:poly_L2tors_rel_hyp_one_ended}.  We
refer the interested reader to Section~\ref{sec:poly} for further results.  Motivated by
our results we raise the following conjecture:

\begin{conjecture}[Vanishing of the $L^2$-torsion of an automorphism of subexponential
  growth]%
  \label{con:Vanishing_of_the_L2-torsion_of_an_automorphism_of_subexponential_growth}
  Let $\Phi \colon G \to G$ be an automorphism of a $\det$-finite group $G$ which has
  subexponential growth. Then $\rho^{(2)}(\Phi) = 0$.
\end{conjecture}

\begin{remark}
  We remark that all of our results and formulas here can be adapted to the setting of
  twisted $L^2$-torsion as developed in~\cite{Lueck(2018)}.
\end{remark}

  %-----------------------------------------------------------------------------

\subsection{Acknowledgments}\label{subsec:Acknowledgements}

The paper is funded by the Deutsche Forschungsgemeinschaft (DFG, German Research
Foundation) under Germany's Excellence Strategy \--- GZ 2047/1, Projekt-ID 390685813,
Hausdorff Center for Mathematics at Bonn. The first author  was supported by a Humboldt
Research Fellowship at Universit\"at Bonn.  We thank the referee for a careful reading of the paper.

The paper is organized as follows:
\tableofcontents

%%%%%%%%%%%%%%%%%%%%%%%%%%%%%%%%%%%%%%%%%%%%%%%%%%%%%%%%%%%%%%%%%%%%%
%%%%%%%%%%%%%%%%%%%%%%%%%%%%%% Section 2 %%%%%%%%%%%%%%%%%%%%%%%%%%%%%%%%
%%%%%%%%%%%%%%%%%%%%%%%%%%%%%%%%%%%%%%%%%%%%%%%%%%%%%%%%%%%%%%%%%%%%%

\typeout{---------- Section 2:  Basics about $L^2$-torsion   ---------------}

\section{Basics about \texorpdfstring{$L^2$}{L2}-torsion}%
\label{sec:Basics_about_L_upper_2-torsion}

There are several notions of volume for a finite volume hyperbolic $3$-manifold $M$.  Its
hyperbolic volume $\mathrm{Vol}(M)$, its $L^2$-torsion $\rho^{(2)}(M)$, and its minimal
volume entropy $\mathcal E_{\mathrm min}(M)$.  Work of L\"uck and Schick shows that
$\mathrm{Vol}(M)$ and $\rho^{(2)}(M)$ are proportional~\cite{Lueck-Schick(1999)}.  Work of
Soma~\cite{Soma(1981)}, Gromov~\cite{Gromov(1982)}, and Thurston~\cite{Thurston(1978)}
shows that $\mathrm{Vol}(M)$ and $\mathcal E_{\mathrm min}(M)^3$ are proportional.

We collect some basic facts about $L^2$-torsion.  Most of it is described
in~\cite[Section~3.4]{Lueck(2002)} provided we consider finite free
$G$-$CW$-complexes. In this section we want to explain that all of this carries over to
proper finite $G$-$CW$-complexes. Recall that a $G$-$CW$-complex is proper if and only if
all its isotropy groups are finite, and is finite if and only if it is cocompact. The
motivation is that thus the notion of $L^2$-torsion makes sense for groups which are not
torsionfree but have a finite model for the classifying space $\eub{G} = \EGF{G}{\calfin}$
for proper $G$-actions, for instance for all hyperbolic groups, provided that
one such model is $\det$-$L^2$-acyclic.

Given a proper finite $G$-$CW$-complex, we denote by $C_*^c(X)$ its cellular $\IZ G$-chain
complex.  Suppose that we have chosen cellular $G$-pushouts for every $n \in \IZ^{\ge 0}$
\begin{equation}
\xymatrix{\coprod_{i \in I_n} G/H_i \times S^{n-1} \ar[r] \ar[d]
    &
    X_{n-1} \ar[d]
    \\
    \coprod_{i \in I_n} G/H_i \times D^n \ar[r]
    &
    X_n.
  }
  \label{cellular_G-pushouts_special}
\end{equation}
Note that then $I_n$ can be identified with the set of open $n$-cells of the finite
$CW$-complex $X/G$.  They induce preferred $\IZ[G]$-isomorphisms
  \begin{equation}
    \varphi_n \colon \bigoplus_{i \in I_n} \IZ[G/H_i] \xrightarrow{\cong} C_n^c(X),
    \label{isos_varphi_n}
  \end{equation}
   and hence preferred  $\IC G$-isomorphisms
  \begin{equation}
    \varphi^{(2)}_n \colon \bigoplus_{i \in I_n}  L^2(G) \otimes_{\IZ G} \IZ[G/H_i]  
    \xrightarrow{\cong} \bigoplus_{i \in I_n}  L^2(G) \otimes_{\IZ G} C_n^c(X).
    \label{isos_varphi_n_upper_(2)}
  \end{equation}
  Note that $ L^2(G) \otimes_{\IZ G} \IZ[G/K]$ is for any finite subgroup $K \subseteq G$
  a finitely generated Hilbert $\caln(G)$-module, as it embeds isometrically and
  $G$-linearly into $L^2(G)$, namely, by sending $x \otimes gK $ to
  $\frac{1}{|K|} \cdot \sum_{k \in K} xgk$.  Hence we get from the
  isomorphism~\eqref{isos_varphi_n_upper_(2)} the structure of a finite Hilbert
  $\caln(G)$-chain complex on $L^2(G) \otimes_{\IZ G} C_*^c(G)$.  Note that the choice of
  the $G$-pushouts~\eqref{cellular_G-pushouts_special} is not part of the
  $G$-$CW$-structure, only their existence is required. So we have to show that the
  structure of a finite Hilbert $\caln(G)$-chain complex on
  $L^2(G) \otimes_{\IZ G} C_*^c(G)$ is independent of the choice of the
  $G$-pushouts~\eqref{isos_varphi_n_upper_(2)}. If we make a different choices, then this
  changes the isomorphism $\varphi_n$ of~\eqref{isos_varphi_n} by an automorphism of the
  shape
  \[
    \nu := \bigoplus_{i \in I_n} \epsilon_i \cdot \nu_i \colon \bigoplus_{i \in I_n} \IZ[G/H_i]
    \xrightarrow{\cong} \bigoplus_{i \in I_n} \IZ[G/H'_i]
\]
where $\epsilon_i \in \{\pm 1\}$ and $\nu_i$ is induced by a bijective $G$-map
$G/H_i \xrightarrow{\cong} G/H_i'$.  Since the map
  \[\id_{L^2(G)} \otimes \nu \colon \bigoplus_{i \in I_n}  L^2(G) \otimes_{\IZ G} \IZ[G/H_i]
    \xrightarrow{\cong} \bigoplus_{i \in I_n}  L^2(G) \otimes_{\IZ G} \IZ[G/H'_i]  
  \]
  is an isometric $G$-linear automorphism, the structure of a finite Hilbert
  $\caln(G)$-chain complex on $L^2(G) \otimes_{\IZ G} C_*^c(G)$ is independent of the
  choice of the $G$-pushouts~\eqref{cellular_G-pushouts_special}.  Hence the finite
  $\caln(G)$-Hilbert chain complex $L^2(G) \otimes_{\IZ G} C_*^2(X)$ depends only the
  $G$-$CW$-structure of the finite proper $G$-$CW$-complex $X$.

  Now for a proper finite $G$-$CW$-complex $X$, the notions of $L^2$-Betti numbers
  $b^{(2)}_n(X;\caln(G))$, of determinant class, and of being $\det$-$L^2$-acyclic are
  defined, and we also have the notion of $L^2$-torsion $\rho^{(2)}(X;\caln(G)) \in \IR$,
  provided that $X$ is $\det$-$L^2$ acyclic.  (It is not necessary to understand the precise definition of these notions
  to read this paper.)
  Note that the class of groups which
  satisfies the Determinant Conjecture, see~\cite[Conjecture~13.2 on page~454]{Lueck(2002)},
  is quite large.  It includes all sofic groups.  If
  $G$ satisfies the Determinant Conjecture, the condition $\det$-$L^2$-acyclic on the
  proper finite $G$-$CW$-complex $X$ reduces to the condition that the $L^2$-Betti number
  $b_n^{(2)}(X;\caln(G))$ vanishes for all $n \in \IZ^{\ge 0}$.

  A \emph{Farrell-Jones group} $G$ is a group $G$ satisfying the \emph{Full Farrell-Jones
    Conjecture}.  Explanations about the Full Farrell-Jones Conjecture and informations
  about the class of Farrell-Jones groups, which is  rather large and contains for
  instance all hyperbolic groups, can be found in~\cite{Lueck(2025book)}.

    For a group $L$ define the  group homomorphism
    \begin{equation}
      {\det}^{(2)}_L \colon \Wh(L) \to \IR^{> 0}
      \label{det_L}
    \end{equation}
    by sending the class $[A]$ of an invertible $(m,m)$-matrix  $A \in \GL_m(\IZ L)$
    to the  Fuglede-Kadison determinant $\det^{(2)}(r_A^{(2)}) \in \IR^{>0}$
    of the automorphism of a finite Hilbert $\caln(G)$-module
     $r^{(2)}_A \colon L^2(L)^m \to  L^2(L)^m$ given by $A$.

    \begin{lemma}\label{:vanishing_of_the_determinant_map}
      Suppose  that the group $L$ satisfies the Determinant Conjecture or is a Farrell-Jones group.
       Then the homomorphism
       \[
        {\det}^{(2)}_L \colon \Wh(L) \to \IR^{> 0}
      \]
      defined in~\eqref{det_L} is trivial.
    \end{lemma}
    \begin{proof} Suppose that $L$ satisfies the Determinant Conjecture.  Given an
      invertible $(m,m)$-matrix $A \in \GL_m(\IZ L)$, we conclude
      $\det^{(2)}(r_A^{(2)})\ge 1$ and $\det^{(2)}(r_{A^{-1}}^{(2)}) \ge 1$ from the
      Determinant Conjecture. Since we have
      \[{\det}^{(2)}(r_A^{(2)}) \cdot {\det}^{(2)}(r_{A^{-1}}^{(2)})
        =
        {\det}^{(2)}(r_A^{(2)} \cdot r_{A^{-1}}^{(2)}) = {\det}^{(2)}(\id_{L^2(G)^m}) = 1.
        \]
        we get $\det^{(2)}(r_A^{(2)}) = 1$.
        
        Suppose that $L$ is a Farrell-Jones group. If $L$ is torsionfree, then $\Wh(L)$
        vanish and the claim is obviously true.  If $L$ is not torsionfree, 
        $\Wh(L)$ is in general non-trivial. Nevertheless, the arguments in the proof
        of~\cite[Theorem~6.7~(2)]{Lueck(2018)} reduce to a proof that
        ${\det}^{(2)}_L \colon \Wh(L) \to \IR$ is trivial, take
        in~\cite[Theorem~6.7~(2)]{Lueck(2018)} $V$ to be the trivial $1$-dimensional
        $L$-representation.
      \end{proof}
      
      Now we collect some properties of $L^2$-torsion. The proof of the following theorem
      can be found in the case, where $G$ is torsionfree, in~\cite[Theorem~3.93 on
      page~161]{Lueck(2002)}.  Here we want to drop the condition torsionfree.

      \begin{definition}[Condition (DFJ)]\label{def:Condition_(DFJ}
    The group $G$ satisfies condition (DFJ) if it satisfies one of the following conditions:
  \begin{enumerate}
  \item The group $G$ satisfies the Determinant Conjecture;
  \item For any finite subgroup $H \subseteq G$ the Weyl group $W_GH$ is a Farrell-Jones  group;
  \item The group $G$ contains a torsionfree group $H$ of finite index such that
    $H$ satisfies the Determinant Conjecture or is a Farrell-Jones Group.
  \end{enumerate}
\end{definition}

If $G$ satisfies condition (DFJ), then every subgroup of $G$ does, since every subgroup of
a group which satisfies the Determinant Conjecture or is a Farrell-Jones group
respectively, satisfies the Determinant Conjecture or is a Farrell-Jones group
respectively, see~\cite[Theorem~3.14~(6) on page~129 and Lemma~13.45~(7) on
page~473]{Lueck(2002)} and~\cite[Theorem~16.5~(iia)]{Lueck(2025book)}.

    \begin{theorem}[Basic properties of $L^2$-torsion]\label{the:Basic_properties_of_L_upper_2-torsion}\

    \begin{enumerate}

    \item\label{the:Basic_properties_of_L_upper_2-torsion:G-homotopy_invariance}
      \emph{Homotopy invariance}\\[1mm]
       Let $X$ and $Y$ be  proper finite $G$-$CW$-complexes which are $G$-homotopy equivalent.
      Suppose  that $X$ is $\det$-$L^2$-acyclic. Then:

      \begin{enumerate}

      \item\label{the:Basic_properties_of_L_upper_2-torsion:G-homotopy_invariance:det_L2-acyclic}
       $Y$ is $\det$-$L^2$-acyclic. 

     \item\label{the:Basic_properties_of_L_upper_2-torsion:G-homotopy_invariance:equality}
       If $X$ and $Y$ are simple $G$-homotopy equivalent or if $G$ satisfies condition (DFJ),
       we get
      \[
       \rho^{(2)}(X; \caln(G)) = \rho^{(2)}(Y; \caln(G));
      \]
    \end{enumerate}
    
    \item\label{the:Basic_properties_of_L_upper_2-torsion:Sum_formula} \emph{Sum formula}\\[1mm]
      Consider the $G$-pushout of proper finite $G$-$CW$-complexes
      \[
        \xymatrix{X_0 \ar[r]^i \ar[d]_{f}
          &
          X_1 \ar[d]^{\overline{f}}
          \\
          X_2 \ar[r]_{\underline{i}}
          &
          X
        }
      \]
      where $i$ is an inclusion of proper $G$-$CW$-complexes, $f$ is a cellular $G$-map of
      $G$-$CW$-complexes, and the $G$-$CW$-structure on $X$ has as $n$-skeleton
      $X_n = \overline{f}((X_2)_n) \cup \underline{i}((X_1)_n)$.  Suppose that $X_0$,
      $X_1$, and $X_2$ are proper finite $G$-$CW$-complexes which are
      $\det$-$L^2$-acyclic.

      Then $X$ is a proper finite $G$-$CW$-complex which is  $\det$-$L^2$-acyclic, and we get
      \[
        \quad \quad \quad \quad \rho^{(2)}(X; \caln(G)) = \rho^{(2)}(X_1; \caln(G))
        + \rho^{(2)}(X_2; \caln(G)) - \rho^{(2)}(X_0; \caln(G));
      \]

      \item\label{the:Basic_properties_of_L_upper_2-torsion:Product_formula} \emph{Product formula}\\[1mm]
        Let $X$ be a proper finite $G$-$CW$-complex which is  $\det$-$L^2$-acyclic.
        Let $Z$ be a finite proper $H$-$CW$-complex.  Denote by $\chi^{(2)}(Z;\caln(H))$ the $L^2$-Euler characteristic of $Z$.

          Then the $G \times H$-space $X \times Z$ is a proper $G \times H$-$CW$-complex which $\det$-$L^2$-acyclic, and  we get 
          \[
           \rho^{(2)}(X \times Z;\caln(G \times H))  = \chi^{(2)}(Z;\caln(H)) \cdot  \rho^{(2)}(X;\caln(G));
         \]

      \item\label{the:Basic_properties_of_L_upper_2-torsion:passage_to_subgroups_of_finite_index}
        \emph{Restriction}\\[1mm]
        Let $H$ be a subgroup of $G$ of finite index $[G:H]$. Let
        $X$ be a proper finite $G$-$CW$-complex.  Let $X|_H$ be the $H$-space obtained
        from $X$ by restriction with $i$.

        Then $X|_H$ is a proper finite $H$-$CW$-complex. It is $\det$-$L^2$-acyclic if and
        only if $X$ is $\det$-$L^2$-acyclic.  If $X$ is $\det$-$L^2$-acyclic, then we get
   \[
     \rho^{(2)}(X|_H;\caln(H)) = [G:H] \cdot \rho^{(2)}(X;\caln(G));
    \]

  \item\label{the:Basic_properties_of_L_upper_2-torsion:induction} \emph{Induction}\\[1mm]
    Let $H$ be a subgroup of $G$. Let $X$ be a proper finite $H$-$CW$-complex.

    Then $G \times_H X$ is a finite proper $G$-$CW$-complex. It is $\det$-$L^2$-acyclic
    if and only if $X$ is $\det$-$L^2$-acyclic.  If $X$ is $\det$-$L^2$-acyclic, then we
    get
    \[
     \rho^{(2)}(G \times_H X;\caln(G)) = \rho^{(2)}(X;\caln(H));
   \]

   \item\label{the:Basic_properties_of_L_upper_2-torsion:finite_quotient} \emph{Finite kernels}\\[1mm]
       Let $1 \to K \to G \xrightarrow{p} Q \to 1$ be an extension of groups with finite $K$.
      Let $X$ be a proper finite $Q$-$CW$-complex. Let $p^*X$ be the $G$-space obtained from $X$ by restriction with $p$.

      Then $p^*X$ is a finite proper $G$-$CW$-complex. If $X$  or $p^*Y$ is $\det$-$L^2$-acyclic, then both 
      $X$  and  $p^*Y$ are $\det$-$L^2$-acyclic and  get
    \[
     \rho^{(2)}(X;\caln(G)) = \frac{\rho^{(2)}(X;\caln(Q))}{|K|};
   \]

  \item\label{the:Basic_properties_of_L_upper_2-torsion:Poincare_duality} \emph{Poincar\'e duality}\\[1mm]
    If $M$ is a
    proper cocompact smooth $G$-manifold without boundary. Then $M$ inherits from a smooth
    triangulation the structure of a proper finite $G$-$CW$-complex. Suppose that $M$ is
    $\det$-$L^2$-acyclic and its dimension is even. Then we get
     \[
     \rho^{(2)}(M;\caln(G)) = 0.
     \]

    \end{enumerate}
  \end{theorem}
  \begin{proof}~\ref{the:Basic_properties_of_L_upper_2-torsion:G-homotopy_invariance} We
    only give a brief sketch of the proof.  One key fact is that one has the notion of
    equivariant Whitehead torsion $\tau^G(f)$ for a $G$-homotopy equivalence
    $f \colon X \to Y$ of proper finite $G$-$CW$-complexes which takes values in
    $\bigoplus_{(H)} \Wh(W_GH)$, where $(H)$ runs through the conjugacy classes of finite
    subgroup $H \subseteq G$ and $W_GH$ is the Weyl group $N_GH/H$,
    see~\cite[Chapters~4  and~12]{Lueck(1989)}.  Then the homomorphism
     \[
       \alpha \colon \bigoplus_{(H)} \frac{1}{|H|} \cdot  \ln \circ {\det}^{(2)}_{W_GH}
       \colon \bigoplus_{(H)} \Wh(W_GH) \to \IR
      \]
      sends $\tau^G(f)$ to $\rho^{(2)}(X;\caln(G)) - \rho^{(2)}(Y;\caln(G))$,
      where ${\det}^{(2)}_{W_GH}$ has been defined in~\eqref{det_L}.
      If $f$ is a simple $G$-homotopy equivalence, then $\tau^G(f)$ vanishes.
      It remains to show  that $\alpha$ is trivial if  $G$ satisfies condition (DFJ).
      
      If $W_GH$ satisfies the Determinant Conjecture or is a Farrell-Jones group for every
      finite subgroup $H \subseteq G$, then the vanishing of $\alpha$ follows from
      Lemma~\ref{:vanishing_of_the_determinant_map}. If $G$ satisfies the Determinant
      Conjecture, then $W_GH$ satisfies the Determinant Conjecture for every finite
      subgroup $H \subseteq G$
      by~\cite[Theorem~3.14~(6) on page~129 and Lemma~13.45~(7) on  page~473]{Lueck(2002)}.
      If $G$ is a Farrell-Jones group and $H \subseteq G$ is a
      finite subgroup, we know that $N_GH$ is a Farrell-Jones group but we do not know in
      general that $W_GH$ is a Farrell-Jones-group.  Now suppose that $K$ is a torsionfree
      subgroup of $G$ of finite index.  If $K$ satisfies the Determinant Conjecture, then
      $G$ satisfies the Determinant Conjecture by~\cite[Theorem~3.14~(5) on
      page~128]{Lueck(2002)}.  If $K$ is a Farrell-Jones group, then
      $K \cap N_GH \subseteq N_GH$ is a Farrell-Jones group and is isomorphic to a
      subgroup of $W_GH$ of finite index which implies that $W_GH$ is a Farrell-Jones
      group, see~\cite[Theorem~16.5~(iia) and~(iif) on page~503]{Lueck(2025book)}.
      \\[1mm]~\ref{the:Basic_properties_of_L_upper_2-torsion:Sum_formula}%
~\ref{the:Basic_properties_of_L_upper_2-torsion:passage_to_subgroups_of_finite_index}
      and~\ref{the:Basic_properties_of_L_upper_2-torsion:induction} The proofs
      of~\cite[Theorem~3.93~(2),~(5), and~(6) on page~161]{Lueck(2002)} carry over.
      \\[1mm]~\ref{the:Basic_properties_of_L_upper_2-torsion:Product_formula}
      We can replace $\chi^{(2)}(Z;\caln(H))$ by the orbifold Euler characteristic
\[\chi_{\operatorname{orb}}(Z) = \sum_{e} (-1)^{\dim(e)} \cdot \frac{1}{|H_e|},
\]
where $e$ runs through the equivariant cells of $Z$, because of
$\chi^{(2)}(Z;\caln(H)) = \chi_{\operatorname{orb}}(Z)$, see~\cite[Subsection~6.6.1]{Lueck(2002)}. Using
assertions~\ref{the:Basic_properties_of_L_upper_2-torsion:G-homotopy_invariance}
and~\ref{the:Basic_properties_of_L_upper_2-torsion:Sum_formula} one can reduce the claim
to the special case $Z = H/L$ for any finite subgroup $L \subseteq H$ by induction over
the dimension of $Z$ and subinduction over the number of top-dimensional equivalent
cells of $Z$. We get an isomorphism of proper finite $G \times H$-$CW$-complexes
$G \times H \times_{G \times L} \pr_L^* X \xrightarrow{\cong} Y \times H/L$ by sending
$((g,h),x)$ to $(gx,hL)$, where 
$\pr_L^* X$ is the proper finite $G \times L$-$CW$-space obtained from $X$ by restriction with the projection 
$\pr_L\colon  G \times L \to G$. We conclude from
assertion~\ref{the:Basic_properties_of_L_upper_2-torsion:induction}  that it suffices to show
the claim for the finite proper $G \times L$-$CW$-complex $\pr_L^*X$.  This
follows from assertion~\ref{the:Basic_properties_of_L_upper_2-torsion:passage_to_subgroups_of_finite_index}  applied
to $G \subseteq G \times L$.
      \\[1mm]~\ref{the:Basic_properties_of_L_upper_2-torsion:finite_quotient}
      This follows from~\cite[Lemma~13.45 on page~473]{Lueck(2002)}.
      \\[1mm]~\ref{the:Basic_properties_of_L_upper_2-torsion:Poincare_duality} This is
      proved in~\cite[Theorem~3.93~(3) on page~161]{Lueck(2002)} under the assumption that
      the $G$-action is free.  If $G$ contains a subgroup of finite index which acts
      freely on $M$ then the claim follows from this special case and
      assertion~\ref{the:Basic_properties_of_L_upper_2-torsion:passage_to_subgroups_of_finite_index}.
      The general case requires some additional arguments which we do not present
      here.
    \end{proof}

         For a group $G$ define its \emph{$n$th  $L^2$-Betti number} $b_n^{(2)}(G)$ to be $b_n^{(2)}(\eub{G}; \caln(G))$
        for any $G$-$CW$-model for $\eub{G}$. This is the same as $b_n^{(2)}(EG; \caln(G))$ for any $G$-$CW$-model
        of $EG$, see~\cite[Theorem~6.54~(1) ans (2) on page~265]{Lueck(2002)}.
        
    \begin{definition}[$L^2$-acyclic group]\label{def:L2_acyclic_group}
    A group $G$ is called \emph{$L^2$-acyclic} if $b_p^{(2)}(G)$ vanishes for all $n \in \IZ^{\ge 0}$.
  \end{definition}

    \begin{definition}[$\FIN$-finite group]\label{def:FIN-finite_group}
    A group $G$ is called \emph{$\FIN$-finite} if it has a has a finite $G$-$CW$-model  for
    the classifying space of proper $G$-actions $\eub{G} = \EGF{G}{\FIN}$.
  \end{definition}

  \begin{definition}[$\det$-$L^2$-acyclic group]\label{def:det_L2-acyclic_group}
    A  $\FIN$-finite group $G$ is called \emph{$\det$-finite} if one (and hence every)
    finite $G$-$CW$-model for $\eub{G}$ is of determinant class.  It is called
    $\det$-$L^2$-acyclic if its both $L^2$-acyclic and $\det$-finite, or, equivalently,
    one (and hence every) finite $G$-$CW$-modle for $\eub{G}$ is $\det$-$L^2$-acyclic.
  \end{definition}

    \begin{definition}[$L^2$-torsion of groups]\label{the:L2-torsion_of_groups}
      Suppose that the $\FIN$-finite  group $G$ is $\det$-$L^2$-acyclic and satisfies condition (DFJ) appearing in 
      Definition~\ref{def:Condition_(DFJ}
      
      Define \emph{its $L^2$-torsion}
      \[
      \rho^{(2)}(G) = \rho^{(2)}(X;\caln(G))
       \]
       for any finite $G$-$CW$-model $X$ for $\eub{G}$.
     \end{definition}

     This is independent of the choice of the finite $G$-$CW$-model for $\eub{G}$
     by Theorem~\ref{the:Basic_properties_of_L_upper_2-torsion}~%
\ref{the:Basic_properties_of_L_upper_2-torsion:G-homotopy_invariance}.

  %%%%%%%%%%%%%%%%%%%%%%%%%%%%%%%%%%%%%%%%%%%%%%%%%%%%%%%%%%%%%%%%%%%%%
%%%%%%%%%%%%%%%%%%%%%%%%%%%%%% Section 3 %%%%%%%%%%%%%%%%%%%%%%%%%%%%%%%%
%%%%%%%%%%%%%%%%%%%%%%%%%%%%%%%%%%%%%%%%%%%%%%%%%%%%%%%%%%%%%%%%%%%%%

\typeout{---------- Section 3:  Blowing up orbits by classifying spaces of families   ---------------}

\section{Blowing up orbits by classifying spaces of families}%
\label{sec:Blowing_up_orbits_by_classifying_spaces_of_families}

Consider a family of subgroups $\calf$.  Let $X$ be a $G$-$CW$-complex with skeletal filtration
\[
X_{-1} =\emptyset \subseteq X_0 \subseteq X_2 \subseteq \cdots \subseteq X = \colim_{n \to \infty} X_n.
\]
Let $I_n$ be the set of open $n$-cells of $X/G$.
Choose for every $n \in \IZ^{\ge 0}$ a cellular $G$-pushout

\begin{equation}
\xymatrix@!C=11em{\coprod_{i_n \in I_n} G/H^n_{i_n} \times S^{n-1} \ar[r]^-{\coprod_{i_n \in I_n} q_{i_n}^n} \ar[d]
    &
    X_{n-1} \ar[d]
    \\
    \coprod_{i_n \in I_n} G/H^n_{i_n} \times D^n \ar[r]^-{\coprod_{i_n \in I_n} Q_{i_n}^n} 
    &
    X_n.
  }
  \label{cellular_G-pushout_general}
\end{equation}
For each $n \in \IZ^{\ge 0}$ and $i_n  \in I_n $ let $E^n_{i_n}$ be an $H_i^n$-$CW$-complex
which is a model for $\EGF{H_i^n}{\calf|_{H^n_i}}$, where  $\calf|_{H^n_i}$ is the family of subgroups of
$H^n_{i_n}$ given by $\{K \cap H^n_{i_n} \mid K \in \calf\}$. Fix a model $Z$ for $\EGF{G}{\calf}$.
We get a filtration by $G$-cofibrations
  \begin{equation}
    Z\times X_{-1} =\emptyset \subseteq Z\times X_0 \subseteq Z\times X_2 \subseteq \cdots
    \subseteq Z\times X = \colim_{n \to \infty} (Z\times X_n).
    \label{filtration_of_EGF(G)(calfin)_times_X}
  \end{equation}
  and $G$-pushouts
  \begin{equation}
    \xymatrix@!C=15em{\coprod_{i_n \in I_n} Z  \times G/H^n_{i_n} \times S^{n-1}
      \ar[r]^-{\coprod_{i_n \in I_n} \id_{Z \times q_{i_n}^n}} \ar[d]
    &
    Z \times X_{n-1} \ar[d]
    \\
    \coprod_{i_n \in I_n} Z\times G/H_{i_n} \times D^n \ar[r]^-{\coprod_{i_n \in I_n} \id_{Z} \times Q_{i_n}^n}
    &
    Z\times X_n,
  }
  \label{EGF(G)(calf)_times_cellular_G-pushout}
\end{equation}
where here and in the sequel we use on a product of two $G$-spaces the diagonal $G$-action.
Let $Z|_{H_i^n}$ be the restriction of the $G$-$CW$-complex $Z$ to $H^n_{i_n} \subseteq G$.
We have the $G$-homeomorphism
\[G \times_{H^n_{i_n}} Z|_{H_i^n} \xrightarrow{\cong} Z\times G/H_i^n,
  \quad (g,z) \mapsto (gz,gH_i^n).
\]
Since $Z|_{H_i^n}$ is a model for $\EGF{H^n_{i_n}}{\calf|_{H^n_{i_n}}}$,
we can choose a $H_i^n$-homotopy equivalence $E^n_{i_n} \xrightarrow{\simeq} Z|_{H_i^n}$.
Thus we get a $G$-homotopy equivalence 
\[
  \varphi^n_{i_n} \colon G \times_{H^n_{i_n}} E^n_{i_n}\xrightarrow{\simeq} Z \times G/H^n_{i_n}.
 \]
 Next we construct a sequence of inclusions of $G$-$CW$-complexes
 \[
 Y_{-1} =\emptyset \subseteq Y_0 \subseteq Y_2 \subseteq Y_3 \subseteq  \cdots
\]
and cellular $G$-homotopy equivalences $f_n \colon Y_n \to Z \times X_n$
satisfying $f_{n+1}|_{Z\times X_n} = f_n$ by induction over $n = -1,0,1,2, \ldots$. The
induction beginning is trivial as $Y_{-1} = Z\times X_{-1} = \emptyset$ holds. The
induction step from $(n-1)$ to $n \ge 0$ is done as follows. Since
$f_{n-1} \colon Y_{n-1} \to Z \times X_{n-1}$ is a $G$-homotopy equivalence, we can find a
cellular $G$-map
\[
  \mu^n  \colon  \coprod_{i_n \in I_n}G \times_{H^n_{i_n}} E^n_{i_n} \times S^{n-1} \to Y_{n-1}
\]
and a $G$-homotopy
\[
h^n \colon \Bigl(\coprod_{i_n \in I_n}G \times_{H^n_{i_n}} E^n_{i_n} \times S^{n-1}\Bigr) \times I \to Z \times X_{n-1}
\]
satisfying $h^n_0 = f_{n-1} \circ \mu_{n-1}$ and
$h^n_1 = \coprod_{i_n \in I_n} (\id_Z \times q^n_{i_n})\circ (\varphi^n_{i_n} \times \id_{S^{n-1}})$.
Define the $G$-$CW$-complex $Y_n$ by the $G$-pushout
\begin{equation}
  \xymatrix{\coprod_{i \in I_n} G \times_{H^n_{i_n}} E^n_{i_n} \times S^{n-1} \ar[r]^-{\mu^n} \ar[d]
    &
    Y_{n-1} \ar[d]
    \\
    \coprod_{i \in I_n} G \times_{H^n_{i_n}} E^n_{i_n} \times D^n \ar[r]
    &
    Y_n.}
  \label{G-pushout_for_Y_n}
\end{equation}
Consider the following commutative diagram
\[
  \xymatrix@!C=13em{\coprod_{i \in I_n} G \times_{H^n_{i_n}} E^n_{i_n} \times D^n \ar[d]_{J_0} 
    &
    \coprod_{i \in I_n} G \times_{H^n_{i_n}} E^n_{i_n} \times S^{n-1} \ar[l]  \ar[d]_{j_0} \ar[r]^-{\mu^n}
    &
    Y_{n-1}  \ar[d]_{f_{n-1}}
    \\
    \Bigl(\coprod_{i \in I_n} G \times_{H^n_{i_n}} E^n_{i_n} \times D^n\Bigr) \times I
    &
    \Bigl(\coprod_{i \in I_n} G \times_{H^n_{i_n}} E^n_{i_n} \times S^{n-1} \Bigr) \times I \ar[l] \ar[r]_-{h^n}
    &
    Z \times X_{n-1}    
    \\
    \coprod_{i \in I_n} G \times_{H^n_{i_n}} E^n_{i_n} \times D^n \ar[d]_{\coprod_{i \in I_n} \varphi^n_{i_n} \times \id_{D^n}}
    \ar[u]^{J_1}
    &
    \coprod_{i \in I_n} G \times_{H^n_{i_n}} E^n_{i_n} \times S^{n-1}
    \ar[l] \ar[r]_-{h^n_1} \ar[d]_{\coprod_{i \in I_n} \varphi^n_{i_n} \times \id_{S^{n-1}}}
    \ar[u]^{j_1}
    &
    Z \times X_{n-1}     \ar[u]^{\id_{Z \times X_{n-1}}}  \ar[d]_{\id_{Z \times X_{n-1}}}
    \\
    \coprod_{i_n \in I_n} Z \times G/H_{i_n}^n \times D^n 
    &
    \coprod_{i_n \in I_n} Z \times G/H_{i_n}^n \times S^{n-1} \ar[l] \ar[r]_-{\coprod_{i_n \in I_n} \id_{Z \times q_{i_n}^n}} 
    &
    Z \times X_{n-1}    
  }
  \]
  where $j_k$ and $J_k$ come from the inclusion $\pt \to I = [0,1]$ with image $\{k\}$ for
  $k = 0,1$.  The pushout of the uppermost row is $Y_n$, see~\eqref{G-pushout_for_Y_n},
  whereas the pushout of the lowermost row is $Z \times X_n$,
  see~\eqref{EGF(G)(calf)_times_cellular_G-pushout}.  Let $T_2$ be the pushout of the
  second row and $T_3$ be the pushout of the third row.  Then we have pairs
  $(T_2,Z \times X_{n-1})$ and $(T_3,Z \times X_{n-1})$. The diagram above yields $G$-maps
  of pairs
  \begin{eqnarray*}
    (a,f_{n-1}) \colon (Y_n,Y_{n-1})
    & \to &
    (T_2,Z \times X_{n-1});
    \\
    (b,\id_{Z \times X_{n-1}})
    \colon
    (T_3,Z \times X_{n-1})
    & \to &
    (T_2,Z \times X_{n-1});
    \\
    (c,\id_{Z \times X_{n-1}})
    \colon
    (T_3,Z \times X_{n-1})
    & \to &
            (Z \times X_n,Z \times X_{n-1}).
  \end{eqnarray*}
  Since in the diagram above all vertical arrows are $G$-homotopy equivalences and the
  left horizontal arrow in each row is a $G$-cofibration, these three $G$-maps are
  $G$-homotopy equivalences of pairs and we can find a $G$-homotopy equivalence of pairs
  \[
    (b',\id_{Z \times X_{n-1}}) \colon (T_2,Z \times X_{n-1}) \to    (T_3,Z \times X_{n-1}).
  \]
  Now define the desired $G$-homotopy equivalence of pairs
  \[
    (f_n,f_{n-1}) \colon (Y_n,Y_{n-1}) \to (Z \times X_n, Z \times X_{n-1})
  \]
  by the composite $ (c,\id_{Z \times X_{n-1}}) \circ (b',\id_{Z \times X_{n-1}})  \circ (a,f_{n-1})$.
  If we put $Y = \colim_{n \to \infty} Y_n$,
  we obtain a $G$-$CW$-complex $Y$ and a $G$-homotopy equivalence
  \begin{equation} f = \colim_{n \to \infty} f_n \colon Y \xrightarrow{\simeq_G} Z \times X.
    \label{_G-homotopy_equivalence_f_colon_Y_to_Z_times_X}
  \end{equation}

  Next we collect the basic properties of this construction.

  \begin{theorem}[Blowing up]\label{the:properties_of_the_blow_up}
    \
    \begin{enumerate}
    \item\label{the:properties_of_the_blow_up:type_d} Consider a number
      $l \in \IZ^{\ge 0}$. Suppose for $n \in \IZ^{\ge 0}$ and $i_n \in I_n$ that the
      $H^n_{i_n}$-$CW$-complex $E^n_{i_n}$ has finite $(l -n)$ skeleton and $X$ has a
      finite $l$-skeleton.
      
      Then the $G$-$CW$-complex $Y$ has a finite $l$-skeleton;
    
    \item\label{the:properties_of_the_blow_up:finite_type} Suppose for $n \in \IZ^{\ge 0}$
      and $i_n \in I_n$ that the $H^n_{i_n}$-$CW$-complex $E^n_{i_n}$ is of finite type and the
      $G$-$CW$-complex $X$ is of finite type.

      Then the $G$-$CW$-complex $Y$ is of finite type;

    \item\label{the:properties_of_the_blow_up:dimension} Consider a number
      $d \in \IZ^{\ge 0}$. Suppose for $n \in \IZ^{\ge 0}$ and $i_n \in I_n$ that the
      $H^n_{i_n}$-$CW$-complex $E^n_{i_n}$ has dimension $\le (l-d)$ and $X$ has a
      dimension $\le d$.
     
      Then the $G$-$CW$-complex $Y$ has  dimension $\le d$.

    \item\label{the:properties_of_the_blow_up:finite} If for all $n \in \IZ^{\ge 0}$ and
      $i_n \in I_n$ the $H^n_{i_n}$-$CW$-complex $E^n_{i_n}$ is finite and the
      $G$-$CW$-complex $X$ is finite, then the $G$-$CW$-complex $Y$ is finite;

    \item\label{the:properties_of_the_blow_up:proper} If each element in $\calf$ is
      finite, then $Y$ and $Z \times X$ are proper.

    \item\label{the:properties_of_the_blow_up_L_upper_(2)-acyclic} If each of the
      $H^n_{i_n}$-$CW$-complexes $E^n_{i_n}$ is $L^2$-acyclic, then the $G$-$CW$-complex
      $Z \times X$ is $L^2$-acyclic;

    \item\label{the:properties_of_the_blow_up:L_upper_(2)_torsion}
      Suppose  that the following conditions are satisfied:
      \begin{enumerate}
       \item $G$ satisfies condition~(DFJ), see Definition~\ref{def:Condition_(DFJ};
       \item Every element in $\calf$ is finite;
       \item For all $n \in \IZ^{\ge 0}$ and $i_n \in I_n$ the $H^n_{i_n}$-$CW$-complex $E^n_{i_n}$ is
      finite and $\det$-$L^2$-acyclic;
    \item $X$ is a finite $G$-$CW$-complex.
    \end{enumerate}
    
      Then the $G$-$CW$-complex $Y$ is finite, proper, and
      $\det$-$L^2$-acyclic, and we get for its $L^2$-torsion
      \[
        \rho^{(2)}(Y;\caln(G))
        = \sum_{n \ge 0} \sum_{i \in I_n} (-1)^n \cdot \rho^{(2)}(E^n_{i_n}; \caln(H^n_{i_n})).
      \]
      Suppose additionally that $X^H$ is contractible for any finite subgroup $H \subseteq G$. Then
      each group $H_i^{i_n}$ satisfies condition (DFJ), 
        there is a finite $G$-$CW$-model for $\eub{G}$ which is $\det$-$L^2$-acyclic, and we get
      \[
        \rho^{(2)}(G)
        = \sum_{n \ge 0} \sum_{i \in I_n} (-1)^n \cdot \rho^{(2)}(H^n_{i_n}).
      \]
      
      \end{enumerate}
    \end{theorem}
    \begin{proof}
      Assertions~\ref{the:properties_of_the_blow_up:type_d},~\ref{the:properties_of_the_blow_up:finite_type},%
\ref{the:properties_of_the_blow_up:dimension},~\ref{the:properties_of_the_blow_up:finite}
      and~\ref{the:properties_of_the_blow_up:proper}
      follow directly from the construction of $Y$ by inspecting the $G$-pushouts~\eqref{G-pushout_for_Y_n}.

      Assertion~\ref{the:properties_of_the_blow_up_L_upper_(2)-acyclic} follows
      from~\cite[Theorem~6.54 on page~265]{Lueck(2002)}.

      Assertion~\ref{the:properties_of_the_blow_up:L_upper_(2)_torsion} follows from
      Theorem~\ref{the:Basic_properties_of_L_upper_2-torsion} and the fact that the
      projection $\eub{G} \times X \to \eub{G}$ is a $G$-homotopy equivalence if $X^H$ is
      contractible for every finite subgroup $H$.
\end{proof}

We extract out a version of (viii) which appeared in the introduction.

\thmCombination

\begin{example}[Group extensions]\label{exa:group_extension_general}
  Suppose we can write $G$ as an extension $1 \to K \to G \xrightarrow{p} Q \to 1$, and
  there is a finite model for $\eub{Q}$.  Then we can consider the $G$-$CW$-complex
  $X = p^*\eub{Q}$ obtained from the finite $Q$-$CW$-complex $\eub{Q}$ by restriction with
  $p$.  Obviously $X^H = \eub{Q}^{p(H)}$ is contractible for any finite subgroup
  $H \subseteq G$.  There is a bijective correspondence between the open equivariant cells
  of $X$ and the open equivariant cells of $\eub{Q}$ given by $c \mapsto p(c)$. We have
  $\dim(c) = \dim(p(c))$ and $G_c = p^{-1}(Q_{p(c)})$.  Hence $G_c$ contains $K$ as a
  subgroup of finite index $[G_c : K] = |Q_{p(c)}|$.  Suppose that $G_c$ has a finite $G_c$-model
  for $\eub{G_c}$  which is $\det$-$L^2$-acyclic. Assume that $G$ satisfies condition (DFJ), see
  Definition~\ref{def:Condition_(DFJ}. We conclude from
  Theorem~\ref{the:Basic_properties_of_L_upper_2-torsion}~%
\ref{the:Basic_properties_of_L_upper_2-torsion:passage_to_subgroups_of_finite_index}
  that $|Q_{p(c)}| \cdot \rho^{(2)}(G_c) = \rho^{(2)}(K)$. Define the orbifold Euler
  characteristic of $\eub{Q}$ to be
  \begin{equation}
    \chi_{\operatorname{orb}}(\eub{Q})  =\sum_{c} (-1)^{\dim(e)} \cdot \frac{1}{|Q_{p(c)}|}.
    \label{orbifold_Euler_characteristic}
  \end{equation}
  If $\chi^{(2)}(\eub{Q};\caln(Q))$ is the $L^2$-Euler characteristic, we get,
  see~\cite[Subsection~6.6.1]{Lueck(2002)}.
  
  \begin{equation} \chi_{\operatorname{orb}}(\eub{Q}) = \chi^{(2)}(\eub{Q};\caln(Q)).
    \label{chi_L2_is_chi_orbi}
  \end{equation}
        Theorem~\ref{the:properties_of_the_blow_up}~\ref{the:properties_of_the_blow_up:L_upper_(2)_torsion}
      implies that both $K$ and $G$ have a $\det$-$L^2$-acyclic finite proper  model for their
      classifying space of proper actions and we have      
      \[\rho^{(2)}(G) = \chi_{\operatorname{orb}}(\eub{Q}) \cdot \rho^{(2)}(K)
        = \chi^{(2)}(\eub{Q};\caln(Q)) \cdot \rho^{(2)}(K).
    \]
    
    In particular we get $\rho^{(2)}(G) = 0$ if $\chi^{(2)}(\eub{Q};\caln(Q))$ vanishes.
  \end{example}

  \begin{example}[Graph of groups]\label{exa:graph_of_groups}
    Let $Y$ be a connected non-empty graph in the sense of~\cite[Definition~1 in
    Section~2.1 on page~13]{Serre(1980)}.  Let $(G,Y)$ be a graph of groups in the sense
    of~\cite[Definition~8 in Section~4.4. on page~37]{Serre(1980)}.  (In the sequel we use
    the notation of~\cite{Serre(1980)}).  Denote by $\overline{\operatorname{edge}(Y)}$
    the quotient of $\operatorname{edge}(Y)$ under the involution
    $y \mapsto \widetilde{y}$. Note that $y$ considered as a $CW$-complex has the set
    $\operatorname{vert}(T)$ as set of $0$-cells and $\overline{\operatorname{edge}(Y)}$
    as set of $1$-cells. Since by definition $G_{y} = G_{\widetilde{y}}$ holds, we can
    define for $\overline{y} \in \overline{\operatorname{edge}(Y)}$ the group
    $G_{\overline{y}}$ to be $G_y$ for a representative $y \in\operatorname{edge}(Y)$ of
    $\overline{y}$.
             
    Let $P_0$ be an element in $\operatorname{vert}(Y)$ and let $T$ be a maximal tree of
    $Y$.  Let $\pi_1(G,Y,P_0)$ and $\pi_1(G,Y,T)$ be the fundamental groups in the sense
    of~\cite[page~42]{Serre(1980)}. Note that $\pi_1(G,Y,P_0)$ and $\pi_1(G,Y,T)$ are
    isomorphic, see~\cite[Proposition~20 in Section~5.1. on page~44]{Serre(1980)}. Then
    there exists
    \begin{itemize}
    \item A graph $\widetilde{X} = \widetilde{X}(G,Y,T)$;
    \item An action of $\pi = \pi_1(G,Y,T)$ on $\widetilde{X}$;
    \item A morphism $p \colon \widetilde{X} \to X$ inducing an isomorphism
      $\pi\backslash \widetilde{X} \to Y$,
    \end{itemize}
    such that the following is true:
    \begin{itemize}
    \item $\widetilde{X}$ is a tree;
    \item $\widetilde{X}^H$ is contractible for every finite subgroup $H \subseteq \pi$;
    \item $\widetilde{X}$ is a $1$-dimensional $\pi$-$CW$-complex for which there exists
      $\pi$-pushout
      \[
        \xymatrix{\coprod_{P \in \operatorname{vert}(T)} \pi/\pi_P \times S^0 \ar[r]
          \ar[d] & \coprod_{\overline{y} \in \overline{\operatorname{edge}(Y)}} \pi/\pi_y
          \ar[d]
          \\
          \coprod_{P \in \operatorname{vert}(T)} \pi/\pi_P \times D^1 \ar[r] &
          \widetilde{X} }
      \]
      such that $\pi_P \cong G_P$ for $P \in \operatorname{vert}(T)$ and
      $\pi_{\overline{y}} = G_{\overline{y}}$ for
      $\overline{y} \in \overline{\operatorname{edge}(Y)}$ holds.
    \end{itemize}
      
    All these claims follows from~\cite[Section~5.3 and Theorem~15 in Section~6.1 on
    page~58]{Serre(1980)}.

    Now suppose that $Y$ is finite, each of the groups $G_P$ for $P \in \operatorname{vert}(T)$ and
    $G_{\overline{y}}$ for $\overline{y} \in \overline{\operatorname{edge}(Y)}$ has a
    finite model for its classifying space for proper actions which is
    $\det$-$L^2$-acyclic, and $\pi$ satisfies condition (DFJ).

    Then we conclude from Theorem~\ref{the:application_to_automorphism} that the groups
    $G_P$ for $P \in \operatorname{vert}(T)$ and $G_{\overline{y}}$ for
    $\overline{y} \in \overline{\operatorname{edge}(Y)}$ satisfy condition (DFJ), there is
    a finite $\pi$-$CW$-model for $\eub{\pi}$ which is $\det$-$L^2$- acyclic, and we get
    \[
      \rho^{(2)}(\pi)= \sum_{P \in \operatorname{vert}(T)} \rho^{(2)}(G_P)
      -\sum_{\overline{y} \in \operatorname{vert}(T)} \rho^{(2)}(G_{\overline{y}}).
    \]
  \end{example}

  \begin{example}[Amalgamated Products]\label{exa:Amalgamated_products}
    Let $G_0$ be a common subgroup of the group $G_1$ and the group $G_2$. Denote by
    $G = G_1 \ast_{G_0} G_2$ the amalgamated product.  Suppose that there is a
    $\det$-$L^2$-acyclic finite $G_i$-$CW$-model for $\eub{G_i}$ for $i = 0,1,2$ and that $G$
    satisfies condition (DFJ).  Then $G_i$ satisfies condition (DFJ) for $i =0,1,2$, there
    is a $\det$-$L^2$-acyclic finite $G$-$CW$-module for $\eub{G}$, and we get
    \[
      \rho^{(2)}(G) = \rho^{(2)}(G_1) + \rho^{(2)}(G_2) - \rho^{(2)}(G_0).
    \]
    This follows from Example~\ref{exa:graph_of_groups} applied to the graph of groups
    associated to $G_1 \ast_{G_0} G_2$, see~\cite[page~43]{Serre(1980)}.
  \end{example}

%%%%%%%%%%%%%%%%%%%%%%%%%%%%%%%%%%%%%%%%%%%%%%%%%%%%%%%%%%%%%%%%%%%%%
%%%%%%%%%%%%%%%%%%%%%%%%%%%%%% Section 4  %%%%%%%%%%%%%%%%%%%%%%%%%%%%%%%%
%%%%%%%%%%%%%%%%%%%%%%%%%%%%%%%%%%%%%%%%%%%%%%%%%%%%%%%%%%%%%%%%%%%%%

\typeout{---------- Section 5:  $L^2$-torsion of selfhomotopy equivalences  ---------------}

\section{\texorpdfstring{$L^2$}{L2}-torsion of a selfhomotopy equivalence}%
\label{sec:L2-torsion_of_a_selfhomotopy_equivalence}

For the remainder of this section we fix a group automorphism
$\Phi \colon G \xrightarrow{\cong} G$.

Let $G \times_{\Phi} \IZ$ be the associated semidirect product.  In the sequel $t \in \IZ$
is a fixed generator of $\IZ$. Then we can write every element in $G \times_{\Phi} \IZ$
uniquely as $gt^n$ for $g \in G$ and $n \in \IZ$ and the multiplication in
$G \rtimes_{\Phi} \IZ$ is given by
$g_0t^{n_0}g_1t^{n_1} = g_0\Phi^{n_0}(g_1)t^{n_0 + n_1}$.

Given two $G$-spaces $X$ and $Y$, a \emph{$\Phi$-map} $f \colon X \to Y$ is a map $f$ from
$X$ to $Y$ satisfying $f (gx) = \Phi(g)f(x)$ for $g \in G$ and $x \in X$. We call $f$ a
\emph{$\Phi$-homotopy equivalence} if there exists a $\Phi^{-1}$-map $f' \colon Y \to X$
such that $f' \circ f$ is $G$-homotopic to $\id_X$ and $f \circ f'$ is $G$-homotopic to
$\id_Y$. We call $f$ a \emph{weak $\Phi$-homotopy equivalence} if
$f^H \colon X^H \to Y^{\Phi(H)}$ is a weak homotopy equivalence for every subgroup
$H \subseteq G$.

A \emph{$G$-$CW$-approximation} $(X,a)$ of a $G$-space $Y$ is a $G$-$CW$-complex $X$
together with a weak $G$-homotopy equivalence $a \colon X \to Y$. Every $G$-space $Y$
admits a $G$-$CW$-approximation, see~\cite[Proposition~2.3 on page~35]{Lueck(1989)}.
Given two $G$-$CW$-approxi\-ma\-tions $(X,a)$ and $(X',a')$ there exists a cellular
$G$-homotopy equivalence $s \colon X \to X'$ which is uniquely characterized up to $G$-homotopy by the
property that $a' \circ s$ and $a$ are $G$-homotopic.  This follows from the Equivariant
Approximation Theorem, see~\cite[Theorem~2.1 on page~32]{Lueck(1989)}, and the
Equivariant Whitehead Theorem, see~\cite[Theorem~2.4 on page~36]{Lueck(1989)}.

\begin{definition}[$\FIN$-finite $G$-space]\label{def:FIN-finite_G-space}
    A $G$-space $Y$ is called \emph{$\FIN$-finite}, if it possesses a $G$-$CW$-approximation $a \colon X \to Y$
    with a finite proper $G$-$CW$-complex as source.
  \end{definition}

  Obviously the property $\FIN$-finite depends only on the $G$-homotopy type of $Y$,
  actually, only on the weak $G$-homotopy type.  Note that a group $G$ is $\FIN$-finite in
  the sense of Definition~\ref{def:FIN-finite_group} if and only if the $G$-space
  $\eub{G}$ is $\FIN$-finite in the sense of Definition~\ref{def:FIN-finite_G-space}.

  Recall that a finite proper $G$-$CW$-complex $X$ is called \emph{of determinant class},
  if the associated Hilbert $\caln(G)$-chain complex $L^2(G) \otimes_{\IZ G} C_*^c(X)$ is
  of determinant class in the sense of~\cite[Definition~3.20 on page~140]{Lueck(2002)}.
  Note that we are not demanding that $X$ is $L^2$-acyclic and the property being of
  determinant class depends only on the $G$-homotopy type of $X$. This follows
  from~\cite[Theorem~3.35~(1) on page~142]{Lueck(2002)}, since the mapping cone of a
  $\IQ G$-chain homotopy equivalence of finite projective $\IQ G$-chain complexes is a
  contractible finite projective $\IQ G$-chain complexes and hence of determinant class
  by~\cite[Lemma~2.18 on page~83 and Lemma~3.30 on page~140]{Lueck(2002)}. If $G$
  satisfies the Determinant Conjecture, then every finite $G$-$CW$-complex $X$ is of
  determinant class.

  \begin{definition}[Det-finite $G$-space]\label{def:det_finite_G-space}
    We call a $G$-space $Y$ \emph{$\det$-finite} if there exists a $G$-$CW$-approximation $(X,a)$
    with a finite proper $G$-$CW$-complex $X$ as source which is of determinant class.
  \end{definition}

  Consider a weak $\Phi$-homotopy equivalence $f \colon Y \to Y$ of a $\FIN$-finite
  $G$-space $Y$.  Choose a $G$-$CW$-approximation $(X,a)$ with a finite proper
  $G$-$CW$-com\-plex $X$ as source, and a $\Phi$-homotopy equivalence
  $\widehat{f} \colon X \to X$ such that $a \circ \widehat{f}$ and $f \circ a$ are
  $\Phi$-homotopic.  We assign to $\widehat{f} $ a finite proper
  $G \times_{\Phi} \IZ$-$CW$-complex $T_{\widehat{f};\Phi}$ by the $G \times_{\Phi} \IZ$-
  pushout
  \begin{equation}
    \xymatrix{(G \rtimes_{\Phi} \IZ) \times_G X \times \{0,1\} \ar[r]^-{q} \ar[d]_i
      &
      (G \rtimes_{\Phi} \IZ) \times_G X \ar[d]
      \\
      (G \rtimes_{\Phi} \IZ) \times_G  X \times [0,1]
      \ar[r]
      &
      T_{\widehat{f};\Phi}
    }
    \label{def_T_upper_Phi_widehat(f)}
  \end{equation}
  where $i$ is the obvious inclusion and $q$ sends $(gt^n, x, k)$ to $(gt^n, x)$ if $k =0$
  and to $(gt^{n-1}, \widehat{f}(x))$ if $k = 1$. Note that $T_{\widehat{f};\Phi}$ is a
  kind of to both sides infinite mapping telescope.  If $\Phi = \id_G$, the quotient space
  $T_{\widehat{f};\Phi}/\IZ$ is the mapping torus of $\widehat{f}$.  The quotient space
  $T_{\widehat{f};\Phi}/(G \rtimes_{\Phi}\IZ)$ is the mapping torus of $\widehat{f}/G$.
  
  \begin{lemma}\label{lem_T_upper_PhI_u_det_L2-acyclic}
    \begin{enumerate} 
    \item\label{lem_T_upper_PhI_u_det_L2-acyclic:L2_acyclic}
      The finite $G \rtimes_{\Phi} \IZ$-$CW$-complex $T_{\widehat{f};\Phi}$ is $L^2$-acyclic;
    \item\label{lem_T_upper_PhI_u_det_L2-acyclic:det_L2-acyclic}
      If $X$ is of determinant class, then $T_{\widehat{f};\Phi}$ is $\det$-$L^2$-acyclic;
    \item\label{lem_T_upper_PhI_u_det_L2-acyclic:peridoc_implies_det_L2-acyclic} Suppose
      that there is $n \in \IZ^{\ge 1}$ such that $\Phi^n = \id_G$ and $f^n$ is
      $G$-homotopic to the identity.  Then $T_{\widehat{f};\Phi}$ is $\det$-$L^2$-acyclic.
    \end{enumerate}
  \end{lemma}
  \begin{proof}~\ref{lem_T_upper_PhI_u_det_L2-acyclic:L2_acyclic}
    Consider $d \in \IZ^{\ge 1}$.  Let $\pr \colon G \times_{\Phi} \IZ \to \IZ$ be the
    projection.  Then $\pr^{-1}(d \cdot \IZ)$ has index $d$ in $G \times_{\Phi} \IZ$ and
    can be identified with $G \times_{\Phi^d} \IZ$.  The restriction of
    $T_{\widehat{f};\Phi}$ to $\pr^{-1}(d \cdot \IZ) = G \times_{\Phi^d} \IZ$ is
    $G \times_{\Phi^d} \IZ$-homotopy equivalent to $T_{\widehat{f}^d;\Phi^d}$.  We get
    from Theorem~\ref{the:Basic_properties_of_L_upper_2-torsion}~%
\ref{the:Basic_properties_of_L_upper_2-torsion:G-homotopy_invariance} 
    and~\ref{the:Basic_properties_of_L_upper_2-torsion:passage_to_subgroups_of_finite_index}
    \[
      b_m^{(2)}(T_{\widehat{f};\Phi};\caln(G \rtimes_{\Phi} \IZ))
      =
     \frac{1}{[d]}\cdot b_m^{(2)}(T_{\widehat{f}^d;\Phi^d};\caln(G \rtimes_{\Phi^d} \IZ)).
   \]
   If $C$ is the number of cells of $X/G$, then the number of cells of $T_{\widehat{f}^d;\Phi^d}/(G \rtimes_{\Phi^d} \IZ)$
   is bounded by $2C$. This implies
   \begin{multline*}
     b_m^{(2)}(T_{\widehat{f}^d;\Phi^d};\caln(G \rtimes_{\Phi^d} \IZ))
   \\
   \le
   \dim_{\caln(G \rtimes_{\Phi^d} \IZ)} (L^2(G \rtimes_{\Phi^d} \IZ) \otimes_{\IZ[G \rtimes_{\Phi^d} \IZ]}
   C_m(T_{\widehat{f}^d;\Phi^d}))
   \le 2C.
 \end{multline*}
 Hence we get for every $d \in \IZ^{\ge 1}$
 \[
   b_m^{(2)}(T_{\widehat{f};\Phi};\caln(G \rtimes_{\Phi} \IZ)) \le \frac{2C}{d}.
 \]
 Therefore $b_m^{(2)}(T_{\widehat{f};\Phi};\caln(G \rtimes_{\Phi} \IZ))$ vanishes for
 every $m \in \IZ^{\ge 0}$.  \\[1mm]~\ref{lem_T_upper_PhI_u_det_L2-acyclic:det_L2-acyclic}
 The $G \times_{\Phi} \IZ$-$CW$-complex $(G \rtimes_{\Phi} \IZ )\times_G X$ is of
 determinant class, since the $G$-$CW$-complex $X$ is of determinant class,
 see~\cite[Theorem~3.14~(6) on page~129]{Lueck(2002)}.  Hence $T_{\widehat{f};\Phi}$ is of
 determinant class by~\cite[Theorem~3.35~(1) on page~142]{Lueck(2002)}.
 \\[1mm]\ref{lem_T_upper_PhI_u_det_L2-acyclic:peridoc_implies_det_L2-acyclic} By
 inspecting the proof of assertion~\ref{lem_T_upper_PhI_u_det_L2-acyclic:L2_acyclic}, we
 can show the following: There is an inclusion of groups
 $G \times \IZ \to G \rtimes_{\Phi} \IZ$ of finite index $n$ such that the restriction
 $i^*T_{\widehat{f};\Phi}$ of the proper finite $G \rtimes _{\Phi} \IZ$-$CW$-complex
 $T_{\widehat{f};\Phi}$ to $G \times \IZ$ with $i$ is $G \times \IZ$-homotopy equivalent
 to $X \times \IR$ with the obvious $G \times \IZ$-action coming from the given $G$-action
 on $X$ and the $\IZ$-action on $\IR$ given by translation.  We conclude
 from~\cite[Theorem~3.14~(5) on page~128]{Lueck(2002)} that it suffices to show that the
 proper finite $G$-$CW$-complex $i^*T_{\widehat{f};\Phi}$ is of determinant class. We have
 explained already above that the property being of determinant class depends only on the
 homotopy type of a finite Hilbert $\caln(G \times \IZ)$-chain complex.  Therefore it
 suffices to show that the proper finite $G\times \IZ$-$CW$-complex $X \times \IR$ is of
 determinant class. Using~\cite[Theorem~3.35~(1) on page~142]{Lueck(2002)} one reduces the
 claim to the case $X = G/H$. Since the proper finite $G$-$CW$-complex $G/H \times \IR$ is
 $G$-homeomorphic to $(G \times \IZ)\times_{H \times \IZ} \IR$ with respect to the
 $H \times \IZ$-action on $\IR$ given by the projection $H \times \IZ \to \IZ$, it
 suffices to show by~\cite[Theorem~3.14~(6) on page~129]{Lueck(2002)} that the proper
 finite $H \times \IZ$-complex $\IR$ is of determinant class.  Because
 of~\cite[Theorem~3.14~(5) on page~128]{Lueck(2002)} it suffices to show that the proper
 finite $\IZ$-$CW$-complex $\IR$ is of determinant class. This follows by a direct
 inspection or the fact that the Determinant Conjecture holds for the group $\IZ$.
\end{proof}

\begin{definition}[Det-finite $\Phi$-self-homotopy equivalence]\label{def:det-finite_Phi-self-homotopy_equivalence}
  We call a weak $\Phi$-self-homotopy equivalence $f \colon Y \to Y$ of a $\FIN$-finite
  $G$-space $Y$ $\det$-finite if the following holds.
  For every $G$-$CW$-approximation $(X,a)$ with a finite proper
  $G$-$CW$-com\-plex $X$ as source and  every $\Phi$-homotopy equivalence
  $\widehat{f} \colon X \to X$ such that $a \circ \widehat{f}$ and $f \circ a$ are
  $\Phi$-homotopic, the finite
  $G \rtimes_{\Phi} \IZ$-$CW$-complex $T_{\widehat{f};\Phi}$ is of determinant class.
\end{definition}

Note  that for a $\det$-finite $\Phi$-automorphism the 
finite $G \rtimes_{\Phi} \IZ$-$CW$-complex $T_{\widehat{f};\Phi}$ is $\det$-$L^2$-acyclic by
Lemma~\ref{lem_T_upper_PhI_u_det_L2-acyclic}~\ref{lem_T_upper_PhI_u_det_L2-acyclic:L2_acyclic}.

\begin{definition}[$L^2$-torsion of a selfhomotopy equivalence]%
\label{def:L2-torsion_of_a_selfhomotopy_equivalence}
Consider a weak $\Phi$-self homotopy equivalence $f \colon Y \to Y$ of the $\FIN$-finite
$G$-space $Y$ which is $\det$-finite.  Choose a $G$-$CW$-approximation $(X,a)$ with a
finite proper $G$-$CW$-com\-plex $X$ as source, and a $\Phi$-homotopy equivalence
$\widehat{f} \colon X \to X$ such that $a \circ \widehat{f}$ and $f \circ a$ are
$\Phi$-homotopic.

  Then  $T_{\widehat{f};\Phi}$ is $\det$-$L^2$-acyclic and we define the $L^2$-torsion of $(f;\Phi)$ to be 
  \[
  \rho^{(2)}(f;\Phi) := \rho^{(2)}(T_{\widehat{f};\Phi};\caln(G \rtimes_{\Phi} \IZ)) \in \IR.
  \]
\end{definition}

In the remainder of this section will  show that the notion of $\rho^{(2)}(f;\Phi)$
appearing in Definition~\ref{def:L2-torsion_of_a_selfhomotopy_equivalence} is well-defined
and collect its main properties, namely we will prove the following Lemma~\ref{lem:checking_Det-finite_for_Phi-automorphism}
and the following Theorem~\ref{the:properties_of_L2-torsion_of_self_homo}.

\begin{lemma}\label{lem:checking_Det-finite_for_Phi-automorphism}
  Consider a $\Phi$-self-homotopy equivalence $f \colon Y \to Y$ of a $\FIN$-finite
  $G$-space $Y$.

  \begin{enumerate}

  \item\label{lem:checking_Det-finite_for_Phi-automorphism:check}
 Then $(f;\Phi)$ is $\det$-finite if and only there is a
  $G$-$CW$-approximation $(X,a)$ with a finite proper
  $G$-$CW$-com\-plex $X$ as source and  a  $\Phi$-homotopy equivalence
  $\widehat{f} \colon X \to X$ such that $a \circ \widehat{f}$ and $f \circ a$ are
  $\Phi$-homotopic and the finite
  $G \rtimes_{\Phi} \IZ$-$CW$-complex $T_{\widehat{f};\Phi}$ is of determinant class;

\item\label{lem:checking_Det-finite_for_Phi-automorphism:torsion}
  The number $\rho^{(2)}(f;\Phi)$ appearing in Definition~\ref{def:L2-torsion_of_a_selfhomotopy_equivalence}
  is independent of the choices of $(X,a)$ and $\widehat{f}$.
\end{enumerate}

\end{lemma}
\begin{proof}
  Suppose that $(X_l,a_l)$ for $l = 0,1$ is a $G$-$CW$-approximation of $Y$ with a finite
  proper $G$-$CW$-complex $X_l$  as source and we have a
  $\Phi$-homotopy equivalence $\widehat{f}_l \colon X_l \to X_l$ such that
  $a_l \circ \widehat{f}_l$ and $f \circ a_l$ are $\Phi$-homotopic.  Then we can choose a cellular
  $G$-homotopy equivalence $s'\colon X_0 \to X_1$ such that $a_1 \circ s'$ and $a_0$ are cellularly
  $G$-homotopic. We obtain a diagram of finite $G$-$CW$-complexes
  \[
    \xymatrix{X_0 \ar[r]^{\widehat{f}_0} \ar[d]_{s'}
      &
      X_0 \ar[d]^{s'}
        \\
        X_1 \ar[r]_{\widehat{f}_1}
        &
          X_1
        }
      \]
      where all arrows are cellular and which commutes up to cellular $\Phi$-homotopy.
      In the sequel we often abbreviate
      \[
      G_{\Phi} = G \times_{\Phi} \IZ.
      \]
      Let $s \colon G_{\Phi} \times_G X_0 \to G_{\Phi} \times_G X_1$ be the cellular $G_{\Phi}$-homotopy equivalence
      $s = \id_{G_{\Phi}} \times_G s'$. Let $s^{-1}  \colon G_{\Phi} \times_G X_1 \to G_{\Phi} \times_G X_0$ be
      some cellular $G_{\Phi}$-homotopy inverse of $s$.  Define 
      \[
       q_k \colon G_{\Phi} \times_G X_k \times \{0,1\} \to G_{\Phi} \times_G X_k
       \]
      by sending  $(gt^n,x,k)$  to $(gt^n,x)$, if $k = 0$, and to $(gt^{n-1},\widehat{f}_k(x))$, if $k = 1$.
      Choose a cellular $G_{\Phi}$-homotopy
      $h \colon G_{\Phi} \times_G X_0 \times \{0,1\} \times [0,1] \to G_{\Phi} \times_G X_1$
      satisfying $h_0 = s \circ q_0$ and  $h_1 =  q_1 \circ (s \times \id_{\{0,1\}})$.
      
    Consider the following commutative diagram of finite $G_{\Phi}$-$CW$-complexes
    \[\xymatrix@!C=11em{G_{\Phi} \times_G  X_0 \times [0,1] \ar[d]^{\id}
        &
       G_{\Phi} \times_G X_0 \times \{0,1\} \ar[r]^-{q_0}  \ar[l] \ar[d]^{\id}
      &
      G_{\Phi} \times_G X_0 \ar[d]^s
      \\
      G_{\Phi} \times_G  X_0 \times [0,1]  \ar[d]^{l_0}
        &
       G_{\Phi} \times_G X_0 \times \{0,1\} \ar[r]^-{s \circ q_0}  \ar[l] \ar[d]^{k_0}
      &
      G_{\Phi} \times_G X_1 \ar[d]^{\id}
      \\
      G_{\Phi} \times_G  X_0 \times [0,1] \times [0,1]
      &
       G_{\Phi} \times_G X_0 \times \{0,1\} \times [0,1]  \ar[r]^-{h}  \ar[l]
      &
      G_{\Phi} \times_G X_1
      \\
      G_{\Phi} \times_G  X_0 \times [0,1] \ar[u]_{l_1} \ar[d]^{s \times \id_{[0,1]}}
        &
       G_{\Phi} \times_G X_0 \times \{0,1\} \ar[r]^-{q_1 \circ (s \times \id_{\{0,1\}})}  \ar[l] \ar[u]_{k_1} \ar[d]^{s \times \id_{\{0,1\}}}
      &
      G_{\Phi} \times_G X_1  \ar[u]_{\id} \ar[d]^{\id}
      \\
       G_{\Phi} \times_G  X_1 \times [0,1] 
        &
       G_{\Phi} \times_G X_1 \times \{0,1\} \ar[r]^-{q_1}  \ar[l] 
      &
      G_{\Phi} \times_G X_1  
          }
    \]
    where the maps $k_m$ and $l_m$ are induced by the inclusion $\pt \to [0,1]$ with image $\{m\}$ for $m = 0,1$.
    
    Recall that $T_{\widehat{f}_0;\Phi}$ is the $G_{\Phi}$-pushout of the uppermost row
      and $T_{\widehat{f}_1;\Phi}$ is the $G_{\Phi}$-pushout of the lowermost row. Let
      $Z_m$ be the $G$-pushout of the $m$-th  row  for $m = 2,3,4$.
      Then the diagram above induced a zigzag of $G_{\Phi}$-homotopy equivalences
        of finite proper $G_{\Phi}$-$CW$-complexes
        \[
          T_{\widehat{f}_0;\Phi} \xrightarrow{u_1} Z_2  \xrightarrow{u_2} Z_3  \xleftarrow{u_3} Z_4
          \xrightarrow{u_4} T_{\widehat{f}_1;\Phi}.
      \]
      Let $u_3^{-1}$ be  a $G_{\Phi}$-homotopy inverse of $u_3$. Define the $G_{\Phi}$-homotopy equivalence
        \[
          u =  u_4 \circ u_3^{-1} \circ u_2 \circ u_1 \colon T_{\widehat{f}_0;\Phi}  \to T_{\widehat{f}_1;\Phi}.
          \]
          The equivariant Whitehead  torsion $\tau^{G_{\Phi}}(u)$ vanishes by the following  computation based
          on~\cite[Theorem~4.8 on page~62]{Lueck(1989)}
          \begin{eqnarray*}
            \tau^{G_{\Phi}}(u)
            & = &
            \tau^{G_{\Phi}}(u_1) + \tau^{G_{\Phi}}(u_2) + \tau^{G_{\Phi}}(u_3^{-1})  + \tau^{G_{\Phi}}(u_4)
            \\
            & = &
            \tau^{G_{\Phi}}(u_1) + \tau^{G_{\Phi}}(u_2) - \tau^{G_{\Phi}}(u_3)  + \tau^{G_{\Phi}}(u_4)
            \\
            & = &
             \tau^{G_{\Phi}}(s) + 0  - 0 +  (\tau^{G_{\Phi}}(s \times \id_{[0,1]}) - \tau^{G_{\Phi}}(s \times \id_{\{0,1\}}))
            \\
            & = &
            \tau^{G_{\Phi}}(s) + 0  - 0 + (\tau^{G_{\Phi}}(s) -2 \cdot  \tau^{G_{\Phi}}(s))
            \\
            & = &
             0.
          \end{eqnarray*}      
          Hence  $T_{\widehat{f}_0;\Phi}$ and $T_{\widehat{f}_1;\Phi}$ are simple $G_{\Phi}$-homotopy equivalent. Now
           Lemma~\ref{lem:checking_Det-finite_for_Phi-automorphism} follows from
        Theorem~\ref{the:Basic_properties_of_L_upper_2-torsion}~%
\ref{the:Basic_properties_of_L_upper_2-torsion:G-homotopy_invariance}.
      \end{proof}
      
\begin{theorem}[Main properties of the $L^2$-torsion of a self homotopy equivalence]%
\label{the:properties_of_L2-torsion_of_self_homo}
    Let $\Phi \colon G \xrightarrow{\cong} G$ be a group automorphism.

    \begin{enumerate} 
    \item\label{the:properties_of_L2-torsion_of_self_homo:homotopy_invariance}
      \emph{Equivariant homotopy invariance}\\[1mm]
      Let $f  \colon Y \to Y$ and $f'  \colon Y' \to Y'$ be  weak $\Phi$-homotopy equivalences
      of $\FIN$-finite $G$-spaces $Y$ and $Y'$ and $u \colon Y \to Y'$ be a weak $G$-homotopy
      equivalence such that $u \circ f$ and $f' \circ u$ are $\Phi$-homotopy equivalent.
      Suppose  that $(f;\Phi)$ or $(f';\Phi)$ is $\det$-finite.
      
      Then both $(f;\Phi)$ and  $(f';\Phi)$ are $\det$-finite and we get
      \[
      \rho^{(2)}(f;\Phi)  = \rho^{(2)}(f';\Phi).
    \]
    In particular the property $\det$-finite and the number $\rho^{(2)}(f;\Phi)$ depends
    only on the $\Phi$-homotopy class of $f$;
     
     \item\label{the:properties_of_L2-torsion_of_self_homo:trace_formula} \emph{Trace formula}\\[1mm]
       Let $\Psi  \colon G \xrightarrow{\cong} G'$ and $\Psi' \colon G' \to G$ be group isomorphisms.
       Consider a $\FIN$-finite $G$-space $Y$ and  a $\FIN$-finite $G'$-space $Y'$.
       
       Let $f \colon Y \to Y'$ be a weak $\Psi$-homotopy equivalence and
       $f' \colon Y' \to Y$ be a weak $\Psi'$-homotopy equivalence. Suppose that
       $(f' \circ f;\Psi' \circ \Psi)$ or $(f \circ f';\Psi \circ \Psi')$ is
       $\det$-finite.

       Then both $(f' \circ f;\Psi' \circ \Psi)$ and  $(f \circ f';\Psi \circ \Psi')$ are $\det$-finite and we get
       \[
        \rho^{(2)}(f' \circ f;\Psi' \circ \Psi) = \rho^{(2)}(f \circ f';\Psi \circ \Psi');
       \]

     \item\label{the:properties_of_L2-torsion_of_self_homo:trace_formula:multiplicativity}
       \emph{Multiplicativity}\\[1mm]
        Let $f  \colon Y \to Y$  be weak $\Phi$-homotopy equivalence
        for the $\FIN$-finite $G$-space $Y$. Consider  $n \in \IZ^{\ge 1}$.
        Suppose that $(f^n;\Phi^n)$ or $(f;\Phi)$ is $\det$-finite.

        Then  both $(f^n;\Phi^n)$ and  $(f;\Phi)$ are $\det$-finite and we get
       \[
        \rho^{(2)}(f^n;\Phi^n) = n \cdot \rho^{(2)}(f;\Phi);
      \]

    \item\label{the:properties_of_L2-torsion_of_self_homo:restriction}\emph{Restriction}\\[1mm]
      Let $f \colon Y \to Y$ be weak $\Phi$-homotopy equivalence for the $\FIN$-finite
      $G$-space $Y$. Let $H \subseteq G$ be a subgroup of $G$ of finite index $[G:H]$
      satisfying $\Phi(H) = H$.

      Then the restriction $Y|_H$ of $Y$ to an $H$-space is $\FIN$-finite.
      If we additionally assume that $(f|_H; \Phi|_H)$ or $(f;\Phi)$ is $\det$-finite, then both
      $(f|_H; \Phi|_H)$ and $(f;\Phi)$ are $\det$-finite and we get
       \[
         \rho^{(2)}(f|_H; \Phi|_H) = [G:H] \cdot \rho^{(2)}(f;\Phi);
       \]

       \item\label{the:properties_of_L2-torsion_of_self_homo:induction}\emph{Induction}\\[1mm]
         Let $H \subseteq G$ be a subgroup satisfying $\Phi(H) = H$. Let $f  \colon Y \to Y$  be a
         weak $\Phi|_H \colon H \to H$-homotopy equivalence
         for the $\FIN$-finite $H $-space $Y$.

         Then the  $G \times_HY$ is a $\FIN$-finite $G$-space and we get a weak $\Phi$-homotopy
         equivalence $F\colon G \times_H Y \to G \times _HY$ by sending $(g,x)$ to $(\Phi(g),f(x))$.
         If we  additionally assume  that $(f;\Phi|_H)$  or $(F;\Phi)$ is $\det$-finite, then both 
         $(f;\Phi|_H)$   and $(F;\Phi)$ are $\det$-finite  and we have
         \[
         \rho^{(2)}(F;\Phi) = \rho^{(2)}(f;\Phi|_H);
       \]

     \item\label{the:properties_of_L2-torsion_of_self_homo:finite_quotient} \emph{Finite quotients}\\[1mm]
       Let $1 \to K \to G \xrightarrow{p} Q \to 1$ be an extension of groups with finite
       $K$.  Let $\widetilde{\Phi} \colon G \xrightarrow{\cong} G$ and
       $\Phi \colon Q \xrightarrow{\cong} Q$ be group automorphisms satisfying
       $p \circ \widetilde{\Phi} = \Phi \circ p$. Let $f \colon Y \to Y$ be a weak
       $\Phi$-homotopy equivalence for the $\FIN$-finite $Q$-space $Y$. Let $p^*X$ be the
       $G$-space obtained from $X$ by restriction with $p$. Then $p^*Y$ is $\FIN$-finite
       and $p^*f \colon p^*Y \to p^*Y$ is a weak $\widetilde{\Phi}$-homotopy equivalence.

       If $(f,\Phi)$ or $(p^*f, \widetilde{\Phi})$ is $\det$-finite, then both $(f,\Phi)$
       and $(p^*f, \widetilde{\Phi})$ are $\det$-finite and we get

    \[
     \rho^{(2)}(p^*f,\widetilde{\Phi} ) = \frac{\rho^{(2)}(f;\Phi)}{|K|};
   \]

     \item\label{the:properties_of_L2-torsion_of_self_homo:composite_with_l_g}
       \emph{Inner automorphisms}\\[1mm]
       Let $f  \colon Y \to Y$  be weak $\Phi$-homotopy equivalence
       for the $\FIN$-finite $G$-space $Y$. Consider $g \in G$. Let
       $l_g \colon Y \to Y$ be the map given by multiplication with $g$ and
       $c_g \colon G \to G$ be the inner automorphism sending to $g'$ to $gg'g^{-1}$.
       Suppose  that $(l_g \circ f;c_g \circ \Phi)$ or $(f;\Phi)$ is $\det$-finite.

       Then both $(l_g \circ f;c_g \circ \Phi)$ and  $(f;\Phi)$ are $\det$-finite and we get 
       \[
         \rho^{(2)}(l_g \circ f;c_g \circ \Phi) = \rho^{(2)}(f;\Phi);
        \]

        \item\label{the:properties_of_L2-torsion_of_self_homo:Sum_formula} 
          \emph{Sum formula}\\[1mm]
          Consider the $G$-pushout
          \[
            \xymatrix{Y_0 \ar[r]^{j_2} \ar[d]_{j_1}
              & Y_2 \ar[d]
              \\
              Y_1 \ar[r]
              &
              Y
            }
          \]
          where $j_1$ is a $G$-cofibration,
          and the commutative diagram
          \[\xymatrix{Y_1 \ar[d]^{f_1}
              &
              Y_0 \ar[d]^{f_0}\ar[r]^{j_2} \ar[l]_{j_1}
              &
              Y_2 \ar[d]^{f_2}
              \\
              Y_1 
              &
              Y_0  \ar[r]^{j_2} \ar[l]_{j_1}
              &
              Y_2.
            }
          \]
          where the vertical arrows are weak $\Phi$-homotopy equivalences.  Let
          $f \colon Y \to Y$ be the map given by $f_0$, $f_1$, and $f_2$ and the
          $G$-pushout property.  Suppose that $Y_0$, $Y_1$, and $Y_2$ are $\FIN$-finite
          and that $(f_0;\Phi)$, $(f_1;\Phi)$, and $(f_2;\Phi)$ are $\det$-finite.

          Then $Y$ is $\FIN$-finite, $f \colon Y \to Y$ is a weak $\Phi$-homotopy equivalence, $(f;\Phi)$ is  $\det$-finite,
          and we get
          \[
            \rho^{(2)}(f;\Phi) = \rho^{(2)}(f_1;\Phi)  + \rho^{(2)}(f_2;\Phi)  - \rho^{(2)}(f_0;\Phi);
          \]
        \item\label{the:properties_of_L2-torsion_of_self_homo:Product_formula} \emph{Product formula}\\[1mm]
          Let $f \colon Y \to Y$ be a weak $\Phi$-homotopy equivalence for the $\FIN$-finite
          $G$-space $Y$ and let $H$ be a group. Let $Z$ be a $\FIN$-finite $H$-space.
          Denote by $\chi^{(2)}(Z;\caln(H))$ the $L^2$-Euler characteristic of $Z$. Suppose
          that $f$ is $\det$-finite.

          Then $f \times \id_Z \colon Y \times Z\to Y \times Z$ is a weak
          $\Phi \times \id_H \colon G \times H \to G \times H$-homotopy equivalence for
          the $\FIN$-finite $G \times H$-space $Y \times Z$, is $\det$-finite, and we get
          \[
           \rho^{(2)}(f \times \id_Z;\Phi \times \id_H)  = \chi^{(2)}(Z;\caln(H)) \cdot  \rho^{(2)}(f;\Phi);
         \]
       \item\label{the:properties_of_L2-torsion_of_self_homo:L2_acyclic_space}
         \emph{$L^2$-acyclic space}\\[1mm]
         Let $f \colon Y \to Y$ be weak $\Phi$-homotopy equivalence for the $\FIN$-finite
         $G$-space $Y$.  Suppose that $Y$ is $\det$-$L^2$-acyclic in the sense that there
         is a $\det$-$L^2$-acyclic finite proper $G$-$CW$-complex $X$ together with a weak
         homotopy equivalence $X \to Y$.

          Then $(f;\Phi)$ is $\det$-finite and we get 
          \[
             \rho^{(2)}(f;\Phi) = 0.
           \]
           
         \item\label{the:properties_of_L2-torsion_of_self_homo:periodic}
           \emph{Periodic self equivalence}\\[1mm]
           Let $f\colon Y\to Y$ be a weak $\Phi$-homotopy equivalence for a $G$-space $Y$.
           Suppose that there exists $n \in \IZ^{\ge1}$ such that $\Phi^n = \id_G$ and
           $f^n$ is $G$-homotopic to the identity $\id_Y$.

            Then $(f;\Phi)$ is $\det$-finite and we get
          \[
          \rho^{(2)}(f;\Phi)=0.
          \]
    \end{enumerate}
  \end{theorem}
  \begin{proof}~\ref{the:properties_of_L2-torsion_of_self_homo:homotopy_invariance} Choose a
cellular $G$-approximation $(X,a)$ of $Y$ for a finite $G$-$CW$-complex $X$ of
 determinant class and a $\Phi$-homotopy equivalence
$\widehat{f} \colon X \to X$ such that $a \circ \widehat{f}$ and $f \circ a$ are
$\Phi$-homotopic. Then $a' := u \circ a \colon X \to Y'$ is a cellular $G$-approximation
and $a' \circ \widehat{f}$ and $f' \circ a'$ are $G$-homotopic. Now we conclude
\[
\rho^{(2)}(f;\Phi)  = \rho^{(2)}(T_{\widehat{f};\Phi};\caln(G \rtimes_{\Phi} \IZ))  = \rho^{(2)}(f';\Phi).
\]
from Definition~\ref{def:L2-torsion_of_a_selfhomotopy_equivalence}.
\\[1mm]~\ref{the:properties_of_L2-torsion_of_self_homo:trace_formula}
Because of Definition~\ref{def:L2-torsion_of_a_selfhomotopy_equivalence}
it suffices to show for finite proper $G$-$CW$-complexes $X$ and $X'$ which are of determinant class,
a $\Psi$-homotopy equivalence $f \colon X \to X'$, and a 
$\Psi'$-homotopy equivalence $f' \colon X' \to X$ 
\begin{equation}
\rho^{(2)}(T_{f' \circ f, \Psi' \circ \Psi}; \caln(G \rtimes_{\Psi' \circ \Psi} \IZ)) =
 \rho^{(2)}(T_{f \circ f', \Psi\circ \Psi'}; \caln(G' \rtimes_{\Psi \circ \Psi'} \IZ)) 
\label{the:properties_of_L2-torsion_of_self_homo:homotopy_invariance:eq_1}
\end{equation}
holds. We obtain an isomorphism of groups
$\mu \colon G \rtimes_{\psi' \circ \psi} \IZ \xrightarrow{\cong} G' \rtimes_{\Psi \circ
  \Psi'} \IZ$ by sending $gt^n$ to $\psi(g)t^n$.  It is not hard to check that there is a
simple $\mu$-homotopy equivalence
$T_{f' \circ f, \Psi' \circ \Psi} \to T_{f \circ f', \Psi\circ \Psi'}$.
Now~\eqref{the:properties_of_L2-torsion_of_self_homo:homotopy_invariance:eq_1} follows
from Theorem~\ref{the:Basic_properties_of_L_upper_2-torsion}~%
\ref{the:Basic_properties_of_L_upper_2-torsion:G-homotopy_invariance}.
\\[1mm]~\ref{the:properties_of_L2-torsion_of_self_homo:trace_formula:multiplicativity} Let
$\pr \colon G \times_{\Phi} \IZ \to \IZ$ be the projection.  Then $\pr^{-1}(n \cdot \IZ)$
has index $n$ in $G \times_{\Phi} \IZ$ and can be identified with $G \times_{\Phi^n} \IZ$.
The restriction of $T_{\widehat{f};\Phi}$ to
$\pr^{-1}(n \cdot \IZ) = G \times_{\Phi^n} \IZ$ is simple $G \times_{\Phi^n} \IZ$-homotopy
equivalent to $T_{\widehat{f}^n;\Phi^n}$. Now the claim follows from
Theorem~\ref{the:Basic_properties_of_L_upper_2-torsion}~%
\ref{the:Basic_properties_of_L_upper_2-torsion:G-homotopy_invariance}
and~\ref{the:Basic_properties_of_L_upper_2-torsion:passage_to_subgroups_of_finite_index}.
\\[1mm]~\ref{the:properties_of_L2-torsion_of_self_homo:restriction} We can view
$H \rtimes_{\Phi_H} \IZ$ as a subgroup of finite index in $G \times_{\Phi} \IZ$.  The
restriction of the $G \times_{\Phi} \IZ$-space $T_{\widehat{f};\Phi}$ to
$G \times_{\Phi} \IZ$ can be identified with the $H \times_{\Phi|_H} \IZ$-space
$T_{\widehat{f}|_H;\Phi|_H}$, where $\widehat{f}|_H$ is the $\Phi|_H$-homotopy equivalence
obtained from $\widehat{f}$ by restriction.  Now the claim follows from
Theorem~\ref{the:Basic_properties_of_L_upper_2-torsion}~%
\ref{the:Basic_properties_of_L_upper_2-torsion:passage_to_subgroups_of_finite_index}.
\\[1mm]~\ref{the:properties_of_L2-torsion_of_self_homo:induction} We can view
$H \rtimes_{\Phi|_H} \IZ$ as a subgroup of $G \rtimes_{\Phi} \IZ$. The
$G \rtimes_{\Phi} \IZ$-space $T_{\widehat{F};\Phi}$ can be identified with
$(G \rtimes_{\Phi} \IZ) \times_{H \rtimes_{\Phi|_H} \IZ} T_{\widehat{f};\Phi|_H}$.  Now
the claim follows from Theorem~\ref{the:Basic_properties_of_L_upper_2-torsion}~%
\ref{the:Basic_properties_of_L_upper_2-torsion:induction}.
\\[1mm]~\ref{the:properties_of_L2-torsion_of_self_homo:finite_quotient} Note that we
obtain an exact sequence
$1 \to K \to G \rtimes_{\widetilde{\Phi}} \IZ \xrightarrow{\widehat{p}} Q \rtimes_{\Phi} \IZ \to 1$,
where $\widehat{p}$ sends $gt^n$ to $p(g)t^n$. Now assertion~\ref{the:properties_of_L2-torsion_of_self_homo:finite_quotient}
follows from Theorem~\ref{the:Basic_properties_of_L_upper_2-torsion}~%
\ref{the:Basic_properties_of_L_upper_2-torsion:finite_quotient} since 
$T_{p^*\widehat{f};\widetilde{\Phi}}$ is $\widehat{p}^*T_{\widehat{f};\widetilde{\Phi}}$.        
\\[1mm]~\ref{the:properties_of_L2-torsion_of_self_homo:composite_with_l_g} The map
$l_g \colon X \xrightarrow{\cong} X$ is a $c_g$-homeomorphism.  Let
$\mu \colon G \rtimes_{\Phi} \IZ \xrightarrow{\cong} G \rtimes_{c_g \circ \Phi} \IZ$ be
the group isomorphism sending $g't^n$ to $g'(g^{-1}t)^n$. Then we obtain a
cellular $\mu$-homeomorphism $T_{\widehat{f};\Phi} \xrightarrow{\cong} T_{l_g \circ \widehat{f};c_g \circ \Phi}$
since the following diagram commutes
\[
\xymatrix{(G\rtimes_{\Phi} \IZ)\times_G X \ar[r]^{q_{\widehat{f}}} \ar[d]_{\mu \times_{G} \id_X}
    &
    (G\rtimes_{\Phi} \IZ)\times_G X \ar[d]^{\mu \times_{G} \id_X}
    \\
    (G\rtimes_{\Phi} \IZ)\times_G X \ar[r]^{q_{l_g \circ \widehat{f}}} 
    &
   (G\rtimes_{\Phi} \IZ)\times_G X
  }
\]
where $q_{\widehat{f}}$ sends $(g't^n, x)$ to $(g't^{n-1}, \widehat{f}(x))$
and $q_{l_g \circ \widehat{f}}$ sends $(g't^n, x)$ to $(g't^{n-1}, l_g \circ \widehat{f}(x))$.
This implies
\begin{multline*}
      \rho^{(2)}(l_g \circ \widehat{f};c_g \circ \Phi)
      = \rho^{(2)}(T_{l_g \circ \widehat{f},c_g\circ \Phi} ;\caln(G \times_{c_g \circ \Phi} \IZ))
      \\
      = \rho^{(2)}(T_{\widehat{f};\Phi} ;\caln(G \times_{\Phi} \IZ))
      = \rho^{(2)}(\widehat{f};\Phi).
    \end{multline*}%
~\ref{the:properties_of_L2-torsion_of_self_homo:Sum_formula} We can asume because of
    assertion~\ref{the:properties_of_L2-torsion_of_self_homo:homotopy_invariance} without
    loss of generality that both inclusions $j_l \colon Y_0 \to Y_l$ for $l = 1,2$ are
    $G$-cofibrations, otherwise replace $j_2$ by the inclusion of $Y_0$ into the mapping
    cylinder of $j_2$.

    By assumption the $G$-space $Y_l$ has the weak $G$-homotopy type of a proper finite
    $G$-$CW$-complex for $l = 0,1,2$. Then there is a commutative diagram of $G$-spaces
\[
  \xymatrix{X_1 \ar[d]^{a_1}
    &
    X_0 \ar[l]_{i_1} \ar[d]^{a_0}  \ar[r]^{i_2} 
    & X_2 \ar[d]^{a_2}
    \\
    Y_1 
    &
    Y_0 \ar[l]_{j_1} \ar[r]^{j_2} 
    & Y_2
  }
\]
such that $i_l \colon X_0 \to X_l$ is an inclusion of proper finite $G$-$CW$-complexes
for $l = 1,2$, see~\cite[4.31 on page~76]{Lueck(1989)}. We obtain a commutative diagram of $G$-spaces
\[
  \xymatrix{X_1 \ar[d]^{b_1}
    &
    X_0 \ar[l]_{i_1} \ar[d]^{b_0}  \ar[r]^{i_2} 
    & X_2 \ar[d]^{b_2}
    \\
    \cyl(a_1) 
    &
    \cyl(a_0) \ar[l]_{k_1}  \ar[r]^{k_2} 
    & \cyl(a_2)
    \\
    Y_1  \ar[u]_{c_1}
    &
    Y_0 \ar[l]_{j_1} \ar[r]^{j_2}  \ar[u]_{c_0}
    & Y_2 \ar[u]_{c_0}
  }
\]
where the vertical maps are the canoncial inclusions into the mapping cylinders, and $k_l$
is the obvious map induced by $i_l$ and $j_l$ for $l = 0,1$. Note $a_l$ is a weak
$G$-homotopy equivalence and $b_l$ is are $G$-homotopy equivalence for $l = 0,1,2$. One
also obtains a commutative diagram of $G$-spaces
\[
  \xymatrix{
    \cyl(a_1)  \ar[d]^{p_1}
    &
    \cyl(a_0) \ar[l]_{k_1}  \ar[r]^{k_2}  \ar[d]^{p_0}
    & \cyl(a_2) \ar[d]^{p_2}
    \\
    Y_1 
    &
    Y_0 \ar[l]_{j_1} \ar[r]^{j_2}  
    & Y_2
  }
\]
where the vertical maps are the canonical projections. Note that $p_l$ is a $G$-homotopy
equivalence and $p_l \circ b_l = a_l$ holds for $l = 0,1,2$. For the proof we need the
following lemma.

\begin{lemma}\label{lem:construction_of_self_homotopy_equivalence}
  In the situation above we can construct $\Phi$-homotopy equivalences
  $\widehat{f_l} \colon X_l \to X_l$ and $v_l \colon \cyl(a_l) \to \cyl(a_l)$ for
  $l = 0,1,2$ such that
  \[\begin{array}{lcrll}
    \widehat{f_l} \circ i_l
    & = &
    i_l \circ\widehat{f_0}
    & & \text{for}\; l = 1,2;
    \\
    v_l \circ k_l
    & = &
    k_l \circ v_0
    & & \text{for}\; l = 1,2;
    \\
    v_l \circ b_l
    & = &
    b_l \circ \widehat{f_l}
      & &
     \text{for}\; l = 0,1,2;
     \\
    v_l \circ c_l
    & = &
    c_l \circ f_l
    & &
   \text{for}\; l = 0,1,2,
    \end{array}
  \]
  holds.
\end{lemma}
\begin{proof}
  Define $v_l \colon \cyl(a_l) \to \cyl(a_l)$ to be $c_l \circ f_l \circ p_l$ for $l = 0,1,2$.
  Then we get $v_l \circ c_l = c_l \circ f_l$ for $l = 0,1,2$ and $v_l \circ k_l= k_l \circ v_0$ for $l = 1,2$.
  
  We can find for $l = 0,1,2$ a $\Phi$-homotopy equivalence
  $\widehat{f}_l \colon X_l \to X_l$ such that $v_l \circ b_l$ and
  $b_l \circ \widehat{f}_l$ are $G$-homotopic by the Equivariant Whitehead Theorem,
  see~\cite[Theorem~2.4 on page~36]{Lueck(1989)}.  Note that then also $f_l \circ p_l$ and
  $p_l \circ \widehat{f}_l$ are $G$-homotopic for $l = 0,1,2$.  Since
  $b_0 \amalg c_0 \colon X_0 \amalg Y_0 \to \cyl(a_0)$ is a $G$-cofibration, we can change
  $v_0 \colon \cyl(a_0) \to \cyl(a_0)$ up to $\Phi $-homotopy such that
  $v_0 \circ b_0 = b_0 \circ \widehat{f}_0$ holds and we keep
  $v_0 \circ c_0 = c_0 \circ f_0$.  Fix $l \in \{1,2\}$. Since $i_l \colon X_0 \to X_l$ is
  a $G$-cofibration, we can change $\widehat{f}_l$ up to $G$-homotopy such that
  $\widehat{f}_l \circ i_l = i_l \circ \widehat{f}_0$ holds. Since the inclusion the
  subspace of $\cyl(a_l)$ given by $b_l(X_l) \cup c_l(Y_l) \cup k_l(\cyl(f_0))$ into
  $\cyl(f_l)$ is a cofibration, we can change $v_l$ up to $\Phi$-homotopy such that
  $v_l \circ k_l = k_l \circ v_0$ and $v_l \circ b_l = b_l \circ \widehat{f}_l$ hold and
  we keep $v_l \circ c_l = k_l \circ f_0$ and $v_l \circ k_l = k_l \circ v_0$.
\end{proof}
Consider the three $G$-pushouts
\[\xymatrix{X_0 \ar[r]^{i_1} \ar[d]_{i_2}
    & X_1 \ar[d]
    \\
    X_2 \ar[r]
    &
    X
  }
  \quad  \quad
\xymatrix{Y_0 \ar[r]^{j_1} \ar[d]_{j_2}
    & Y_1 \ar[d]
    \\
    Y_2 \ar[r]
    &
    Y
  }
   \quad \quad 
\xymatrix{\cyl(a_0) \ar[r]^{k_1} \ar[d]_{k_2}
    & \cyl(a_1) \ar[d]
    \\
    \cyl(a_2) \ar[r]
    &
    Z.
  }
\]
From the $\Phi$-homotopy equivalences $\widehat{f}_l$ for $l =0,1,2$, we obtain by the
$G$-pushout property a $\Phi$-homotopy equivalence $\widehat{f} \colon X \to X$.  From the
$\Phi$-homotopy equivalences $v_l$ for $l =0,1,2$ we obtain by the $G$-pushout property a
$\Phi$-homotopy equivalence $v \colon Z \to Z$. From the $\Phi$-homotopy equivalences
$f_l$ for $l =0,1,2$ we obtain by the $G$-pushout property a $\Phi$-homotopy equivalence
$f \colon Y \to Y$. From the weak $G$-homotopy equivalences $b_l$ for $l =0,1,2$ we obtain
by the $G$-pushout property a weak $G$-homotopy equivalence $b \colon X \to Z$. From the
$G$-homotopy equivalences $c_l$ for $l =0,1,2$ we obtain by the $G$-pushout property a
$G$-homotopy equivalence $c\colon X \to Z$. One easily checks that following diagram
comutes
\[\xymatrix{X \ar[r]^{\widehat{f}} \ar[d]^b
& X \ar[d]^b
\\
Z \ar[r]^{v}
&Z
\\
Y \ar[r]^{f} \ar[u]_{c}
&
Y  \ar[u]_{c}.
}
\]
Hence the following diagram for $l = 0,1,2$
\[
\xymatrix{X_l \ar[r]^{\widehat{f}_l} \ar[d]_{a_l}
  & X_l \ar[d]^{a_l}
  \\
  Y_l \ar[r]_{f_l}
  &
  Y_l
}
\]
and the diagram
\[
\xymatrix{X \ar[r]^{\widehat{f}} \ar[d]_{a}
  & X \ar[d]^{a}
  \\
  Y \ar[r]_{f}
  &
  Y
}
\]
commute up to $G$-homotopy and have weak $G$-homotopy equivalences as vertical arrows, and
$\Phi$-homotopy equivalences as horizontal arrows. Since $X_l$ for $l =0,1,2$ is a proper
finite $G$-$CW$-complexes of determinant class, $X$ is a proper finite $G$-$CW$-complexes
of determinant class by
Theorem~\ref{the:properties_of_L2-torsion_of_self_homo}~\ref{the:properties_of_L2-torsion_of_self_homo:Sum_formula}.
We get
\begin{eqnarray*}
  \rho^{(2)}(f_l;\Phi)
  & = &
  \rho^{(2)}(T_{\widehat{f_l};\Phi}; \caln(G \rtimes_{\Phi} \IZ)) \quad \text{for}\; l = 0,1,2;
  \\
  \rho^{(2)}(f;\Phi)
  & = &
  \rho^{(2)}(T_{\widehat{f};\Phi}; \caln(G \rtimes_{\Phi} \IZ)).
\end{eqnarray*}      
from the definitions. We obtain a $G$-pushout  of finite proper $G \rtimes_{\Phi} \IZ$-$CW$-complexes
\[\xymatrix{T_{\widehat{f}_0;\Phi} \ar[r] \ar[d]
    &
    T_{\widehat{f}_1;\Phi} \ar[d]
    \\
    T_{\widehat{f}_2;\Phi} \ar[r]
    &
    T_{\widehat{f};\Phi}
  }
\]
where all maps are inclusions of $G \rtimes_{\Phi} \IZ$-$CW$-complexes.
Theorem~\ref{the:Basic_properties_of_L_upper_2-torsion}~%
\ref{the:Basic_properties_of_L_upper_2-torsion:Sum_formula}
implies
\begin{multline*}
  \rho^{(2)}(T_{\widehat{f};\Phi};\caln(G \rtimes_{\Phi} \IZ))
  \\
  =
  \rho^{(2)}(T_{\widehat{f_1};\Phi};\caln(G \rtimes_{\Phi} \IZ)) +
  \rho^{(2)}(T_{\widehat{f_2};\Phi};\caln(G \rtimes_{\Phi} \IZ)) -
  \rho^{(2)}(T_{\widehat{f_0};\Phi};\caln(G \rtimes_{\Phi} \IZ)).
\end{multline*}
This finishes the proof of
assertion~\ref{the:properties_of_L2-torsion_of_self_homo:Sum_formula}.
\\[1mm]~\ref{the:properties_of_L2-torsion_of_self_homo:Product_formula}.
Obviously we can assume without loss of generality that $Z$ itself is a finite proper $H$-$CW$-complex.
Moreover we can replace $\chi^{(2)}(Z;\caln(H))$ by the orbifold Euler characteristic
\[\chi_{\operatorname{orb}}(Z) = \sum_{e} (-1)^{\dim(e)} \cdot \frac{1}{|H_e|},
\]
where $e$ runs through the equivariant cells of $Z$, because of
$\chi^{(2)}(Z;\caln(H)) = \chi_{\operatorname{orb}}(Z)$, see~\cite[Subsection~6.6.1]{Lueck(2002)}. Using
assertions~\ref{the:properties_of_L2-torsion_of_self_homo:homotopy_invariance}
and~\ref{the:properties_of_L2-torsion_of_self_homo:Sum_formula} one can reduce the claim
to the special case $Z = H/L$ for any finite subgroup $L \subseteq H$ by induction over
the dimension of $Z$ and subinduction over the number of top-dimensional equivalent
cells of $Z$. We get an isomorphism of proper finite $G \times H$-$CW$-complexes
$G \times H \times_{G \times L} \pr_L^* Y \xrightarrow{\cong} Y \times H/L$ by sending
$((g,h),y)$ to $(gy,hL)$, where 
$\pr_L^* Y$ is the $G \times L$-$CW$-space obtained from $Y$ by restriction with the projection 
$\pr_L\colon  G \times L \to G$. Under this identification the $\Phi \times \id_H$-map $f \times \id_{H/L}$
becomes the induction from $G \times L$ to $G \times H$ of the $\Phi \times \id_L$-map
$\pr_L^* f \colon \pr_L^*Y \to \pr_L^*Y$.  We conclude from
assertion~\ref{the:properties_of_L2-torsion_of_self_homo:induction} that suffices to show
the claim for the $\Phi \times \id_L$-map $\pr_L^* f \colon \pr_L^*Y \to \pr_L^*Y$.  This
follows from assertion~\ref{the:properties_of_L2-torsion_of_self_homo:restriction} applied
to $G \subseteq G \times L$.
\\[1mm]~\ref{the:properties_of_L2-torsion_of_self_homo:L2_acyclic_space} This follows from
Theorem~\ref{the:Basic_properties_of_L_upper_2-torsion}~%
\ref{the:Basic_properties_of_L_upper_2-torsion:Sum_formula} and
assertion~\ref{the:properties_of_L2-torsion_of_self_homo:Product_formula} applied to the
$G \rtimes_{\Phi} \IZ$-pushout~\eqref{def_T_upper_Phi_widehat(f)}.
\\[1mm]~\ref{the:properties_of_L2-torsion_of_self_homo:periodic}
Lemma~\ref{lem_T_upper_PhI_u_det_L2-acyclic}~\ref{lem_T_upper_PhI_u_det_L2-acyclic:peridoc_implies_det_L2-acyclic}
implies that $(f,\Phi)$ is $\det$-finite.  Now $\rho^{(2)}(f;\Phi) = 0$ follows from
assertions~\ref{the:properties_of_L2-torsion_of_self_homo:homotopy_invariance}
and~\ref{the:properties_of_L2-torsion_of_self_homo:trace_formula:multiplicativity}.

This finishes the proof
of Theorem~\ref{the:properties_of_L2-torsion_of_self_homo}.
\end{proof}
         
%%%%%%%%%%%%%%%%%%%%%%%%%%%%%%%%%%%%%%%%%%%%%%%%%%%%%%%%%%%%%%%%%%%%%
%%%%%%%%%%%%%%%%%%%%%%%%%%%%%% Section 5 %%%%%%%%%%%%%%%%%%%%%%%%%%%%%%%%
%%%%%%%%%%%%%%%%%%%%%%%%%%%%%%%%%%%%%%%%%%%%%%%%%%%%%%%%%%%%%%%%%%%%%

\typeout{---------- Section 5:  $L^2$-torsion of an automorphism of an admissible group   ---------------}

\section{\texorpdfstring{$L^2$}{L2}-torsion of a group  automorphism}%
\label{sec:L2-torsion_of_an_automorphism_of_a_det-finite_group}

Consider a group automorphism $\Phi \colon G \xrightarrow{\cong} G$. Then there is  up to
$\Phi$-homotopy precisely one $\Phi$-homotopy equivalence $f_{\Phi} \colon \eub{G} \to \eub{G}$.
Now suppose that $G$ is $\FIN$-finite. In Definition~\ref{def:det-finite_Phi-self-homotopy_equivalence}
we have defined when we call $f_{\Phi}$ to be $\det$-finite.
This this notion depends only on the $\Phi$-homotopy type of $f_{\Phi}$
by Theorem~\ref{the:properties_of_L2-torsion_of_self_homo}~%
\ref{the:properties_of_L2-torsion_of_self_homo:homotopy_invariance},
Hence the following definition makes sense, i.e., is independent of the choice of $f_{\Phi}$.

\begin{definition}[Det-finite group automorphism]\label{def:Det-finite_group_automorphism}
  A group automorphism $\Phi \colon G \xrightarrow{\cong} G$ of a $\FIN$-finite group $G$
  is called \emph{$\det$-finite}, if $f_{\Phi}$ is $\det$-finite for one (and hence every)
  choice of a $\Phi$-homotopy equivalence
  $f_{\Phi} \colon \eub{G} \to \eub{G}$.
\end{definition}

Again by Theorem~\ref{the:properties_of_L2-torsion_of_self_homo}~%
\ref{the:properties_of_L2-torsion_of_self_homo:homotopy_invariance} the following definition makes sense,
i.e., is independent of the choice of $f_{\Phi}$.

\begin{definition}[$L^2$-torsion of a group automorphism]\label{def:L2_torsion_of_a_group_automorphism}
  Let $\Phi \colon G \xrightarrow{\cong} G$ be an automorphism of the  $\FIN$-finite group $G$.
  Suppose that $\Phi$ is $\det$-finite.

  Then we define its $L^2$-torsion
  \[\rho^{(2)}(\Phi) := \rho^{(2)}(f_{\Phi};\Phi)  \in \IR
    \]
    for any choice of a $\Phi$-homotopy equivalence $f_{\Phi} \colon \eub{G} \to \eub{G}$,
    where $\rho^{(2)}(f_{\Phi};\Phi)$ has been introduced in Definition~\ref{def:L2-torsion_of_a_selfhomotopy_equivalence}.
  \end{definition}

  Next we collect the basic properties of this invariant.

  \begin{theorem}[Main properties of the $L^2$-torsion of a group automorphism]%
\label{the:elementary_properties_of_L2-torsion_of_autos}\

\begin{enumerate}

\item\label{the:elementary_properties_of_L2-torsion_of_autos:trace} \emph{Trace formula}\\[1mm]
      Let $\Psi  \colon G \xrightarrow{\cong} G'$ and $\Psi' \colon G' \to G$ be group isomorphisms.
      Suppose that $\Psi' \circ \Psi$ or $\Psi \circ \Psi'$ is $\det$-finite. Then both
      $\Psi' \circ \Psi$ and  $\Psi \circ \Psi'$ are $\det$-finite and we get
       \[
        \rho^{(2)}(\Psi' \circ \Psi) = \rho^{(2)}(\Psi \circ \Psi');
       \]

     \item\label{the:elementary_properties_of_L2-torsion_of_autos:multiplicativity}
       \emph{Multiplicativity}\\[1mm]
       Let $\Phi \colon G \xrightarrow{\cong} G$ a group automorphism of 
       the $\FIN$-finite group $G$. Consider $n \in \IZ^{\ge 1}$.
       Suppose  that $\Phi$ or $\Phi^n$ is $\det$-finite.

       Then both $\Phi$ and $\Phi^n$ are  $\det$-finite and we get
       \[
        \rho^{(2)}(\Phi^n) = n \cdot \rho^{(2)}(\Phi);
      \]

      \item\label{the:elementary_properties_of_L2-torsion_of_autos:periodic}
        \emph{Periodic automorphism}\\[1mm]
        Let $\Phi \colon G \xrightarrow{\cong} G$ a group automorphism of 
        the $\FIN$-finite group $G$. Consider $n \in \IZ^{\ge 1}$. Suppose that
        $\Phi^n = \id_G$.

        Then $\Phi$ is $\det$-finite and we get
        \[
          \rho^{(2)}(\Phi) = 0;
        \]

        \item\label{the:elementary_properties_of_L2-torsion_of_autos:Product_formula} \emph{Product formula}\\[1mm]
          Let $\Phi \colon G \xrightarrow{\cong} G$ a group automorphism of 
          the $\FIN$-finite group $G$. Suppose  that $\Phi$ is $\det$-finite. Let $H$ be a $\FIN$-finite group.

          Then $G \times H$ is a $\FIN$-finite group, the group automorphism $\Phi \times \id_H\colon G \times H\xrightarrow{\cong} G \times H$
          is $\det$-finite, and we get
          \[
           \rho^{(2)}(\Phi \times \id_H)  = \chi^{(2)}(H) \cdot  \rho^{(2)}(\Phi);
         \]

    \item\label{the:elementary_properties_of_L2-torsion_of_autos:restriction}
    \emph{Restriction}\\[1mm]
        Let $H \subseteq G$ be a subgroup of $G$ of finite index $[G:H]$.
        Let $\Phi \colon G \xrightarrow{\cong}  G$ be a group automorphism satisfying $\Phi(H) = H$.
        Suppose that one of the following two  conditions holds:

        \begin{enumerate}
        \item $G$ is $\FIN$-finite and $\Phi$ is $\det$-finite;
        \item $H$ is $\FIN$-finite and $\Phi|_H$ is $\det$-finite.
        \end{enumerate}

        Then both conditions are satisfied and we get
               \[
         \rho^{(2)}(\Phi|_H) = [G:H] \cdot \rho^{(2)}(\Phi);
       \]

       \item\label{the:elementary_properties_of_L2-torsion_of_autos:finite_quotient} \emph{Finite quotients}\\[1mm]
       Let $1 \to K \to G \xrightarrow{p} Q \to 1$ be an extension of groups with finite
       $K$.  Let $\widetilde{\Phi} \colon G \xrightarrow{\cong} G$ and
       $\Phi \colon Q \xrightarrow{\cong} Q$ be group automorphisms satisfying
       $p \circ \widetilde{\Phi} = \Phi \circ p$. 
       If $\Phi$ or $\widetilde{\Phi}$ is $\det$-finite, then both $\Phi$
       and $\widetilde{\Phi}$ are $\det$-finite and we get

    \[
     \rho^{(2)}(\widetilde{\Phi} ) = \frac{\rho^{(2)}(\Phi)}{|K|};
   \]

      \item\label{the:elementary_properties_of_L2-torsion_of_autos:conjugation_invariance}
  \emph{Conjugation invariance}\\[1mm]
  Let $\Phi \colon G \xrightarrow{\cong} G$ a group automorphism of 
  the $\FIN$-finite group $G$. Consider $g \in G$. Let $c_g \colon G \to G$ be associated the inner automorphisms
  sending $g'$ to $gg'g^{-1}$.  Suppose that $\Phi$ or $c_g \circ \Phi$ is $\det$-finite.

  Then  both  $\Phi$ and $c_g \circ \Phi$ are $\det$-finite and we get
    \item\label{the:elementary_properties_of_L2-torsion_of_autos:L2-acyclic_group} \emph{$L^2$-acyclic group}\\[1mm]
      Let $\Phi \colon G \xrightarrow{\cong}  G$ be an automorphism of the $\FIN$-finite group $G$.
      Suppose that one and (hence every) finite $G$-$CW$-model for $\eub{G}$ is $\det$-$L^2$-acyclic. Then
      $\Phi$ is $\det$-finite and we get
      \[
      \rho^{(2)}(f) = 0.
      \]
    \end{enumerate}
  \end{theorem}
  \begin{proof}~\ref{the:elementary_properties_of_L2-torsion_of_autos:trace} 
    This follows from  This follows from Theorem~\ref{the:properties_of_L2-torsion_of_self_homo}~%
\ref{the:properties_of_L2-torsion_of_self_homo:trace_formula}
    \\[1mm]~\ref{the:elementary_properties_of_L2-torsion_of_autos:multiplicativity}
   This follows from Theorem~\ref{the:properties_of_L2-torsion_of_self_homo}~%
\ref{the:properties_of_L2-torsion_of_self_homo:trace_formula:multiplicativity}.
   \\[1mm]~\ref{the:elementary_properties_of_L2-torsion_of_autos:periodic}
   This follows from Theorem~\ref{the:properties_of_L2-torsion_of_self_homo}~%
\ref{the:properties_of_L2-torsion_of_self_homo:periodic}.
   \\[1mm]~\ref{the:elementary_properties_of_L2-torsion_of_autos:Product_formula} 
    This follows from  Theorem~\ref{the:properties_of_L2-torsion_of_self_homo}~%
\ref{the:properties_of_L2-torsion_of_self_homo:Product_formula}.
   \\[1mm]~\ref{the:elementary_properties_of_L2-torsion_of_autos:restriction}
   This follows from Theorem~\ref{the:properties_of_L2-torsion_of_self_homo}~%
\ref{the:properties_of_L2-torsion_of_self_homo:restriction}.
  \\[1mm]~\ref{the:elementary_properties_of_L2-torsion_of_autos:finite_quotient}
 This follows from Theorem~\ref{the:properties_of_L2-torsion_of_self_homo}~%
\ref{the:properties_of_L2-torsion_of_self_homo:finite_quotient}.
  \\[1mm]~\ref{the:elementary_properties_of_L2-torsion_of_autos:conjugation_invariance}
    This follows from Theorem~\ref{the:properties_of_L2-torsion_of_self_homo}~%
\ref{the:properties_of_L2-torsion_of_self_homo:composite_with_l_g}.
    \\[1mm]~\ref{the:elementary_properties_of_L2-torsion_of_autos:L2-acyclic_group}
    This follows from Theorem~\ref{the:properties_of_L2-torsion_of_self_homo}~%
\ref{the:properties_of_L2-torsion_of_self_homo:L2_acyclic_space} 
 \end{proof}

 \begin{remark}[No composition formula]\label{rem:no_compostion_formula}
   We mention that there is \emph{no} composition formula for the $L^2$-torsion of group
   automorphisms.  In other words, the formula
   $\rho^{(2)} (\Phi \circ \Psi) = \rho^{(2)}(\Phi) +\rho^{(2)}(\Psi)$ is \emph{not} true
   in general for two $\det$-finite group automorphisms $\Phi$ and $\Psi$ of the same
   $\FIN$-group $G$.  For example for a pseudo-Anosov diffeomorphism $f$ of a closed
   surface the mapping tori $T_f$ and $T_{f^{-1}}$ are homeomorphic and hence have the
   same non-zero $L^2$-torsion, but the $L^2$-torsion of $T_{f \circ f^{-1}} = T_{\id}$
   vanishes.
 \end{remark}

Let $S$ be a compact  connected orientable $2$-dimensional manifold, possibly with boundary. Let
$f\colon  S \to S$ be an orientation preserving homeomorphism. The mapping torus
$T_f$ is a compact  connected orientable
$3$-manifold whose boundary is empty or a disjoint
union of $2$-dimensional tori. 
By the Nielson--Thurston decomposition for mapping classes and Thurston's hyperbolisation theorem
there is a maximal family of embedded  incompressible tori,
which are pairwise not isotopic and not boundary parallel, such
that it decomposes $T_f$ into pieces, which are Seifert or hyperbolic.
Let $M_1$, $M_2$, $\ldots$,  $M_r$ be the hyperbolic pieces.
They all have finite volume $\vol(M_i)$.

 Choose a base point $x \in S$ and a path $w \colon I \to S$ from $x$ to $f(x)$.
  Let $t_w \colon \pi_1(S,f(x)) \xrightarrow{\cong} \pi_1(S,x)$ be the isomorphism sending the class $[v]$
  of a loop $v$ in $S$ at $f(x)$ to the class $[w \ast v \ast w^-]$ of the loop $w \ast v \ast w^-$ at $x$
  given by concatenation of paths, where $w^-$ is the inverse of $w$ given by $w^-(t) = w(1-t)$.
  Let $\Phi \colon \pi_1(S,x) \xrightarrow{\cong} \pi_1(S,x)$ be the automorphism given by the composite
  $\pi_1(S,x) \xrightarrow{\pi_1(f,x)} \pi_1(S,f(x)) \xrightarrow{t_w} \pi_1(S,s)$.
  Then $\pi_1(S,x)$ is $\det$-finite and the real number $\rho^{(2)}(\Phi)$ is defined.
  Theorem~\ref{the:elementary_properties_of_L2-torsion_of_autos}~%
\ref{the:elementary_properties_of_L2-torsion_of_autos:trace} 
      and~\ref{the:elementary_properties_of_L2-torsion_of_autos:conjugation_invariance} imply
      that $\rho^{(2)}(\Phi)$ is independent of the choices of $x \in S$ and $w$. 

 \begin{theorem}[Surface automorphisms]\label{the:rho(surface_homeomorphism)}
   In the situation described above, the following statement true:

   \begin{enumerate}
\item\label{the:rho(surface_homeomorphism):special}
 If $S$ is $S^2$, $D^2$, or $T^2$,
 then $\rho^{(2)}(\Phi) = 0$;

 \item\label{the:rho(surface_homeomorphism):general}
If  $S$ is not $S^2$, $D^2$, or $T^2$, then
\[
\rho^{(2)}(\Phi)
=  \frac{-1}{6\pi} \cdot \sum_{i=1}^r \vol(M_i).
\]
\end{enumerate}
\end{theorem}
\begin{proof}
This follows from~\cite[Theorem~7.28 on page~307]{Lueck(2002)}. 
\end{proof}

Let $\Phi \colon G \to G$ be an automorphism of the $\FIN$-finite group $G$.  Let $X$
be a finite $G$-$CW$-complex such that for any finite subgroup $H \subseteq G$ the
$H$-fixed point set $X^H$ is contractible. Let 
$a \colon X \to X$ be  a $\Phi$-homeomorphism such that the following diagram commutes
\[\xymatrix{X \ar[rr]^{a} \ar[dr]_{\pr} & & X \ar[ld]^{\pr}
    \\
    & X/G & }
\]
where $\pr$ is the canonical projection.  Suppose that the isotropy group
$G_x$ for each $x \in X$ is $\FIN$-finite.
    
For $\overline{c} \in \pi_0(X_n \setminus X_{n-1})/G$ choose an element
$c \in \pi_0(X_n \setminus X_{n-1})$ representing $\overline{c}$, an element $x \in c$,
and an element $g \in G$ satisfying $a(x) = g x$. Note that the element $g$ exists
because of the assumption $\pr \circ a = \pr$. The isotropy group of $a(x) \in X$ agrees
with the isotropy group $G_{gx} = gG_xg^{-1} $ of $gx \in X$ and is given by
    \begin{multline*}
      \{g' \in G \mid g'a(x) = a(x)\} = \{g' \in G \mid a(\Phi^{-1}(g')x)= a(x)\}
      \\
      = \{g' \in G \mid \Phi^{-1}(g')x= x\} =  \{g' \in G \mid \Phi^{-1}(g') \in G_x\}  = \Phi(G_x).
    \end{multline*}
    Hence we get $gG_xg^{-1} = \Phi(G_x)$. Therefore we can define an automorphism of the
    $\FIN$-finite group $G_x$
    \[\Phi_{c,x,g} \colon G_x \to G_x, \quad g' \mapsto g^{-1}\Phi(g')g.
    \]
    Now suppose that $\Phi_{c,x,g}$ is $\det$-finite. 
    Then  the real number $\rho^{(2)}(\Phi_{c,x,g};\caln(G_x))$ is defined.
    The condition that $\Phi_{c,x,g}$ is $\det$-finite and
    the real number $\rho^{(2)}(\Phi_{c,x,g};\caln(G_x))$
    depend only on $\overline{c}$
    and are  independent of the choices of $c$, $x$, and $g$ by
    Theorem~\ref{the:elementary_properties_of_L2-torsion_of_autos}~%
\ref{the:elementary_properties_of_L2-torsion_of_autos:trace}
and~\ref{the:elementary_properties_of_L2-torsion_of_autos:conjugation_invariance}.

Consider $\overline{c} \in \pi_0(X_n \setminus X_{n-1})/G$. Then we say that
\emph{$\Phi_{\overline{c}}$ is $\det$-finite} and in this case can define 

        \begin{equation}
    \rho^{(2)}(\Phi_{\overline{c}}) \in \IR
    \label{rho_upper_(2)(Phi_overline(c)}
  \end{equation}
  by requiring that $\Phi_{c,x,g}$ is $\det$-finite
and  putting $\rho^{(2)}(\Phi_{\overline{c}}) = \rho^{(2)}(\Phi_{c,x,g}; \caln(G_x))$
  for one (and hence every) choice of $c$, $x$, and $g$.
    
   \begin{theorem}\label{the:application_to_automorphism}
     Let $\Phi \colon G \to G$ be an automorphism of the $\FIN$-finite group $G$.  Let
     $X$ be a finite $G$-$CW$-complex such that for any finite subgroup $H \subseteq G$
     the $H$-fixed point set $X^H$ is contractible and there is a cellular
     $\Phi$-homeomorphism $a \colon X \to X$ satisfying $\pr \circ a = \pr$ for the projection $\pr \colon X \to X/G$.
     Suppose that the isotropy group $G_x$ for each $x \in X$ is $\FIN$-finite
     and  that the automorphisms $\Phi_{\overline{c}}$ is $\det$-finite
     for every $\overline{c} \in \pi_0(X_n \setminus X_{n-1})/G$
     and $n \in \IZ^{\ge 0}$.

    Then the group $G$ is $\FIN$-finite,  $\Phi_{\overline{c}}$ is $\det$-finite,  and we get
  \[
    \rho^{(2)}(\Phi) = \sum_{n \ge 0} \; (-1)^n \cdot
    \sum_{\overline{c} \in \pi_0(X_n \setminus X_{n-1})/G}  \rho^{(2)}(\Phi_{\overline{c}}),
\]
where $\rho^{(2)}(\Phi_{\overline{c}})$ has been defined in~\eqref{rho_upper_(2)(Phi_overline(c)}.
 \end{theorem}
 \begin{proof}
   Choose a cellular $\Phi$-homotopy equivalence
 $f \colon \eub{G} \to \eub{G}$. If $\eub{G}$ is a $\FIN$-finite $G$-space, then we get from
   the definitions that $\Phi$ is $\det$-finite if and only if $(f;\Phi)$ is $\det$-finite, and in this case
   \begin{equation}
     \rho^{(2)}(\Phi)  = \rho^{(2)}(f;\Phi).
     \label{proof_of_the:application_to_automorphism:(1)}
   \end{equation}

   Let $I_m = \pi_0(X_m \setminus X_{m-1})/G$ be the set of equivariant $m$-cells of $X$
   for $m = 0,1,2, \ldots, \dim(X)$.  We show by induction for
   $n = -1,0,1,2 , \ldots, \dim(X)$ that $\eub{G} \times X_n$ is a $\FIN$-finite $G$-space,
   $(f \times a|_{X_n};\Phi)$  is
   $\det$-finite, and we have
   \begin{equation}
     \rho^{(2)}(f \times a|_{X_n};\Phi) =
     \sum_{m =  0}^n \; (-1)^m \cdot
     \sum_{\overline{c} \in I_m}  \rho^{(2)}(\Phi_{\overline{c}}).
     \label{the:application_to_automorphism:claim_for_X_n}
   \end{equation}
    The induction beginning
   $n = -1$ is trivial, since $X_{-1}$ is empty. The induction step from $(n-1)$ to
   $0 \le n \le \dim(X)$ is done as follows.

Choose  a cellular $G$-pushout
\begin{equation}
  \xymatrix@!C=10em{\coprod_{\overline{c} \in I_n} G/H_{\overline{c}} \times S^{n-1}
    \ar[r]^-{q = \coprod_{i \in \overline{c}} q_{\overline{c}}} \ar[d]
    &
    X_{n-1} \ar[d]
    \\
    \coprod_{\overline{c} \in I_n} G/H_{\overline{c}} \times D \ar[r]^-{Q = \coprod_{i \in \overline{c}} Q_{\overline{c}}}
    &
    X_n.
  }
  \label{proof_of_the:application_to_automorphism:(2)}
\end{equation}
Then we get a $G$-pushout by taking the cross product with $\eub{G}$  and the diagonal $G$-actions
\begin{equation}
\xymatrix{\coprod_{\overline{c} \in I_n} \eub{G}  \times G/H_{\overline{c}}  \times S^{n-1} \ar[r] \ar[d]
    &
    \eub{G}  \times X_{n-1} \ar[d]
    \\
    \coprod_{\overline{c} \in I_n} \eub{G}  \times G/H_{\overline{c}} \times  D \ar[r]
    &
    \eub{G}  \times X_n.
  }
  \label{proof_of_the:application_to_automorphism:(3)}
\end{equation}
Let $x_{\overline{c}} \in X_n$ be the image of $(eH_{\overline{c}}, 0)$ for $0 \in D^n$ the origin under
the characteristic map $Q_{\overline{c}}$ appearing the
$G$-pushout~\eqref{proof_of_the:application_to_automorphism:(2)}.  Choose
$g_{\overline{c}} \in G$ such that $a(x_{\overline{c}}) = g_{\overline{c}} x_{\overline{c}}$ holds.  Then we get a $\Phi$-map
$u_{\overline{c}} \colon G/H_{\overline{c}} \to G/H_{\overline{c}}$ sending $gH_{\overline{c}}$ to
$\Phi(g)g_{\overline{c}}H_{\overline{c}}$ such that the following diagram commutes
\begin{equation}
  \xymatrix{\coprod_{\overline{c} \in I_n} G/H_{\overline{c}} \times S^{n-1} \ar[r]^-{q}
    \ar[d] \ar[ddrrr]^(0.75){\quad \;\coprod_{\overline{c} \in I_n} u_{\overline{c}} \times \id_{S^{n-1}}}
    &
    X_{n-1} \ar[d] \ar[ddrrr]^{a|_{X_{n-1}}}
    &
    &
    &
    \\
    \coprod_{\overline{c} \in I_n} G/H_{\overline{c}} \times D^n \ar[r]^-{Q}
    \ar[ddrrr]_(0.5){\coprod_{\overline{c} \in I_n} u_{\overline{c}} \times \id_{D^n}\quad}
    &
    X_n \ar[ddrrr]_(0.22){a}|(0.42)\hole|(0.44)\hole|(0.46)\hole|(0.48)\hole|(0.50)\hole|(0.52)\hole|(0.54)\hole|(0.56)\hole|(0.58)\hole|(0.60)\hole
    &
    &
    &
    \\
    & & &
    \coprod_{\overline{c} \in I_n} G/H_{\overline{c}} \times S^{n-1} \ar[r]^(0.64){q} \ar[d]
    &
    X_{n-1} \ar[d]
    \\
    &
    &
    &
    \coprod_{\overline{c} \in I_n} G/H_{\overline{c}} \times D^n \ar[r]^(0.64){Q}
    &
    X_n
  }
  \label{proof_of_the:application_to_automorphism:(4)}
\end{equation}
We obtain a  $G$-homeomorphisms
\[
  s \colon G \times_{H_{\overline{c}}} (\res_{G}^{H_{\overline{c}}}\eub{G})\xrightarrow{\cong} \eub{G} \times G/H_{\overline{c}}
\]
sending $(g,x)$ to $(gx,gH_{\overline{c}})$.  The following diagram of $G$-spaces commutes
\begin{equation}
\xymatrix{G \times_{H_{\overline{c}}} \eub{G} \ar[r]^s \ar[d]_v
  &
  \eub{G} \times G/H_{\overline{c}} \ar[d]^{f \times u_{\overline{c}}}
  \\
  G \times_{H_{\overline{c}}} \eub{G} \ar[r]^s
  &
  \eub{G} \times G/H_{\overline{c}}
}
\label{proof_of_the:application_to_automorphism:(5)}
\end{equation}
where the vertical maps are $\Phi$-maps and $v$ sends $(g,x)$ to
$(\Phi(g)g_{\overline{c}},g_{\overline{c}}^{-1}f(x))$.  The map
$w \colon \eub{G} \to \eub{G}$ sending $x$ to $g_{\overline{c}}^{-1}f(x)$ satisfies
\[
  w(hx) = g_{\overline{c}}^{-1}f(hx) = g_{\overline{c}}^{-1}\Phi(h)f(x)
  = g_{\overline{c}}^{-1}\Phi(h)g_{\overline{c}}g_{\overline{c}}^{-1}f(x) 
  = g_{\overline{c}}^{-1}\Phi(h)g_{\overline{c}}w(x)
\]
and hence is $\Phi_{\overline{c}}$-equivariant, where
$\Phi_{\overline{c}} \colon H_{\overline{c}} \to H_{\overline{c}}$ sends $h$ to
$g_{\overline{c}}^{-1}\Phi(h)g_{\overline{c}}$.  Note that
$\res_G^{H_{\overline{c}}} \eub{G}$ is a model for $\eub{H_{\overline{c}}}$ and
$\eub{H_{\overline{c}}}$ is by assumption a $\FIN$-finite $G$-space. Since
$\Phi_{\overline{c}}$ is $\det$-finite by assumption, $(w;\Phi_{\overline{c}})$ is
$\det$-finite.  Now Theorem~\ref{the:properties_of_L2-torsion_of_self_homo}~%
\ref{the:properties_of_L2-torsion_of_self_homo:induction} implies that
$\eub{G} \times G/H_{\overline{c}}$ is $\FIN$-$G$-space,
$(f \times u_{\overline{c}};\Phi)$ is $\det$-finite, and we get
 \begin{equation}
   \rho^{(2)}(\Phi_{\overline{c}}) = \rho^{(2)}(f \times u_{\overline{c}};\Phi).
 \label{proof_of_the:application_to_automorphism:(6)}
\end{equation}

By the induction hypothesis  the $G$-space $\eub{G} \times X_{n-1}$ is $\FIN$-finite,
$f \times a|_{X_{n-1}};\Phi$ is $\det$-finite, and  we have
   \begin{equation}
     \rho^{(2)}(f \times a|_{X_{n-1}};\Phi) =
     \sum_{m =  0}^{n-1} \; (-1)^m \cdot
     \sum_{\overline{c} \in I_m}  \rho^{(2)}(\Phi_{\overline{e}}).
    \label{proof_of_the:application_to_automorphism:(7)}
  \end{equation}
  We conclude from
  Theorem~\ref{the:properties_of_L2-torsion_of_self_homo}~\ref{the:properties_of_L2-torsion_of_self_homo:Sum_formula}
  and~\ref{the:properties_of_L2-torsion_of_self_homo:Product_formula} and the
  $G$-pushout~\eqref{proof_of_the:application_to_automorphism:(3)} that the $G$-space
  $\eub{G} \times X_n$ is $\FIN$-finite, $(f \times a|_{X_n}; \Phi)$ is $\det$-finite, and
  we get
  \begin{equation}
     \rho^{(2)}(f \times a|_{X_n};\Phi) = \rho^{(2)}(f \times a|_{X_{n-1}};\Phi) + (-1)^n \cdot 
    \sum_{\overline{c} \in I_m}  \rho^{(2)}(\Phi_{\overline{e}}).
    \label{proof_of_the:application_to_automorphism:(8)}
  \end{equation}
  Now the induction step from $(n-1)$ to $n$
  follows from~\eqref{proof_of_the:application_to_automorphism:(7)}
  and~\eqref{proof_of_the:application_to_automorphism:(8)}. 

  If we apply~\eqref{the:application_to_automorphism:claim_for_X_n}  in the case $n = \dim(X)$,
  we conclude  that $(f \times a; \Phi)$ is $\det$-finite and we have
  \begin{equation}
     \rho^{(2)}(f \times a;\Phi) =
     \sum_{m \ge   0} \; (-1)^m \cdot
     \sum_{\overline{c} \in I_m}  \rho^{(2)}(\Phi_{\overline{c}}).
     \label{the:application_to_automorphism:claim_for_X}
   \end{equation}
   Since the projection $\eub{G} \times X \to \eub{G}$ is a $G$-homotopy equivalence
   and  $(f \times a; \Phi)$ is $\det$-finite,
   Theorem~\ref{the:properties_of_L2-torsion_of_self_homo}~\ref{the:properties_of_L2-torsion_of_self_homo:trace_formula}
   implies that  $(f;\Phi)$ is $\det$-finite and 
   we have  $\rho^{(2)}(f \times a; \Phi) = \rho^{(2)}(f; \Phi) = \rho^{(2)}(\Phi)$.
   Hence Theorem~\ref{the:application_to_automorphism} follows
   from~\eqref{the:application_to_automorphism:claim_for_X}
     \end{proof}

     Although one may only be interested in $\rho^{(2)}(\Phi)$, it is useful that we have
     the $L^2$-torsion $\rho^{(2)}(f;\Phi)$ of
     Definition~\ref{def:L2-torsion_of_a_selfhomotopy_equivalence} at hand, since in the
     proof of Theorem~\ref{the:application_to_automorphism} one needs to consider this
     more general notion $\rho^{(2)}(f;\Phi)$, see~\eqref{the:application_to_automorphism:claim_for_X_n}.

     \begin{example}[Group extensions]\label{exa:group_extension_general_second}
       Suppose that we can write $G$ as an extension $1 \to K \to G \xrightarrow{p} Q \to 1$      
       and there is a finite model for $\eub{Q}$.   Then we can consider
       the $G$-$CW$-complex $X = p^*\eub{Q}$ obtained from the finite $Q$-$CW$-complex     
  $\eub{Q}$ by restriction with $p$.  Obviously $X^H = \eub{Q}^{p(H)}$ is contractible for
  any finite subgroup $H \subseteq G$.  There is a bijective correspondence between the
  open equivariant cells of $X$ and the open equivariant cells of $\eub{Q}$ given by
  $c \mapsto p(c)$. We have $\dim(c) = \dim(p(c))$ and $G_c = p^{-1}(Q_{p(c)})$.  Hence
  $G_c$ contains $K$ as a subgroup of finite index $[G_c : K] = |Q_{p(c)}|$.  Suppose that
  $G_c$ is $\FIN$-finite for every open cell $c$. (Note that it is not true  that
  a group is a $\FIN$-finite group if it contains a $\FIN$-finite subgroup of finite index,
  see~\cite{Leary-Nucinkis(2003)}.) Consider a group automorphism
  $\Phi \colon G \xrightarrow{\cong} G$ with $\Phi(K) = K$ and $p \circ \Phi = p$
  such that $\Phi|_{G_c}$ is $\det$-finite.
  We conclude from
  Theorem~\ref{the:elementary_properties_of_L2-torsion_of_autos}~%
\ref{the:elementary_properties_of_L2-torsion_of_autos:restriction}
  that $|Q_{p(c)}| \cdot \rho^{(2)}(G_c) = \rho^{(2)}(K)$. Recall the orbifold Euler
  characteristic of $\eub{Q}$
  \[
    \chi_{\operatorname{orb}}(\eub{Q})  =\sum_{c} (-1)^{\dim(c)} \cdot \frac{1}{|Q_{p(c)}|}.
  \]
   which agrees with the $L^2$-Euler characteristic $\chi^{(2)}(\eub{Q};\caln(Q))$.
      
      Theorem~\ref{the:application_to_automorphism}
      implies  that  $G$ is $\FIN$-finite, $\Phi$ is $\det$-finite,  and we have      
      \begin{equation}
        \rho^{(2)}(\Phi) = \chi_{\operatorname{orb}}(\eub{Q}) \cdot \rho^{(2)}(\Phi|_K)
        = \chi^{(2)}(\eub{Q};\caln(Q)) \cdot \rho^{(2)}(\Phi|_K).
        \label{exa:group_extension_general_second:formula}
      \end{equation}
      In particular we get $\rho^{(2)}(\Phi) = 0$ if $\chi^{(2)}(\eub{Q};\caln(Q))$ vanishes.

    Now  assume that $K$ is an infinite amenable $\FIN$-finite group
    and that each group $G_{c}$ is $\FIN$-finite. Consider a group automorphism
    $\Phi \colon G \xrightarrow{\cong} G$ with $\Phi(K) = K$ and $p \circ \Phi = p$.
    Then $K$ and $G_{c}$ for every equivariant cell $c$ are infinite amenable $\FIN$-finite groups
    and satisfy the Determinant Conjecture. Hence
    $\Phi|_K$ and $\Phi|_{G_c}$ for every equivariant cell $c$ are $\det$-finite. 
    Moreover,  $\eub{K}$ is $\det$-$L^2$-acyclic
  and we get $\rho^{(2)}(\Phi|_K) = 0$
    by Theorem~\ref{the:elementary_properties_of_L2-torsion_of_autos}~%
\ref{the:elementary_properties_of_L2-torsion_of_autos:L2-acyclic_group}
    and~\cite[Theorem~6.54~(8) on page~266 and Theorem~6.75 on page~274]{Lueck(2002)}.
    Hence $\Phi$ is $\det$-finite and satisfies $\rho^{(2)}(\Phi) = 0$
    by~\eqref{exa:group_extension_general_second:formula}.

    Note that Theorem~\ref{the:elementary_properties_of_L2-torsion_of_autos}~%
\ref{the:elementary_properties_of_L2-torsion_of_autos:Product_formula} is a special case of this example.
  \end{example}

     \begin{example}[Graph of groups]\label{exa:graph_of_groups_group_autos}
       Let $Y$ be a connected non-empty graph in the sense of~\cite[Definition~1 in
       Section~2.1 on page~13]{Serre(1980)}.  Let $(G,Y)$ be a graph of groups in the
       sense of~\cite[Definition~8 in Section~4.4. on page~37]{Serre(1980)}.  (In the
       sequel we use the notation of~\cite{Serre(1980)}).  An automorphism
       $\varphi \colon (G,T) \xrightarrow{\cong} (G,T)$ consists of a collection of
       automorphisms $\varphi_P \colon G_P \xrightarrow{\cong} G_P$ for every
       $P \in \operatorname{vert}(T)$ and an automorphism
       $\varphi_y \colon G_y \xrightarrow{\cong} G_y$ for every
       $y \in \operatorname{edge}(T)$ which are compatible with the monomorphisms
       $G_y \to G_{t(y)}$ and satisfy $\varphi_{\widetilde{y}} = \varphi_y$.

     Let $P_0$ be an element in $\operatorname{vert}(Y)$ and let $T$ be a maximal tree of
     $Y$. Denote by $\overline{\operatorname{edge}(Y)}$ the quotient of
     $\operatorname{edge}(Y)$ under the involution $y \mapsto \widetilde{y}$. Note that
     $Y$ considered as a $CW$-complex has the set $\operatorname{vert}(T)$ as set of
     $0$-cells and $\overline{\operatorname{edge}(Y)}$ as set of $1$-cells. Since by
     definition $G_{y} = G_{\widetilde{y}}$ holds, we can define for
     $\overline{y} \in \overline{\operatorname{edge}(Y)}$ the group $G_{\overline{y}}$ to
     be $G_y$ for a representative $y \in\operatorname{edge}(Y)$ of $\overline{y}$ and
     analogously define the automorphism
     $\varphi_{\overline{y}} \colon G_{\overline{y}} \xrightarrow{\cong}
     G_{\overline{y}}$.
             
     Let $\pi_1(G,Y,P_0)$ and $\pi_1(G,Y,T)$ be the fundamental groups in the sense
     of~\cite[page~42]{Serre(1980)}. Note that $\pi_1(G,Y,P_0)$ and $\pi_1(G,Y,T)$ are
     isomorphic, see~\cite[Proposition~20 in Section~5.1. on page~44]{Serre(1980)}. Then
     there exists
     \begin{itemize}
     \item A graph $\widetilde{X} = \widetilde{X}(G,Y,T)$;
     \item An action of $\pi = \pi_1(G,Y,T)$ on $\widetilde{X}$;
     \item A morphism $p \colon \widetilde{X} \to X$ inducing an isomorphism $\pi\backslash \widetilde{X} \to Y$;
     \end{itemize}
     such that the following is true
     \begin{itemize}
     \item $\widetilde{X}$ is a tree;
       \item $\widetilde{X}^H$ is contractible for every finite subgroup $H \subseteq \pi$;
      \item $\widetilde{X}$ is a $1$-dimensional $\pi$-$CW$-complex for which there exists $\pi$-pushout
     \[
       \xymatrix{\coprod_{P \in \operatorname{vert}(T)}  \pi/\pi_P \times S^0 \ar[r] \ar[d]
         &
         \coprod_{\overline{y} \in \overline{\operatorname{edge}(Y)}} \pi/\pi_y \ar[d]
        \\
        \coprod_{P \in \operatorname{vert}(T)}  \pi/\pi_P \times D^1 \ar[r]
            &
            \widetilde{X}
          }
        \]
        such that $\pi_P \cong G_P$ for $P \in \operatorname{vert}(T)$ and
        $\pi_{\overline{y}} = G_{\overline{y}}$ for $\overline{y} \in \overline{\operatorname{edge}(Y)}$ holds.
        \item $\varphi$ induces an automorphism $\Phi \colon \pi \xrightarrow{\cong} \pi$
        and a cellular $\Phi$-homeomorphism $f \colon \widetilde{X} \to \widetilde{X}$.
      \end{itemize}
      
     All these claims follows from~\cite[Section~5.3 and Theorem~15 in Section~6.1 on page~58]{Serre(1980)}.

     Now suppose that $Y$ is finite, each of the groups $G_P$ for $P \in \operatorname{vert}(T)$ and
     $G_{\overline{y}}$ for $\overline{y} \in \overline{\operatorname{edge}(Y)}$ are $\FIN$-finite,
     and each of the automorphisms $\varphi_P$ for $P \in \operatorname{vert}(T)$ and
     $\varphi_{\overline{y}}$ for $\overline{y} \in \overline{\operatorname{edge}(Y)}$ is $\det$-finite

     Then we conclude from Theorem~\ref{the:application_to_automorphism} that $\pi$ is $\FIN$-finite,
     $\Phi$ is $\det$-finite, and we get
     \[\rho^{(2)}(\Phi)= \sum_{P \in \operatorname{vert}(T)} \rho^{(2)}(\varphi_P)
       -\sum_{\overline{y} \in \operatorname{vert}(T)} \rho^{(2)}(\varphi_{\overline{y}}).
   \]
 \end{example}

 \begin{example}[Amalgamated Products]\label{exa:Amalgamated_products_group_autos}
   Let $G_0$ be a common subgroup of the group $G_1$ and the group $G_2$. Denote by
   $G = G_1 \ast_{G_0} G_2$ the amalgamated product. Let
   $\Phi_i \colon G_i \xrightarrow{\cong} G_i$ be an automorphism for $i = 0,1,2$ such
   that $\Phi_j|_{G_0} = \Phi_0$ holds for $j = 1,2$. Denote by
   $\Phi \colon G \xrightarrow{\cong} G$ be the automorphism of $G$ induced by the
   automorphisms $\Phi_i$ for $i = 0,1,2$. Suppose that $G_i$ is $\FIN$-finite and $\Phi_i$ is $\det$-finite
   for $i = 0,1,2$.

   Then $G$ is $\FIN$-finite, $\Phi$ is $\det$-finite, and we 
   get
     \[
     \rho^{(2)}(\Phi) = \rho^{(2)}(\Phi_1)  + \rho^{(2)}(\Phi_2) - \rho^{(2)}(\Phi_0).
     \]
     This follows from Example~\ref{exa:graph_of_groups_group_autos} applied to the graph of groups
     associated to $G_1 \ast_{G_0} G_2$, see~\cite[page~43]{Serre(1980)}.
   \end{example}

   %%%%%%%%%%%%%%%%%%%%%%%%%%%%%%%%%%%%%%%%%%%%%%%%%%%%%%%%%%%%%%%%%%%%%
%%%%%%%%%%%%%%%%%%%%%%%%%%%%%% Section 6 %%%%%%%%%%%%%%%%%%%%%%%%%%%%%%%%
%%%%%%%%%%%%%%%%%%%%%%%%%%%%%%%%%%%%%%%%%%%%%%%%%%%%%%%%%%%%%%%%%%%%%

\section{A short discussion of the finiteness assumptions}%
\label{sec:A_short_discussion_of_the_finiteness_assumptions}

We have introduced the notion of $\FIN$-finite group in
Definition~\ref{def:FIN-finite_group}, of $\FIN$- finite $G$-space in
Definition~\ref{def:FIN-finite_G-space}, of a $\det$-finite $G$-space in
Definition~\ref{def:det_finite_G-space}, of a $\det$-finite $\Phi$-homotopy equivalence
$(f;\Phi)$ of a $\FIN$-$G$-space in
Definition~\ref{def:det-finite_Phi-self-homotopy_equivalence}, and of a $\det$-finite
automorphism of a $\FIN$-finite group in
Definition~\ref{def:Det-finite_group_automorphism}. We have  introduced the notions of the
$L^2$-torsion $\rho^{(2)}(f;\Phi)$ and $\rho^{(2)}(\Phi)$ in
Definitions~\ref{def:L2-torsion_of_a_selfhomotopy_equivalence}
and~\ref{def:L2_torsion_of_a_group_automorphism}. The conditions $\det$-finite were needed
to make sense of $\rho^{(2)}(f;\Phi)$ and $\rho^{(2)}(\Phi)$.

  \begin{definition}[Det-finite group]\label{def:weakly_admissible_group}
  We call a group $G$ \emph{$\det$-finite} if it satisfies the following conditions:
\begin{enumerate}
\item There exists a  finite $G$-$CW$-model for the classifying space  of proper actions $\eub{G}$;
\item For one (and hence all) finite $G$-$CW$-models  $X$ for $\eub{G}$,
  the $G$-$CW$-complex $X$ is of determinant class.
    \end{enumerate}
  \end{definition}

  \begin{lemma}\label{lem:criteria_for_det-finite}
    Let $G$ be a group and $X$ be a $G$-space.

    \begin{enumerate}
    \item\label{lem:criteria_for_det-finite:FIN_finite_for_G}
      The group $G$ is $\FIN$-finite if and only if the $G$-space $\eub{G}$ is $\FIN$-finite;
    \item\label{lem:criteria_for_det-finite:det-finite_for_G}
      The group $G$ is $\det$-finite if and only if the $G$-space $\eub{G}$ is $\det$-finite;
    \item\label{lem:criteria_for_det-finite:det-finite_for_autos_G}
      Suppose that the group $G$ is det-finite.
      Then every automorphism $\Phi \colon G \to G$ is $\det$-finite;      
     \item\label{lem:criteria_for_det-finite:Det_Con}
       Suppose  that $G$ is a $\FIN$-finite-group and satisfies the Determinant Conjecture.
       Consider any automorphism $\Phi \colon G \to G$. Then $\Phi$  is $\det$-finite, the group
       $G \rtimes_{\Phi} \IZ$ is $\det$-$L^2$-acyclic,   and we get
       \[
       \rho^{(2)}(\Phi) = \rho^{(2)}(G \rtimes_{\Phi} \IZ);
       \]
       
     \item\label{lem:criteria_for_det-finite:sofic}
       If the group $G$ is sofic, then it satisfies the Determinant Conjecture.
       In particular any automorphism $\Phi$ of a sofic $\FIN$-finite group $G$ is
       $\det$-finite, $G \rtimes_{\Phi} \IZ$ is $\det$-$L^2$-acyclic,   and we get
       \[\rho^{(2)}(\Phi) = \rho^{(2)}(G \rtimes_{\Phi} \IZ).
       \]
     \end{enumerate}
   \end{lemma}
   \begin{proof}~\ref{lem:criteria_for_det-finite:FIN_finite_for_G} and~\ref{lem:criteria_for_det-finite:det-finite_for_G}
     follow directly from the definitions.
     \\[1mm]~\ref{lem:criteria_for_det-finite:det-finite_for_autos_G}
     This follows from Lemma~\ref{lem_T_upper_PhI_u_det_L2-acyclic}~%
\ref{lem_T_upper_PhI_u_det_L2-acyclic:det_L2-acyclic}
\\[1mm]~\ref{lem:criteria_for_det-finite:Det_Con}
Since $G$ satisfies the Determinant Conjecture , $G \rtimes_{\Phi} \IZ$ satisfies the Determinant Conjecture,
see~\cite[Proposition~13.39 on page~469]{Lueck(2002)}. (Note  that proof
of~\cite[Proposition~13.39 on page~469]{Lueck(2002)} is not correct in the generality as stated,
but still works in the special case that $H \subseteq G$ is a normal subgroup with amenable quotient $G/H$.)
Now the claim follows from Theorem~\ref{the:properties_of_L2-torsion_of_self_homo}~%
\ref{the:properties_of_L2-torsion_of_self_homo:homotopy_invariance}
since $T_{\widehat{f};\Phi}$ is a model for $\eub{(G \rtimes_{\Phi} \IZ)}$.
\\[1mm]~\ref{lem:criteria_for_det-finite:sofic} See~\cite[Theorem~5]{Elek-Szabo(2005)}.
   \end{proof}

   \begin{remark}\label{rem:conditions_on_G}
     Note that in the setting of
     Section~\ref{sec:Blowing_up_orbits_by_classifying_spaces_of_families} we had always
     to assume that the group $G$ for which we want to make a computation in terms of
     subgroups had to satisfy the condition (DFJ). The advantage of the setup of
     Section~\ref{sec:L2-torsion_of_an_automorphism_of_a_det-finite_group} is that we only
     have to make assumptions about the restrictions of an automorphism to certain subgroups and
     then the necessary assumptions are
     automatically satisfied for the automorphism $\Phi$ itself.
     For instance, it is not known whether a graph of groups has a sofic fundamental group if
     all edge and vertex groups are sofic and whether any hyperbolic group is sofic.
    
     This advantage is essentially due to the fact that in the
     definition of $\rho^{(2)}(f)$ we use the finite proper $G \rtimes_{\Phi} \IZ$-$CW$-complex
     $T_{\widehat{f};\Phi}$. Note that $T_{\widehat{f};\Phi}$ is a model
      for $\eub{(G \rtimes_{\Phi} \IZ)}$.  If we additionally assume that $G$
      is sofic,  we could define $\rho^{(2)}(\Phi)$ just by
      $\rho^{(2)}(G \rtimes_{\Phi} \IZ)$, see
      Lemma~\ref{lem:criteria_for_det-finite}~\ref{lem:criteria_for_det-finite:sofic}.
      \end{remark}

%%%%%%%%%%%%%%%%%%%%%%%%%%%%%%%%%%%%%%%%%%%%%%%%%%%%%%%%%%%%%%%%%%%%%
%%%%%%%%%%%%%%%%%%%%%%%%%%%%%% Section 7 %%%%%%%%%%%%%%%%%%%%%%%%%%%%%%%%
%%%%%%%%%%%%%%%%%%%%%%%%%%%%%%%%%%%%%%%%%%%%%%%%%%%%%%%%%%%%%%%%%%%%%

\section{Computations}\label{sub:Computations}
In this section we use
Theorem~\ref{the:properties_of_the_blow_up}~\ref{the:properties_of_the_blow_up:L_upper_(2)_torsion}
and the theory developed in the later sections to compute the $L^2$-torsion for a large
range of groups and automorphisms.  We expect that there are many more applications of the
formula so have selected applications which require both additional techniques and which
may appeal to a range of audiences.

An essential point in some of our arguments in the subsequent sections involves verifying
that the groups involved satisfy the Determinant Conjecture.  In some of the cases this
amounts to proving that the group is sofic and then appealing
to~\cite[Theorem~5]{Elek-Szabo(2005)}. However, for automorphisms of groups
$\Phi \colon G \to G$ we try to avoid the use of Determinant Conjecture for the group $G$
and only assume it for the restrictions to isotropy groups of cells.  In this we do not
know whether $G$ itself or $G\rtimes_\Phi \IZ$ satisfies the Determinant Conjecture.

 %-----------------------------------------------------------------------------

\subsection{Lattices}
Let $G$ be a second countable locally compact group with Haar measure $\mu$.  We say a
discrete subgroup $\Gamma\leqslant G$ is a \emph{lattice} if $\mu(\Gamma\backslash G)$ is
finite.  We say a lattice $\Gamma$ is \emph{uniform} if $\Gamma\backslash G$ is compact.

\begin{lemma}
  Let $G$ and $H$ be second countable locally compact groups admitting lattices and
  suppose that there exists a lattice in $G$ which is $L^2$-acyclic.  If $\Gamma$ is a
  lattice in $G\times H$, then $\Gamma$ is $L^2$-acyclic.
\end{lemma}
\begin{proof}
  Let $\Lambda$ be an $L^2$-acyclic lattice in $G$ and let $\Lambda'$ be any lattice in
  $H$.  The group $L:=\Lambda\times \Lambda'$ is a lattice in $G\times H$ and hence is
  measure equivalent to $\Gamma$.  From the K\"unneth formula for $L^2$-Betti numbers,
  see~\cite[Theorem~6.54~(5) on page~266]{Lueck(2002)}, we conclude that $L$ is
  $L^2$-acyclic.  Now, by Gaboriau's invariance of $L^2$-Betti numbers under measure
  equivalence~\cite{Gaboriau(2002a)}, we see that $\Gamma$ is $L^2$-acyclic.
\end{proof}

For a locally finite $\CAT(0)$ polyhedral complex $X$ we denote by $\Isom(X)$ the set of
isometries of $X$ such that if $g$ fixes a cell $\sigma\in X$ setwise, then $g$ fixes
$\sigma$ pointwise.  We say $\Isom(X)$ acts \emph{minimally} if there is no non-empty
$\Isom(X)$-invariant proper subspace $Y\subset X$.
Following~\cite[Definition~5.6]{Lueck(2002)}, for a symmetric space $M=G/K$ with
$G=\Isom_0(M)$ semi-simple, we define the \emph{fundamental rank} of $M$ to be
\[
  \fr(M)= \rk_\IC (G)-\rk_\IC(K).
\]

\begin{proposition}\label{vanishing_CAT0}
  Let $n\geq 0$ and let $M=M_1\times\cdots\times M_k\times \IE^n$ be a symmetric space
  with each $M_i$ irreducible of non-compact type.  Let $X$ a locally finite $\CAT(0)$
  polyhedral complex with $\Isom(X)$ acting minimally and cocompactly.  Let $\Gamma$ be
  satisfy (DFJ) or be virtually torsion-free.  If $\Gamma$ is a uniform lattice in
  $\Isom(M)\times\Isom^+(X)$ and either
  \begin{enumerate}
  \item $\fr(M_i)\geq2$ for some $i$,
  \item or $n\geq 1$,
  \end{enumerate}
  then $\rho^{(2)}(\Gamma)=0$.
\end{proposition}
The hypothesis that $\Gamma$ satisfies (DFJ) could be dropped if we either knew the
Determinant Conjecture for $\CAT(0)$ groups, or if we knew the Farrell--Jones Conjecture
for the Weyl groups of finite subgroups of $\CAT(0)$ groups.  Note there are non-virtually
torsion-free $\CAT(0)$ groups, see e.g.,~\cite{Hughes(2022)}.
\begin{proof}
  If $\Gamma$ is virtually torsion-free, then it satisfies (DFJ).  Indeed, $\Gamma$
  satisfies the Farrell--Jones Conjecture
  by~\cite{Bartels-Lueck(2012annals),Wegner(2012),Kasprowski-Rueping(2017long-thin)}.
  By~\cite[Theorem A]{Hughes(2021graphs)}, any stabiliser $\Gamma_\sigma$ in $\Gamma$ of a
  cell $\sigma$ in $X$ fits into short exact sequence
  $1\to F_\sigma \to \Gamma_\sigma \to \Lambda_\sigma \to 1$, where $F_\sigma$ is finite
  and $\Lambda_\sigma$ is a uniform lattice in $\Isom(M)$.  We have
  $\rho^{(2)}(\Gamma_\sigma)=\frac{1}{|F_\sigma|}\rho^{(2)}(\Lambda_\sigma)$ by
  Theorem~\ref{the:Basic_properties_of_L_upper_2-torsion}\ref{the:Basic_properties_of_L_upper_2-torsion:finite_quotient}.
  By~\cite{Olbrich(2002)}, either condition in the theorem ensures that
  $\rho^{(2)}(\Lambda_\sigma)=0$.  The result follows from
  Theorem~\ref{the:properties_of_the_blow_up}~\ref{the:properties_of_the_blow_up:L_upper_(2)_torsion}.
\end{proof}

\begin{example}[$S$-arithmetic subgroups of $\GL_n(\IC)$]
  Let $k$ be a number field, let $S$ be a finite set of places containing the archimedian
  ones, and let $\mathbf G$ be a simply connected simple $k$-group. Let
  $\Gamma< \mathbf G(k)$ be an $S$-arithmetic subgroup.  By the general theory of
  $S$-arithmetic lattices  $\Gamma$ acts on a product of
  symmetric spaces $M=M_1\times \dots M_k$ with each $M_i$ non-compact and irreducible and
  a product of Euclidean buildings $X=X_1\times \dots X_\ell$ such that the diagonal
  action on $X\times M$ is cocompact.  If $\fr(M_i)\geq 2$ for some $i$, then
  $\rho^{(2)}(\Gamma)=0$.  Note that the lattices here often have a strict fundamental
  domain on the building and so one could instead
  apply~\cite[Theorem~6.1]{Okun-Schreve(2024torsion)}.
\end{example}

We also provide an example where the theorem applies to a non-residually finite $\CAT(0)$ lattice.

\begin{example}[Leary--Minasyan groups]\label{ex:LearyMinasyan}
  Let $L'$ and $L''$ denote a pair of finite index subgroup of $L=\IZ^n$, and let
  $A\in\GL_n(\IQ)$ such that multiplication by $A$ defines an isomorphism $L'\to L''$.  We
  form the HNN extension
  \[\LM(A,L')=\langle x_1,\dots,x_n,t\ |\ [x_i,x_j],\ txt^{-1} =A(x)\ \forall x\in L'
    \rangle,\] where $\langle x_1,\dots, x_n\rangle =L$.  Let $\calt$ denote the
  Bass--Serre tree of the HNN extension and note that every vertex has valence
  $|L:L'|+|L:L''|$.  By~\cite[Theorem~7.5]{Leary-Minasyan(2021)}, if $A$ is conjugate in
  $\GL_n(\IR)$ to an orthogonal matrix, then $\LM(A,L')$ is a lattice in
  $\Isom(\IE^n)\times \Aut(\calt)$.  By~\cite[Theorem~1.1(1)]{Leary-Minasyan(2021)}, the
  group $\LM(A,L')$ is residually finite if and only if $A$ is conjugate in $\GL_n(\IQ)$
  to a matrix in $\GL_n(\IZ)$.  Now, by
  Theorem~\ref{the:properties_of_the_blow_up}~\ref{the:properties_of_the_blow_up:L_upper_(2)_torsion} we have
  $\rho^{(2)}(\LM(A,L'))=0$.
\end{example}

We mention that a zoo of non-residually finite examples of $\CAT(0)$ lattices can be constructed using the tools
from~\cite{Hughes(2021graphs)}.

 %-----------------------------------------------------------------------------

\subsection{Higher dimensional graph manifolds}
Higher dimensional graph manifolds were introduced in~\cite{Frigerio-Lafont-Sisto(2015)}
and studied from the point of view of the Borel Conjecture and quasi-isometric rigidity.  We now recall the construction: 

Let $n\geq 3$ and let $\Gamma$ be a finite graph.  For each vertex $v$ of $\Gamma$, let
$M_v$ be a compact manifold of dimension $n_v$ whose interior is a complete hyperbolic
manifold of finite volume, and hence has toral cusps. The boundary of $M_v$ is a
collection of tori, where $2\leq n_v\leq n$.  Write $N_v=T^{n-n_v}\times N_v$, where
$T^\ell$ denotes the $\ell$-torus.  We call each $N_v$ a \emph{piece}. Its boundary is a
collection of $(n-1)$-dimensional tori.  An \emph{extended graph manifold} $M$ is any
manifold obtained from pieces as above by gluing their torus boundaries by affine
diffeomorphisms such that each edge of $\Gamma$ corresponds to some gluing.  For a piece
$N_v$ of $M$, if $n-n_v=0$, then we call $N_v$ a \emph{hyperbolic piece}.  We denote the
set of hyperbolic pieces in $M$ by $\calh$.

Recall that the $L^2$-torsion of a closed hyperbolic manifold of odd dimension $(2d+1)$
has been computed by Hess--Schick~\cite{Hess-Schick(1998)} and of a compact manifold of
odd dimension $(2d+1)$ whose interior is a complete hyperbolic manifold of finite volume
by L\"uck--Schick~\cite[Theorem~0.5]{Lueck-Schick(1999)}, namely, it is proportional by a
dimension constant $C_{2n+1} \not= 0$ to its volume.

We may decomposes a graph manifold as a graph of spaces: each
vertex space $M_v$  is a piece $N_i$ and each edge space $M_e$ is the collar of the
torus boundary that is obtained by gluing two geometric pieces via an affine
diffeomorphism.  Note that on the level of fundamental groups this decomposes $\pi_1 M$ as
a graph of groups with vertex groups $\pi_1 M_v=\pi_1 N_i$ and edge groups $\IZ^{n-1}$.

\begin{theorem}\label{thm:graphMan}
  Let $M$ be a $(2n+1)$-dimensional extended graph manifold with hyperbolic pieces
  $\calh$.  Then, $\pi_1(M)$ is sofic, satisfies the Determinant Conjecture, and 
  \[
  \rho^{(2)}(\widetilde M)=\sum_{M_v\in\calh}\rho^{(2)}(\widetilde M_v).
  \]
\end{theorem}
\begin{proof}
  We first establish that $\pi_1 (M)$ is sofic satisfies the Determinant Conjecture.  We
  have that $\pi_1 M$ admits a decomposition as a graph of groups such that: Every edge
  group $G_e$ is $\IZ^{2n}$, and every vertex group $G_v$ splits as a direct product
  $A_v\times Q_v$ where $A_v$ is a (possibly trivial) abelian group and $Q_v$ is the
  fundamental group of a compact manifold of dimension $2n+1 - \dim A_v$ whose interior is
  a complete hyperbolic manifold of finite volume.  In the case that $A_v$ is trivial, we
  have $G_v\in \calh$.  It follows that every edge group is amenable and every vertex
  group is residually finite.  Thus, $\pi_1 M$ is sofic
  by~\cite[Theorem~1.2]{Ciobanu-Holt-Rees(2014)} and so satisfies the Determinant
  Conjecture by~\cite[Theorem~5]{Elek-Szabo(2005)}.  Note that this implies the stabiliser
  of any vertex or edge satisfies Determinant Conjecture.
  
  Note that $\widetilde M\simeq E\pi_1 (M)$ is a cocompact $\pi_1(M)$-space. In the graph
  of groups decomposition for $\pi_1(M)$ we have that every vertex group either contains
  an infinite finitely generated abelian normal subgroup and is $L^2$-acyclic
  by~\cite[Theorem~7.4(1),(2)]{Lueck(2002)}, or is the fundamental group of a compact
  manifold of odd dimension whose interior is a complete hyperbolic manifold of finite
  volume and is $L^2$-acyclic by~\cite[Corollary~6.5]{Lueck-Schick(1999)}.

  Now, we have that $\rho^{(2)}(\widetilde M)=\rho^{(2)}(\pi_1 M)$.  For an edge $e$, the
  group $G_e$ satisfies $\rho^{(2)}(G_e)=0$ by~\cite{Wegner(2009)}.  Each vertex group
  $G_v$ is either isomorphic to $\pi_1 (M_v)$ for some $M_v\in\calh$, or has an infinite
  abelian normal subgroup.  In the later case, the $L^2$-torsion of $G_v$ vanishes
  by~\cite{Wegner(2009)}.  The result follows from
  Theorem~\ref{the:properties_of_the_blow_up}~\ref{the:properties_of_the_blow_up:L_upper_(2)_torsion}.
\end{proof}

\begin{corollary}
  Let $M$ be an $n$-dimensional graph manifold.  If the graph decomposition of $M$
  contains a hyperbolic piece, then $M$ does not admit any non-trivial
  action of the circle $S^1$.
\end{corollary}
\begin{proof}
  If $M$ is even dimensional, then one easily sees that $M$ has a non-vanishing
  $L^2$-Betti number by the $L^2$-Mayer--Vietoris sequence and~\cite{Borel(1985)}.  The
  non-vanishing $L^2$-Betti of an aspherical number obstructs non-trivial circle actions
  by~\cite[Corollary~1.43]{Lueck(2002)}.  If $M$ is odd-dimensional, then by
  Theorem~\ref{thm:graphMan}, $M$ has non-vanishing $L^2$-torsion.  Non-zero $L^2$-torsion
  of (the universal cover of) an aspherical space obstructs non-trivial circle actions
  by~\cite[Theorem~3.105]{Lueck(2002)}.
\end{proof}

%-----------------------------------------------------------------------------

\subsection{Relatively hyperbolic groups}

\begin{definition}
  We say a group $G$ is \emph{one-ended relative to a collection of subgroups $\calp$} if
  there does not exist a splitting of $G$ over finite subgroups such that each group in
  $\calp$ is conjugate into some vertex group. Note that a one-ended group is one-ended
  relative to every collection of subgroups. We write $\Aut(G; \calp)$ to denote the group
  of automorphisms of $G$ which preserve the conjugacy classes of every subgroup $P\in \calp$.
\end{definition}

\begin{lemma}\label{lem:rel_hyp_fin_finite}
    Let $G$ be hyperbolic relative to a finite collection $\calp$ of $\FIN$-finite groups.  Then, $G$ is $\FIN$-finite.
\end{lemma}
\begin{proof}
  Let $\calf$ denote the family of subgroups of $G$ generated by $\FIN$ and $\calp$.  The
  main theorem of~\cite{Martinez-Pedroza-Przytzycki(2019)} states that there exists a
  finite model for the space $E_\calf G$.  Applying
  Theorem~\ref{the:properties_of_the_blow_up} to each element of $\calp$ we obtain a
  finite model for $E_\FIN G$.  That is $G$ is $\FIN$-finite.
\end{proof}

Suppose that $G$ is one-ended and hyperbolic relative to $\calp$. By the work of Guirardel and
Levitt~\cite[Corollary~9.20]{Guirardel-Levitt(2017JSJ)} (see
also~\cite[Section~3.3]{Guirardel-Levitt(2015)}), there is a canonical \emph{JSJ tree}
$\calt_G$ for $G$.  We denote the quotient classes of vertices in $G\backslash \calt_G$ by
$\JSJ(G)$.  More precisely, $\calt_G$ is a simplicial $G$-tree 
such that the $G$-equivariant homeomorphism class of $\calt_G$
is preserved by the elements of $\Aut(G; \mathcal{P})$.

Note that if each $P\in \calp$ satisfies the Farrell--Jones Conjecture, then $G$ above
satisfies the Farrell--Jones Conjecture as well~\cite{Bartels(2017)}.

\begin{prop}
  Let $G$ be an $L^2$-acyclic group which is hyperbolic and one-ended relative to a finite
  collection $\calp$ of $\FIN$-finite groups.  If $G$ satisfies (DFJ), then
  \[\rho^{(2)}(G)=\sum_{v\in\JSJ(G)} \rho^{(2)}(G_v),\]
  where $\JSJ(G)$ is the set of vertex groups in some JSJ decomposition for $G$.
\end{prop}
\begin{proof}
  This follows immediately from
  Theorem~\ref{the:properties_of_the_blow_up}~\ref{the:properties_of_the_blow_up:L_upper_(2)_torsion}
  applied to the JSJ tree, noting the groups involved are $\FIN$-finite by hypothesis or
  by Lemma~\ref{lem:rel_hyp_fin_finite}.  Note that the passage from an action on the tree
  to a graph of groups is explained in Example~\ref{exa:graph_of_groups_group_autos}.
\end{proof}

 %-----------------------------------------------------------------------------

\subsection{Automorphisms of one-ended relatively hyperbolic groups}\label{sec:Ltwo_rel_hyp}
Recall that a \emph{toral relatively hyperbolic group} is a torsion-free group hyperbolic
relative to a finite collection of finitely generated abelian groups.  These are a natural
generalisation of hyperbolic manifolds with toral cusps.  In the next section we show that
$L^2$-torsion of an automorphism of a class containing all toral relatively hyperbolic
groups reduces to the sum of the $L^2$-torsion of the restrictions of the automorphisms to
certain surface subgroups in the JSJ decomposition.  Said differently, the $L^2$-torsion
of an automorphism of a toral relatively hyperbolic group is carried by its restriction to
surface subgroups.

Let $G$ be $\{$hyperbolic and one-ended$\}$ relative to $\calp$ and let $\calt_G$ denote
the canonical JSJ tree of $G$.  The group $\Aut(G; \calp)$ has a finite index subgroup
$\calk(\calt_G)$ whose action on $\calt_G$ descends to the identity on the quotient graph $G \backslash \calt_G$.
Note that if the groups in $\calp$ are not relatively hyperbolic groups, e.g., if they are
virtually polycyclic groups that are not virtually cyclic, then $\Aut(G;\calp)$ has finite
index in $\Aut(G)$.

\thmRelHypAutos
\begin{proof}
  We are in the setting of Example~\ref{exa:graph_of_groups_group_autos} with the tree
  given by $\calt_G$.  We need to argue that each of the vertex and edge stabilisers are
  $\FIN$-finite and each of the automorphisms $\Phi_v,\Phi_e$ for $v\in V(T)$ and
  $e\in E(T)$ are det-finite.

  The stabiliser $G_v$ in $G$ of a vertex $v$ in $\calt_G$ satisfies one of the following:
  \begin{enumerate}
  \item $v$ is \emph{flexible}: the group $G_v$ fits into a short exact sequence
    $1\to F_v\to G_v\to Q_v\to 1$, where $F_v$ is finite and $Q_v$ is isomorphic to the
    fundamental group of a compact hyperbolic orbifold $S$ and the image of the natural
    homomorphism $\calk(T_G) \to \Out(G_v)$ is contained in the mapping class group
    $\Mod(S)$ of $S$;
  \item there exists $P_i \in \calp$ with $G_v=P=P_i^g$ for some $g\in G$. In particular, the group $G_v$ is
    virtually polycyclic;
  \item the image of the natural homomorphism $\calk(\calt_G) \to \Out(G_v)$ is finite.
  \end{enumerate}
 
  In the first case $G_v$ is commensurable with the fundamental group of a compact
  hyperbolic Riemann surface.  One way to see this is to note that $Q_v$ is a Fuchsian
  group which, by~\cite{Grunewald-JaikinZapirain-Zalesskii(2008)}, are `good' in the sense
  of Serre~\cite[Page~16]{Serre(1997)}.  So $G_v$ is residually finite
  by~\cite[Page~16]{Serre(1997)} and hence contains a torsion-free subgroup of finite
  index---this subgroup is the desired surface subgroup.  We claim that $G_v$ is
  $\FIN$-finite (a model for $E_\FIN G_v$ is given by the hyperbolic plane with action
  factoring through $Q_v$).  Since $G_v$ is residually finite we see that it is sofic and
  hence satisfies the Determinant Conjecture~\cite{Elek-Szabo(2005)}.  Thus, in this case
  we have $\rho^{(2)}(\Phi_v)=\rho^{(2)}(G_v\rtimes|_{\Phi_v}\IZ)$.
 
  In the second case (or in the case of an edge group) we have
 \[
 \rho^{(2)}(\Phi_v)=\rho^{(2)}(G_v\rtimes_{\Phi_v}\IZ)=0
\] 
by~\cite{Wegner(2009)}, because $G_v\rtimes \IZ$ is $\{$virtually polycyclic$\}$-by-$\IZ$
which is again virtually polycyclic.  Note that virtually polycyclic groups are
$\FIN$-finite.  Hence, $\Phi_v$ or $\Phi_e$ in this case is det-finite.
  
In the third case we have that $G_v$ is a group hyperbolic relative to finitely many
virtually polycyclic groups and so is $\FIN$-finite by Lemma~\ref{lem:rel_hyp_fin_finite}.
Moreover, we have that $\Phi_v$ is periodic.  In particular, $\Phi_v$ is det-finite and
$\rho^{(2)}(\Gamma_v)=0$ by
Theorem~\ref{the:elementary_properties_of_L2-torsion_of_autos}~\ref{the:elementary_properties_of_L2-torsion_of_autos:periodic}.
  
The formula claimed in the proposition now follows from
Theorem~\ref{the:properties_of_the_blow_up}~\ref{the:properties_of_the_blow_up:L_upper_(2)_torsion}.
\end{proof}

\begin{remark}\label{rem:rel_hyp_admissible}
  Let $G$ and $\Phi$ be as in the Proposition~\ref{thmRelHypAutos}.  If
  $\Gamma=G\rtimes_\Phi \IZ$ satisfies (DFJ), then $\rho^{(2)}(\Gamma)=\rho^{(2)}(\Phi)$.
  In particular, this holds whenever $G$ is toral relatively hyperbolic.  Indeed, in this
  case $\Gamma$ is torsion-free and satisfies the Farrell--Jones
  Conjecture~\cite{Andrew-Guerch-Hughes(2023)}.  However, in the case where $\Gamma$ has
  non-trivial torsion we do not know that the Farrell--Jones Conjecture holds for the Weyl
  groups of finite subgroups of $\Gamma$.
\end{remark}

 %-----------------------------------------------------------------------------
\subsection{Polynomially growing automorphisms}\label{sec:poly}
Let $G$ be a group generated by a finite set $S$.  Let $|\cdot|$ denote the word metric on
$G$ with respect to $S$.  We say that $\Phi\in\Aut(G)$ has \emph{polynomial growth} (or
\emph{grows polynomially}) of \emph{degree} at most $d$ if for each $g\in G$ there is a
constant $C$ such that $|\Phi^n(g)|<Cn^d+C$ for all $n\in\IN$.

For a conjugacy class $[g]$ in $G$ let $||[c]||$ denote the length of a shortest
representative.  We say $\Phi\in\Out(G)$ is \emph{polynomially growing} if for each
conjugacy class $[g]$ of $G$ there is a constant $C$ such that $||\Phi^n([g])||<C n^d +C$
for all $n\in\IN$.

We remark that for any two finite generating sets $S_1$ and $S_2$ the corresponding word
metrics on $G$ are bi-Lipschitz equivalent. It follows that the two definitions of growth
are independent of the choice of a finite generating set.

Throughout the next few subsections we will prove vanishing of $L^2$-torsion for various
polynomially growing automorphisms.  We collect these results in
Theorem~\ref{thm:poly_anthology_L2tors} and remark that many of the arguments here are
inspired by arguments in~\cite{Andrew-Guerch-Hughes-Kudlinska(2023)} proving vanishing of
torsion homology growth using the cheap rebuilding
property~\cite{Abert-Bergeron-Fraczyk-Gaboriau(2025)}.

%-----------------------------------------------------------------------------

\subsubsection{Free products}
For any splitting of a group $G$ as a free product $\ast_{i=1}^kG_i$ with each $G_i$
non-trivial and not necessarily freely irreducible, we call the collection of conjugacy
classes of the $G_i$ a \emph{free factor system}.

Let $G_1,\dots,G_k$ be non-trivial finitely generated groups and let $F_N$ denote a free
group of rank $N$.  Let $G=\ast_{i=1}^k G_i \ast F_N$ and denote by $\calf$ the set of
conjugacy classes of the subgroups $G_i$ in $G$.

We say that the pair $(G,\calf)$ is a \emph{sporadic free product} if one of the following holds:
\begin{enumerate}
    \item $k=0$ and $G=\IZ$;
    \item $k=1$ and $G=G_1$ or $G=G_1\ast \IZ$; or
    \item $k=2$ and $G=G_1\ast G_2$.
\end{enumerate}

The key point for us is that automorphisms of sporadic free products have a canonical
$G$-tree that they preserve.  We recall this result here noting that it built on work of
Guirardel and Horbez~\cite{Guirardel-Horbez(2022)}

\begin{prop}~\cite[Proposition
  2.1]{Andrew-Guerch-Hughes-Kudlinska(2023)}\label{Prop:existenceBStree} Let $(G,\calf)$
  be a free product and let $\Phi \in \Out(G,\calf)$ be polynomially growing. There is a
  free factor system $\calf'$ of $(G, \calf)$ and $k \in \IN$ such that $\Phi^k$ preserves
  $\calf'$ and $(G, \calf')$ is sporadic. In particular, $\Phi^k$ preserves a Bass--Serre
  tree associated to $\calf'$.
\end{prop}

We are now ready to prove a vanishing result for the $L^2$-torsion of polynomially growing
automorphisms of free products.

\begin{prop}\label{prop:l2tors_poly_free_prod}
  Let $G=G_1 \ast \ldots \ast G_k \ast F_N$ be a free product of groups satisfying
  condition (DFJ), let $\Phi$ be a polynomially-growing automorphism of $G$ which
  preserves the conjugacy classes of the factors $G_i$, and let $\Phi_i\colon G_i\to G_i$
  be the appropriate restriction of $\Phi$ up to conjugacy. Suppose that for every
  $i\in \{1,\ldots,k\}$, we have $\rho^{(2)}(\Phi_i)=0$, then $\rho^{(2)}(\Phi)=0$.
\end{prop}
\begin{proof}
  This proof follows a similar structure
  to~\cite[Theorem~3.1]{Andrew-Guerch-Hughes-Kudlinska(2023)}.  We proceed by induction on
  the Grushko rank $k+N$.  If $k=1$ and $N=0$, then $\rho^{(2)}(\Gamma)=0$, since
  $\Gamma\cong\IZ$ by hypothesis.  If $k=0$ and $N=1$, then $\Gamma$ has finite index
  subgroup isomorphic to $\IZ^2$, in particular, $\rho^{(2)}(\Gamma)=0$.  Suppose now
  $k+N\geq 2$.  Let $\Phi$ be the image of $\Phi$ in $\Out(G)$ and let $\calf$ be the
  sporadic free factor system given by Proposition~\ref{Prop:existenceBStree}.  Let
  $\calt$ be the Bass--Serre tree of $G$ associated to the free factor system $\calf$ and
  note that there is a positive power of $\Phi$ which preserves $\calt$.  The vertex
  stabilisers of $G$ acting on $\calt$ are proper free factors of $G$ and hence have
  smaller Grushko rank than $G$ and the edge stabilisers are trivial.  It follows that a
  finite index subgroup $\Gamma'$ of $\Gamma$ acts on $\calt$ with vertex stabilisers of
  the form $\Gamma'_v=G_v\rtimes_{\Phi^n|_{G_v}} \IZ$ and with infinite cyclic edge
  stabilisers.  Now, $\Phi^n$ is det-finite and $\rho^{(2)}(\Phi^n_v)=0$ by the inductive
  hypothesis and $\rho^{(2)}(\IZ)=0$.  Thus, the result follows from
  Example~\ref{exa:graph_of_groups_group_autos} and
  Theorem~\ref{the:elementary_properties_of_L2-torsion_of_autos}~%
\ref{the:elementary_properties_of_L2-torsion_of_autos:multiplicativity}.
\end{proof}

We remark that in the special case of free-by-cyclic groups this recovers a result of Clay~\cite{Clay(2017free)}.

\begin{corollary} If $\Phi\in\Aut(F_N)$ is polynomially growing, then $\rho^{(2)}(F_N\rtimes_\Phi\IZ)=0$.
\end{corollary}

%-----------------------------------------------------------------------------

\subsubsection{Relatively hyperbolic groups}
We first deal with the one-ended case.

\begin{theorem}\label{thm:poly_L2tors_rel_hyp_one_ended}
  Let $G$ be a group hyperbolic and one-ended relative to a finite collection $\calp$ of
  virtually polycyclic groups and let $\Phi \in \Aut(G)$ be polynomially growing.  Then
  $\rho^{(2)}(\Phi)=0$.
\end{theorem}
\begin{proof}
  Let
  $\Gamma=G\rtimes_\Phi\IZ$.
  Let $H$ be a characteristic finite index torsion-free subgroup of $G$, let $n>0$ be such
  that $\Phi^n\in\calk(\calt_G)$, and let $\Lambda=G\rtimes_{\Phi^n} \IZ$.  Note that
  $\Lambda$ has finite index in $\Gamma$.   We denote $\Phi^n$ by $\Psi$.

  By Proposition~\ref{thmRelHypAutos}, the only possible non-zero contribution to
  $\rho^{(2)}(\Psi)$ is from the terms $\rho^{(2)}(G_v\rtimes_{\Psi_v}\IZ)$ where $v$ runs over
  flexible vertices.  In this case, $\Lambda_v=G_v\rtimes_{\Psi|_{v}} \IZ$, where
  $G_v$ is the fundamental group of a compact hyperbolic surface $S_v$ and $\Psi|_{v}$
  is an element of $\Mod(S_v)$ with polynomial growth.  Hence, $\Psi^n|_{v}$ is periodic
  or a Dehn twist.  Thus, by~\cite[Theorem~0.7]{Lueck-Schick(1999)} we have $\rho^{(2)}(\Lambda_v)=0$.
  The result now follows from Theorem~\ref{the:elementary_properties_of_L2-torsion_of_autos}~%
\ref{the:elementary_properties_of_L2-torsion_of_autos:multiplicativity}
\end{proof}

\begin{proposition}\label{prop:poly_L2tors_rel_hyp_general}
  Let $G$ be a group hyperbolic relative to a finite collection $\calp$ of
  virtually polycyclic groups and let $\Phi \in \Aut(G)$ be polynomially growing.  Then
  $\rho^{(2)}(\Phi)=0$.
\end{proposition}
\begin{proof}
  Let $\Gamma=G\rtimes_\Phi\IZ$.  If $G$ has finitely many ends we are done by
  Proposition~\ref{thm:poly_L2tors_rel_hyp_one_ended}. If not, then since $G$ is finitely
  presented, $G$ admits a finite index characteristic subgroup $H$ which splits as a free
  product of a (possibly trivial) free group $F_N$ and finitely many groups $G_i$, each
  hyperbolic relative to a finite collection of virtually polycyclic groups and each with
  finitely many ends.  We may pass to a large power $\ell$ of $\Phi$ such that $\Phi^\ell$
  preserves the conjugacy classes of the $G_i$.  Now, by
  Proposition~\ref{thm:poly_L2tors_rel_hyp_one_ended} we have
  $\rho^{(2)}(G_i\rtimes_{\Phi^\ell}\IZ)=0$ for each $i$.  We are have now verified the
  hypothesis of Proposition~\ref{prop:l2tors_poly_free_prod} applied to the group
  $H\rtimes_{\Phi^\ell}\IZ$.  The result follows from
  Theorem~\ref{the:Basic_properties_of_L_upper_2-torsion}~%
\ref{the:Basic_properties_of_L_upper_2-torsion:passage_to_subgroups_of_finite_index}.
\end{proof}

%-----------------------------------------------------------------------------

\subsubsection{Right-angled Artin and Coxeter groups}

Let $L$ be a flag complex.  The \emph{right-angled Artin group} (RAAG) $A_L$ is defined to
be the group with presentation
\[\langle L^{(0)}\ |\ [v,w] \text{ if } \{v,w\}\in L^{(1)}\rangle.\]
The \emph{right-angled Coxeter group} (RACG) is the group
\[W_L=A_L/\langle\langle v^2\ |\ v\in L^{(0)}\rangle\rangle.\] We define a number of
automorphisms of RAAGs and RACGs:
\begin{enumerate}
\item \emph{graph automorphisms}, that is automorphism induced from $L$;
\item \emph{inversions}, which send $v\mapsto v^{-1}$ and $u\mapsto u$ for $u\neq v$ and
  $u,v\in L^{(0)}$;
\item \emph{partial conjugations} $k_{W,C}$ for $w\in L^{(0)}$ and a connected component
  $C$ of $L\setminus \mathrm{st}(w)$, which are defined by $k_{w,C}(u)=w^{-1}uw$ if
  $u\in C^{(0)}$ and $k_{w,C}(u)=u$ if $u\in L^{(0)}\setminus C$;
\item \emph{folds} $t_{v,w}$ for any $v,w\in L^{(0)}$ with
  $\mathrm{lk}(v)\subseteq \mathrm{lk}(w)$, which are defined by $t_{v,w}(v)=vw$ and
  $t_{v,w}(u)=u$ for all $u\in L^{(0)}\setminus\{v\}$.
\end{enumerate}
We say an automorphism of $A_L$ (resp. $W_L$) is \emph{untwisted} if it is contained in
the subgroup $U(A_L)\leqslant \Aut(A_L)$ (resp. $U(W_L)\leqslant \Aut(W_L)$) which is
generated by the graph automorphisms, inversions, partial conjugations, and folds.

By~\cite[Proposition~A(3)]{Fioravanti(2024)}, the untwisted automorphisms of $A_L$ are
exactly the automorphism which preserve the standard coarse median structure on $A_L$.
By~\cite{Sale-Susse(2019)}, the subgroup of untwisted automorphisms of $\Aut(W_L)$ has
finite index.

\begin{prop}\label{prop:poly_L2tors_RAAG}
  Let $L$ be a flag complex on $[m]$ and let $\Gamma=A_L\rtimes_\Phi \IZ$.  If $\Phi$ is
  an untwisted and polynomially growing automorphism of $A_L$, then
  $\rho^{(2)}(\Gamma)=0$.
\end{prop}
We note that the following argument is structurally very similar to the proof
of~\cite[Theorem~5.1]{Andrew-Guerch-Hughes-Kudlinska(2023)}
\begin{proof}
  We proceed by induction on $m$.  The base case, when $m=1$, implies $\Gamma$ is
  virtually isomorphic to $\IZ^2$.  In this case we have $\rho^{(2)}(\Gamma)=0$ as
  required.  We now suppose $m>1$.  Note that if $K$ is a full subcomplex of $L$ then any
  untwisted automorphism of $A_L$ preserving $A_K\leqslant A_L$ restricts to an untwisted
  automorphism of $A_K$.  There are three cases to consider:

  The first case is when $A_L$ is freely reducible.  In this case
  $A_L=A_{K_1}\ast\dots A_{K_k}\ast F_n$ and each $K_i$ and $[n]$ is a full subcomplex of
  $L$.  Note that each $K_i$ and if $n\neq 0$ the complex $[n]$ all contain at least one
  vertex and strictly less than $m$ vertices.  Now, pass to a sufficiently high power
  $\ell$ of $\Phi$ which preserves the conjugacy classes of the $A_{K_i}$.  Then, by the
  inductive hypothesis $\rho^{(2)}(A_{K_i}\rtimes_{\Phi^\ell} \IZ)=0$.  The case then
  follows from Proposition~\ref{prop:l2tors_poly_free_prod} and
  Theorem~\ref{the:Basic_properties_of_L_upper_2-torsion}~%
\ref{the:Basic_properties_of_L_upper_2-torsion:passage_to_subgroups_of_finite_index}.

  The second case is when $A_L$ is both freely and directly irreducible.  In this case,
  by~\cite[Proposition~D]{Fioravanti(2024)} we have that
  $A_L=A_{K_1}\ast_{A_{K_3}}A_{K_2}$ with $K_3=K_1\cap K_2$ and each $K_1,K_2,K_3$ a
  non-empty proper full subcomplex of $L$.  Passing to a large enough power $\ell$ of
  $\Phi$, the Bass-Serre tree $\calt$ of the splitting is $\Phi^\ell$ invariant and
  $\Phi^\ell$ preserves the stabilisers, that is $\Phi^\ell(A_{K_i})=A_{K_i}$.  Thus,
  $\Gamma$ admits a finite index subgroup $\Lambda$ acts on $\calt$ with stabilisers of
  the form $A_{K_i}\rtimes_{\Phi^\ell} \IZ$.  By induction these stabilisers have
  vanishing $L^2$-torsion.  Thus, the case follows from
  Theorem~\ref{the:properties_of_the_blow_up}~\ref{the:properties_of_the_blow_up:L_upper_(2)_torsion}
  and
  Theorem~\ref{the:Basic_properties_of_L_upper_2-torsion}~%
\ref{the:Basic_properties_of_L_upper_2-torsion:passage_to_subgroups_of_finite_index}.

  The final case is when $A_L$ is directly reducible.  In this case
  $A_L\cong \prod_{i=1}^n A_{K_i}\times \IZ^k$ where each $K_i$ is a full subcomplex of
  $L$ and $A_{K_i}$ is directly irreducible and non-cyclic.  If $k\geq 0$, then $A_L$ is
  $L^2$-acyclic and so $\rho^{(2)}(\Gamma)=0$ by
  Theorem~\ref{the:elementary_properties_of_L2-torsion_of_autos}~%
\ref{the:elementary_properties_of_L2-torsion_of_autos:L2-acyclic_group}.
  Thus, we may suppose $k=0$.  Now, we have that $A_L$ acts on a product of trees
  $X=\prod_{i=1}^n \calt_{i}$, where $\calt_i$ is a tree either provided by
  Proposition~\ref{Prop:existenceBStree} or by~\cite[Proposition~D]{Fioravanti(2024)}
  depending on if $A_{K_i}$ is freely reducible or not.  We now pass to a sufficiently
  large power $\ell$ of $\Phi$ such that $\Phi^\ell|_{A_{K_i}}$ preserves each tree
  $\calt_i$.  In particular, if $\sigma$ is a cell of $X$ we have
  $\Phi^\ell(\Stab_{A_L}(\sigma))=\Stab_{A_L}(\sigma)$.  Moreover, each stabiliser of a
  cell in $X$ under the action of $A_L$ is a product of the stabilisers of some of the
  groups $A_{K_i}$ acting on $\calt_i$ and the remaining $A_{K_j}$.  In particular, each
  stabiliser of $A_L\rtimes_{\Phi^\ell}\IZ$ acting on $X$ is of the form
  $A_{J_\sigma}\rtimes_{\Phi^\ell|_{A_{J_\sigma}}}\IZ$ where $J_\sigma$ is a full proper
  non-empty subcomplex of $L$.  Hence, the $L^2$-torsion of the stabilisers vanishes by
  induction and the proposition follows from
  Theorem~\ref{the:properties_of_the_blow_up}~%
\ref{the:properties_of_the_blow_up:L_upper_(2)_torsion}
  and
  Theorem~\ref{the:Basic_properties_of_L_upper_2-torsion}~%
\ref{the:Basic_properties_of_L_upper_2-torsion:passage_to_subgroups_of_finite_index}.
\end{proof}

The following proposition is proved verbatim taking into account the remarks after Theorem
E and at the start of Section 5~\cite{Fioravanti(2024)} and after noting the subgroup of
untwisted automorphisms of $\Aut(W_L)$ has finite index~\cite{Sale-Susse(2019)}.

\begin{prop}\label{prop:poly_L2tors_RACG}
  Let $L$ be a flag complex on $[m]$ and let $\Gamma=W_L\rtimes_\Phi \IZ$.
  If $\Phi$ is a polynomially growing automorphism of $A_L$, then $\rho^{(2)}(\Gamma)=0$.
\end{prop}

%-----------------------------------------------------------------------------

\subsubsection{Anthology}
We collect the above results on polynomially growing automorphisms, namely 
Propositions~\ref{prop:poly_L2tors_rel_hyp_general},~\ref{prop:poly_L2tors_RAAG} and~\ref{prop:poly_L2tors_RACG}, into one theorem.  This
answers~\cite[Question~1.2]{Andrew-Guerch-Hughes-Kudlinska(2023)}.

\thmAnthology

 %-----------------------------------------------------------------------------

\subsection{Handlebody groups}

Let $V_g$ denote a genus $g$ handlebody and let $\Mod(V_g)$ denote its mapping class
group, \emph{the genus $g$ handlebody group}.  The reader is referred
to~\cite{Andrew-Hensel-Hughes-Wade(2025)} for more information.  Our final result
answers~\cite[Problem~28]{Andrew-Hensel-Hughes-Wade(2025)}.

\thmHandlebody
\begin{proof}
  Let $X$ denote the disc complex for $\Mod(V_g)$.  We denote by $G$ the intersection of
  the handlebody group with the pure mapping class group in $\Mod(S_g)$.  This is a finite
  index torsion-free subgroup of $\Mod(V_g)$ which acts on the disc complex $X$
  cocompactly and cellularly such that the set-wise stabiliser of any cell is equal to its
  point-wise stabiliser.  In particular, $X$ is a finite $G$-CW-complex.  We claim $G$ is
  admissible, indeed, $EG$ is $L^2$-acyclic
  by~\cite[Theorem~6.1]{Andrew-Hensel-Hughes-Wade(2025)}.  Since $G$ is a subgroup of the
  mapping class group it is residually finite~\cite{Grossman(1974)}.  Hence, sofic and so
  satisfies the Determinant Conjecture~\cite[Theorem~5]{Elek-Szabo(2005)}.
    
  By~\cite[\S1.A]{Andrew-Hensel-Hughes-Wade(2025)}, for every cell $\sigma\in X$, the
  stabiliser $G_\sigma$ fits into an exact sequence
  \[1\to \IZ^{n_\sigma} \to G_\sigma \to H_\sigma \to 1\] where $H_\sigma$ has finite
  cohomological dimension and is a torsion-free finite index subgroup of a group of type
  $\mathsf{VF}$.  In particular, $\rho^{(2)}(G_\sigma)=0$ by~\cite{Wegner(2009)}.
  Applying this to Theorem~\ref{the:properties_of_the_blow_up}~%
\ref{the:properties_of_the_blow_up:L_upper_(2)_torsion} we obtain the vanishing of the
  $L^2$-torsion of $\Mod(V_g)$.  The  vanishing of the torsion homology growth
  is~\cite[Theorem~6.1]{Andrew-Hensel-Hughes-Wade(2025)}.
\end{proof}

%%%%%%%%%%%%%%%%%%%%%%%%%%%%%%%%%%%%%%%%%%%%%%%%%%%%%%%%%%%%%%%%%%%% 
%%%%%%%%%%%%%%%%%%%%%%%%%%%% Reference  %%%%%%%%%%%%%%%%%%%%%%%%%%%%%%%%
%%%%%%%%%%%%%%%%%%%%%%%%%%%%%%%%%%%%%%%%%%%%%%%%%%%%%%%%%%%%%%%%%%%%

\typeout{----------------------------- References ------------------------------}

\addcontentsline{toc<<}{section}{References} 
% \bibliographystyle{abbrv}
% \bibliography{dbpub,dbpre}

%\version{17.03.2026 (Wolfgang)}

\end{document}